\documentclass[12pt]{amsart}

\usepackage[left=2.2cm,tmargin=2.5cm,bmargin=2.5cm,right=2.2cm]{geometry}

\usepackage{amsmath,amssymb,color}
\usepackage{hyperref}
\usepackage[capitalise]{cleveref}
\usepackage{physics}
\usepackage{tikz-cd}
\usepackage{verbatim}
\usepackage[shortlabels]{enumitem}

\usepackage{pdfsync}

\newcommand{\Hom}{\operatorname{Hom}}
\newcommand{\coker}{\operatorname{coker}}
\newcommand{\Sym}{\operatorname{Sym}}

\newcommand{\Surj}{\operatorname{Surj}}
\newcommand{\Out}{\operatorname{Out}}
\newcommand{\Aut}{\operatorname{Aut}}

\newcommand{\Img}{\operatorname{Im}}
\newcommand{\ra}{\rightarrow}
\newcommand{\Z}{\mathbb Z}
\newcommand{\Q}{\mathbb Q}
\newcommand{\E}{\mathbb E}
\newcommand{\F}{\mathbb F}
\newcommand{\Path}{\operatorname{Path}}
\newcommand{\im}{\operatorname{im}}
\newcommand{\Prob}{\operatorname{Prob}}

\newcommand{\QZ}{\mathbb Q/\mathbb Z}
\newcommand{\bK}{\mathbf K}
\newcommand{\bG}{\mathbf G}
 
\newcommand{\bH}{\mathbf H}
\newcommand{\tensor}{\otimes}
\newcommand{\Sp}{\operatorname{Sp}}
\newcommand{\SL}{\operatorname{SL}}

\newcommand{\ASp}{\operatorname{ASp}}
\newcommand{\Prof}{\operatorname{Prof}}
\newcommand{\C}{\mathcal C}
\newcommand{\pl}{t}
\def\Mdcorated/{{oriented}}

\numberwithin{equation}{section}
\newtheorem{prop}[equation]{Proposition}
\newtheorem{theorem}[equation]{Theorem}
\newtheorem{cor}[equation]{Corollary}

\newtheorem{lemma}[equation]{Lemma}

\newtheorem{definition}[equation]{Definition}

\theoremstyle{remark}
\newtheorem{remark}[equation]{Remark}

\title{Finite quotients of $3$-manifold groups}

\author{Will Sawin}
\address{Department of Mathematics\\
Princeton University \\
Fine Hall, Washington Road \\
Princeton, NJ 08540 USA}  
\email{wsawin@math.princeton.edu}

\author{Melanie Matchett Wood}
\address{Department of Mathematics\\
Harvard University\\
Science Center Room 325\\
1 Oxford Street\\
Cambridge, MA 02138 USA}  
\email{mmwood@math.harvard.edu}

\begin{document}

\begin{abstract}
For $G$ and $H_1,\dots, H_n$ finite groups, does there exist a $3$-manifold
group with $G$ as a quotient but no $H_i$ as a quotient?  
We answer all such questions in terms of the group cohomology of  finite groups.
We prove non-existence with topological results generalizing the theory of semicharacteristics.  
To prove existence of
3-manifolds with certain finite quotients but not others, we use a
probabilistic method, by first proving a formula for the distribution
of the (profinite completion of) the fundamental group of a random
3-manifold in the Dunfield-Thurston model of random Heegaard splittings
as the genus goes to infinity.   We believe this is the first construction of a new distribution of random groups from its moments.
\end{abstract}

\maketitle

\section{Introduction}

In this paper, we  
address the question of what finite quotients \emph{in what combinations} the fundamental group of a $3$-manifold can have and not have.
It is well-known that for any finite group $G$, there exists a closed $3$-manifold $M$ with $G$ as a quotient of $\pi_1(M)$.
However, we can ask more detailed questions about the possible finite quotients of $3$-manifold groups.  
If $G$ and $H_1,\dots,H_m$ are finite groups, does there exist a closed $3$-manifold $M$ with $G$ as a quotient but no $H_i$ as a quotient?
In this paper, we give an answer to all such questions in terms in the cohomology of  finite groups.

First, we prove a topological theorem that provides certain obstructions to the existence of $3$-manifold groups with certain quotients but not others.
Then, we prove that whenever these obstructions vanish, not only do such $3$-manifolds exist, but that a positive proportion of $3$-manifolds (under a natural distribution on Heegaard splittings) have quotients and non-quotients as desired.  We do this by determining completely the  asymptotic distribution of profinite completions $\widehat{\pi_1(M)}$ of random $3$-manifold groups (as the genus of the Heegaard splitting goes to infinity).

For example, if $\pi_1(M)$ admits $(\F_3)^2 \rtimes \SL_2(\mathbb F_3) $ as a quotient, with $\SL_2(\mathbb F_3)$ acting on $(\F_3)^2$ by the standard representation, then it also admits $(\F_3)^4 \rtimes \SL_2(\mathbb F_3) $ as a quotient, with $\SL_2(\mathbb F_3)$ acting on $(\F_3)^4$ by the sum of two copies of the standard representation (\cref{negative-non-projective}). 
For an example of existence, we note that there is a group of order $120$, the generalized quaternion group $Q(8,3,5)$, such that, for each natural number $n$, there exists  a 3-manifold $M$ such that $\pi_1(M)$ admits $ Q(8,3,5)$ as a quotient and all finite groups of order $\leq n$ that are quotients of $\pi_1(M)$ are also quotients of $Q(8,3,5)$ (\cref{fake-3-manifolds}).
This is despite the fact that, by the geometrization theorem, $Q(8,3,5)$ itself is not the fundamental group of any 3-manifold.

We now define some notation to state our main result providing the obstructions discussed above.
For $V$ a symplectic vector space over a finite field $\kappa$ of characteristic $2$, we denote by $\Sp_\kappa(V)$ the group of $\kappa$-linear 
 automorphisms of $V$ preserving the symplectic form. 
Following Gurevich and Hadani \cite{Gurevich2012}, we let $1 \to \mathbb Z/4 \to \mathcal H \to V \to 1$ be the central extension of $V$ by $\mathbb Z/4$ with extension class corresponding to the trace of the symplectic form on $V$ (see Section~\ref{S:Notation}), and let the affine symplectic group $\ASp_{\kappa} (V)$ be the group of automorphisms of $\mathcal H$ acting trivially on $\mathbb Z/4$ and on $V$ by a $\kappa$-linear map, which lies in an exact sequence $1 \to \Hom(V,\F_2) \to \ASp(V)_\kappa \to \Sp_\kappa(V) \to 1$. We will see below there is a class $c_V \in H^3( \ASp_{\kappa}(V), \F_2)$ such that the following theorem holds. 
\begin{theorem}\label{pi1-properties} Let $M$ be a closed, oriented $3$-manifold. Let $V$ be an irreducible representation of $\pi_1(M)$ over a finite  field $\kappa$. Then 

\begin{enumerate}

\item We have $\dim H^1(\pi_1(M) , V) = \dim H^1(\pi_1(M), V^\vee)$.

\item  For each nonzero $\alpha \in H^2( \pi_1(M), V) $ there exists $\beta$ in $H^1( \pi_1(M), V^\vee)$ where $\int_M (\alpha \cup \beta)\neq 0$.

\item  If $V$ is a symplectic representation and $\kappa$ has odd characteristic then $\dim_{\kappa} H^1(\pi_1(M) , V)$ is even.

\item If $V$ is a symplectic representation, $\kappa$ has even characteristic, and
if the map $\pi_1(M) \to \Sp_\kappa(V)$ lifts to $\ASp_\kappa(V)$, then $\dim_{\kappa} H^1(\pi_1(M) , V)$ is congruent mod $2$ to  $\int_M c_V$. 
\end{enumerate}
\end{theorem}

Properties (1) and (2) are immediate consequences of Poincar\'e duality and the vanishing of Euler characteristics for the cohomology of $M$, and the spectral sequence relating the cohomologies of $M$ and $\pi_1(M)$. Property (3) may be less familiar -- it can be proved using a Heegaard splitting of $M$ and some algebraic arguments. Property (4) is even stranger, and its proof uses cobordism.  We will prove below in Theorem~\ref{closure-characterization} a converse of a strengthening of Theorem~\ref{pi1-properties} showing  that the properties in Theorem~\ref{pi1-properties} exactly describe the closure of the set of the profinite completions of $3$-manifold groups in the set of all profinite groups.  This requires proving existence of certain $3$-manifolds, which we do using a probabilistic method.

Dunfield and Thurston
\cite{DunfieldThurston}
 introduced the idea of considering a random $3$-manifold constructed from a random Heegaard splitting.  Briefly, the Heegaard splitting is given by a random element in the mapping class group of genus $g$ by taking a uniform random word of length $L$ in a set of generators (including the identity),  and then letting $L\ra\infty $ and then $g\ra\infty$. 
Dunfield and Thurston   asked \cite[\S6.7]{DunfieldThurston}  if the distribution of the number of surjections from the random $3$-manifold group to a fixed finite group $Q$ has a limit as $g\ra\infty$.  When $Q$ is a simple group, Dunfield and Thurston proved these statistics have a limiting distribution as $g\ra\infty$, but they were not able to show this for general $Q$.
In \cref{C:dis-num-surj},  we are able to answer their question, and 
 prove that these statistics have a limiting distribution for a general $Q$.  
 We do so  as a consequence of a far more general result showing that the random profinite group $\widehat{\pi_1(M)}$ itself has a limiting distribution, which we describe explicitly, as $g\ra \infty$.  

To describe this  distribution, we need to give a topology on the set of relevant profinite groups.
Let $\mathcal{C}$ be a set of finite groups.  We say a group is level-$\mathcal C$ if is contained in the smallest set of finite groups containing $\C$ and closed under fiber products and quotients. Then for a group $H$, we define $H^{\mathcal{C}}$ 
to be inverse limit of all level-$\mathcal C$ quotients of $H$.  We call $H^{\mathcal{C}}$ the \emph{level-$\mathcal{C}$ 
completion of $H$.} 
For example, if $\mathcal{C}=\{\Z/p\Z\}$, then $\pi_1(M)^{\mathcal{C}}=H_1(M,\Z/p\Z)$.

Let $\operatorname{Prof}$ be the set of isomorphism classes of profinite groups which have finitely many open subgroups of index $n$ for each natural number $n$,
with a topology generated by the basic opens $U_{\mathcal C, G}= \{X \in \Prof \mid X^{\mathcal C} \cong G\}$ for finite sets $\mathcal{C}$ and finite groups $G$.
We then describe the limiting distribution of $\widehat{\pi_1(M)}$ by giving the limiting probabilities that $\pi_1(M)^{\mathcal{C}}\cong G$ for each finite $\mathcal{C}$ and $G$.
\begin{theorem}\label{intro-prob} Let $M$ be the random $3$-manifold described above.  As $L\ra\infty$ and then $g\ra\infty$, the distributions of $\widehat{\pi_1(M)}$
weakly converge to a probability distribution $\mu$ on $\operatorname{Prof}$ such that 
\begin{equation}\label{mu-U-formula} \mu(U_{\mathcal C,G} ) 
=  \frac{  \abs{H_2 (G, \mathbb Z) } \abs{G} }{ \abs{H_1(G,\mathbb Z)} \abs{H_3(G, \mathbb Z)} \abs{ \Aut(G) }}  \sum_{ \tau \colon H^3(G, \QZ) \to \QZ} \prod_{i=1}^n w_{V_i}(\tau) \prod_{i=1}^m w_{N_i}(\tau),\end{equation}
 where the $V_i$, $N_i$ are those kernels of surjections from level-$\mathcal{C}$ groups $H$ to $G$ that are also minimal normal subgroups of $H$
 (so, e.g., the $V_i$ are certain irreducible representations of $G$ over certain $\F_p$),
and the $w_{V_i}(\tau) , w_{N_i}(\tau) $ are  constants defined in terms of the action of $G$ on $V_i$, $N_i$ in
Tables 1 and 2 of
 \cref{ss-lc-notation}, with $W_i$ is the set of all the 
level-$\mathcal{C}$ extensions of $G$ by $V_i$. \end{theorem}

For example,  let $S$ be a finite set of primes, and  write $\pi_1^S(M)$ for the pro-$S$ completion of $\pi_1(M)$ (i.e. $\pi_1(M)^\C$ for $\C$ the set of all finite groups whose order is a product of powers of primes in $S$).  \cref{intro-prob} implies
\begin{equation}\label{E:Sgroups}
\lim_{g\ra\infty} \lim_{L\ra\infty} \Prob( \pi_1^S(M)  \textrm{ is trivial})=\prod_{p\in S}  \prod_{j=1}^\infty (1+p^{-j})^{-1} \prod_{N}
 e^{- \frac{ |H_2(N,\Z)|}{|\Out(N)|}},
\end{equation}
where the second product is over non-abelian simple groups $N$ whose order is a product of powers of primes in $S$.
One can check using the classification that there are only finitely many simple $S$-groups, and let $\mathcal{C}$ be the set of these groups.
Then $\pi_1^S(M)=1$ if and only if $\pi_1(M)^{\mathcal{C}}=1$, which allows one to deduce \eqref{E:Sgroups} from 
\cref{intro-prob}.

We can see
from the definitions of the $w_{V_i}$ and $w_{N_i}$ 
 that $\mu(U_{\mathcal C,G} )$ is positive if and only if it contains a profinite group satisfying the four conditions of Theorem~\ref{pi1-properties}.
Thus, from this explicit limiting distribution we can determine the closure of $ \{ \widehat{ \pi_1(M)} \}$ in $\operatorname{Prof}$.
This follows a suggestion of Dunfield and Thurston \cite[\S1.6]{DunfieldThurston} to use their model of random $3$-manifolds for existence results of $3$-manifolds with certain properties.    

\begin{theorem}\label{closure-characterization} A group $G \in \operatorname{Prof}$ lies in the closure of the set \[ \{ \widehat{ \pi_1(M)} \mid M \textrm{ a closed, oriented 3-manifold}\}\] if and only if there exists $\tau\colon H^3( G, \mathbb Q/\mathbb Z) \to \mathbb Q/\mathbb Z$ such that, for each irreducible continuous representation $V$ of $G$ over a finite field $\kappa$, 

\begin{enumerate}

\item We have $\dim H^1(G , V) = \dim H^1(G, V^\vee)$.

\item  For each nonzero $\alpha \in H^2( G, V) $ there exists $\beta$ in $H^1( G, V^\vee)$ where $\tau (\alpha \cup \beta)\neq 0$.

\item  If $V$ is a symplectic representation and $\kappa$ has odd characteristic then $\dim_{\kappa} H^1(G , V)$ is even.

\item If $V$ is a symplectic representation, $\kappa$ has even characteristic, and the map $G \to \Sp_\kappa(V)$ lifts to $\ASp_\kappa(V)$, then $\dim_{\kappa} H^1(G , V)$ is congruent mod $2$ to  $2\tau( c_V)$.

\end{enumerate}

\end{theorem}

In particular, we can apply \cref{closure-characterization} to classify all finite groups in the closure of the set of $3$-manifold groups in $\operatorname{Prof}$.
In \cref{obs-class}, we find such groups are either fundamental groups of spherical $3$-manifolds or $Q(8a,b,c) \times \Z/d$, where 
$Q(8a,b,c)$ are certain generalized quaternion groups.

We note that the topology on Prof is the most natural topology from a number of perspectives. For example, it is the topology generated by the open sets $\{X \mid G \textrm{ is a quotient of } X \}$ for finite groups $G$, along with their (open) complements.   In particular, the set of $X$ with $G$ but none of $H_1,\dots,H_n$
as a quotient is open in this topology,  and thus by describing the closure of the set of $3$-manifold groups  in Theorem~\ref{closure-characterization},  we answer the question of whether there is
a $3$-manifold group with $G$ but none of $H_1,\dots,H_n$
as a quotient.

Another natural question is: if $G$ is a finite group and $E_1,\dots, E_m$ are extensions of $G$, does there exist a $3$-manifold $M$ with a surjection $\pi_1(M)\to G$ that doesn't lift to any $E_i$? (The question from the last paragraph is a special case of this one by letting the  $E_i$ be all  subgroups of the $G \times H_i$ whose projections onto both factors are surjective.)  Theorem~\ref{closure-characterization} answers this question in the same sense as above,  but for some sets of extensions there are particularly nice direct answers as well (see \cref{main-existence}).  For example,  if $V$ is an  absolutely irreducible representation of $G$ over a finite field $\kappa$ of odd characteristic, there is an oriented, closed $3$-manifold $M$ where $\pi(M)$ has a surjection to $G$ that does not lift to $V\rtimes G$ if and only if 
$\dim H^1(G , V) \geq \dim H^1(G, V^\vee)$ and condition (3) of Theorem~\ref{closure-characterization} is satisfied.

Our results also can be applied to answer other questions raised by Dunfield and Thurston.  For example, we show in \cref{virtually-fibered} that, for each natural number $n$, the proportion of $3$-manifolds arising from random Heegaard splittings which have a covering of degree $n$ with positive first Betti number goes to $0$ as the genus of the Heegaard splitting goes to $\infty$,  addressing the question set out at the start of \cite[Section 9]{DunfieldThurston}. 
This shows that, in Agol's Virtual First Betti number Theorem \cite{Agol2013}, it is crucial that the finite cover have arbitrarily large degree.

In Section~\ref{SS:torsionlinking}, we show that our results can explain
 the discrepancies noted by Dunfield and Thurston \cite[\S8]{DunfieldThurston} between the homology of a random $3$-manifold and the abelianization of a random group given by generators and random relations (see also \cite[Chapter 7]{Kowalski2008} for further discussion of this contrast), by consideration of the torsion linking pairing. 
 We show that the homology of a random $3$-manifold, along with its torsion linking pairing, has the distribution of the most natural 
distribution on abelian groups with symmetric pairings (arising, e.g., in \cite{Clancy2015a,Clancy2015,Wood2017,Meszaros2020}). 
Let $G$ be a finite abelian $p$-group and let $\ell \colon G \times G \to \mathbb Q/\mathbb Z$ be a symmetric nondegenerate pairing. For a random $3$-manifold $\pi_1(M)$, we show (in \cref{abelian-p-group-prob}) the probability that the $p$-power torsion part of $H_1(M) = \pi_1(M)^{ab}$ is isomorphic to $G$, by an isomorphism sending the torsion linking pairing to $\ell$, is equal to $\frac{1}{ \abs{G} \abs{\Aut(G,\ell)}}$ times a constant $\prod_{j=1}^{\infty} \frac{1}{ 1+ p^{-j}}$ .

\subsection{Previous work and new approaches}

Dunfield and Thurston's introduction of the model of random Heegaard splittings \cite{DunfieldThurston} is a central motivation for our work.  They proved several results on this model, including those mentioned above, and that
the limiting probability of such a manifold having positive first Betti number was 0 \cite[Corollary 8.5]{DunfieldThurston}.
Moreover, the proof of their Theorem 6.21 on the average of $\#\Surj(\pi_1(M),Q)$ (what we would call the ``moments'' of the random group $\pi_1(M)$) is a key input into our Theorem~\ref{intro-prob}.  In this context, the task to prove Theorem~\ref{intro-prob} is to show that (certain refinements of) these averages actually determine entirely the distribution of random groups.  

There is a significant history of work on this ``moment problem'' for random abelian groups.
Heath-Brown \cite{Heath-Brown1994} and Fouvry and Kl\"uners \cite{Fouvry2006} proved and applied moment problem results for random $\F_2$-vector spaces to find the distribution
of Selmer groups of quadratic twists of the congruent number curve, and four ranks of class groups of quadratic fields, respectively.  
See also \cite{Ellenberg2016, Lipnowski2020,Wang2021} for other number theoretic applications of the moment problem for more general random abelian groups.
The second author \cite{Wood2017} proved and applied moment problem results for random finite abelian groups to find the distribution of Jacobians of random graphs.

In the setting of non-abelian groups, Boston and the second author \cite{Boston2017} proved and applied a moment problem result for random pro-$p$-groups, to determine the pro-$p$ completion of fundamental groups of random quadratic function fields as $q$ and the genus go to infinity.  
The first author \cite{Sawin2020} proved a moment problem result for random profinite groups with an action of a fixed finite group.
All of these prior results prove the \emph{uniqueness} aspect of the moment problem, i.e. that two distributions with the same moments are the same (under various hypotheses).  They are applied 
in the setting where one knows some distribution and its moments and wishes to show another distribution agrees because it has the same moments.  In our setting, there was no known conjectural distribution  for the profinite completions of random $3$-manifold groups, and so we have the more challenging task of constructing the distribution from the moments, the \emph{existence} aspect of the moment problem.  
One of the main achievements of this paper is the development of a method that explicitly constructs a distribution on random groups from its moments.  We expect this method can be generalized and will be of use in many other contexts (e.g. see below on our forthcoming work in number theory).
 To our knowledge, this paper is the first that constructs a distribution on groups from given moments.

One of the challenges in this paper is that the moments of $\widehat{\pi_1(M)}$ are in fact too large  not only to apply the results in \cite{Sawin2020} to find that they determine a unique distribution, but in fact they are too large to even  determine a unique distribution in theory.
To overcome this challenge requires two new efforts.  First, we provide a method that proves a nearly optimal results for the non-abelian group moment problem, i.e. it is known that multiple distributions can give the same moments just beyond our growth bound.  Second, we confront cases in which the moments are of a size 
where uniqueness in the moment problem fails.  In these cases, 
we leverage the information from Theorem~\ref{pi1-properties}, which shows that the groups $\widehat{\pi_1(M)}$ have certain properties, and in particular parity properties, that mean that we only seek a distribution on a smaller class of profinite groups.  On this class we are able to prove that the moments determine a unique distribution using the nearly optimal result mentioned above.

  Another major challenge is that the construction of the distribution from the moments involves many infinite alternating sums of group cohomology of general finite groups, and one needs to organize and simplify these sums sufficiently to be able to, e.g. prove analytic bounds on their growth and detect if they are $0$ or not.  

Our proof of Theorem~\ref{pi1-properties} (3) relies on understanding the sign of the action of elements in the mapping class group on a homology group $H^1(\Sigma_g,V)$.  This action is studied  by Grunewald,  Larsen,  Lubotzky, and Malestein \cite{Grunewald2015}, who use it to find quotients of $\pi_1(\Sigma_g)$ that are finite index subgroups of a wide range of arithmetic groups.  However our interest is in whether the action is through a certain index 2 subgroup, information which is lost if one only considers quotients up to finite index.

The parity properties \cref{pi1-properties}(3-4), in the special case where $V$ is an \emph{projective} representation of a quotient $G$ of $\pi_1(M)$, were previously obtained in the topology literature \cite{Lee,Davis1989}, using the language of \emph{semicharacteristics}. The connection to this prior work is explained in \cref{ss-semicharacteristics}. 

Many others have considered the model of random Heegaard splittings introduced by Dunfield and Thurston and proven asymptotic properties of these random $3$-manifolds.
Kowalski has given quantitative results on the first homology groups of random Heegaard splittings \cite[Proposition 7.19]{Kowalski2008}.
Maher \cite{Maher2010} found the distribution of the distance between the disk sets of random Heegaard splittings, and thereby deduced that a random Heegaard splitting 
%of fixed genus $g\geq 2$
 is hyperbolic with asymptotic probability 1.
Dunfield and Wong \cite{Dunfield2011} found the distribution of certain topological quantum field theories on a random  Heegaard splitting.
%This one is really about $g\ra\infty
%Section 2 of this paper also lies out the general setup well, including the contrast between taking just length to infinity and also taking g to infinity
 Rivin \cite{Rivin} determined a large number of properties of these random Heegaard splittings, including asymptotics for the size of the first homology group, their Kneser-Matve’ev complexity, volume, Cheeger constant, and the injectivity radius. 
%seems all fixed genus 
 Lubotzky, Maher, and Wu \cite{Lubotzky2016} found the growth of splitting distance and distribution of Casson
invariants of random Heegaard splittings, and gave an improved convergence bound for Maher's earlier result of asymptotic hyperbolicity.  
%fixed g
Hamenstaedt and Viaggi \cite{Hamenstaedt2021} have found information on the spectrum of the Laplacian of random Heegaard splittings, including an upper bound on the smallest eigenvalue.  %fixed g
 Viaggi \cite{Viaggi2021} has found the asymptotic volume of random Heegaard splittings. %fixed g
Feller, Sisto, and Viaggi \cite{Feller2020} gave a constructive proof of hyperbolicity for random Heegaard splittings, and use it to find the diameter
growth rate and systole decay, as well as show asymptotically the $3$-manifolds are not arithmetic or in a fixed commensurability class.  %fixed g
With the exception of \cite{DunfieldThurston,Dunfield2011}, most of this previous work has focused on Heegaard splittings of a fixed genus $g$, so the work is to understand the limit as the random walk on the mapping class groups grows in length.  In contrast, in our work, the main interest and difficultly is the limit as the genus goes to infinity.  
We remark that other models of random $3$-manifolds have also been studied; see \cite[Section 7.4]{Aschenbrenner2015} for an overview of this broad area of work.

In subsequent work of the authors, the methods of this paper have been extended from the study of profinite groups to pro-objects in a wide variety of categories \cite{Sawin2022},  and applied to give conjectures on the distributions of class groups of $G$-extensions of a fixed number field \cite{Sawin2023}, that in particular take into account the roots of unity in the number field.  (As noted by Malle \cite{Malle2008}, 
the original conjectures of Cohen, Lenstra, and Martinet \cite{Cohen1984,Cohen1990} need to be modified under the presence of roots of unity.) 
In a forthcoming paper by the authors, the same methods will be used to give results on the asymptotic distribution of $\pi'_1(C)$, where $C$ is a random 
$G$-cover of $\mathbb{P}^1_{\F_q}$, and $q\ra\infty$ and then the genus of the cover goes to infinity.   
Here $\pi'_1$ denotes the maximal quotient of $\pi_1$ of order relatively prime to $|G|$ and $q$.  This will lead to conjectures on the distributions of non-abelian generalizations of class groups.

\subsection{Outline of the paper}

In Section~\ref{s-3-properties},  we prove Theorem~\ref{pi1-properties} using, largely, methods of algebraic topology.  In Section~\ref{S:cv}, we prove several properties of the class $c_V$ that 
appears in Theorem~\ref{pi1-properties}.

In Section 3, we  review the definition by Dunfield and Thurston of a random model of 3-manifolds, and slightly strengthen a result of Dunfield and Thurston by calculating the expected number of surjections from the fundamental group of a random $3$-manifold to a fixed finite group $Q$ 
(the \emph{$Q$-moments} of the random group)
that send the fundamental class of that manifold to a fixed class in the group homology of $Q$.

In Section 4, we state a general probabilistic theorem (\cref{localized-counting}) that will imply \cref{intro-prob}, and then give a proof of the theorem relying on results which will be proven in the following few sections.  We also give a heuristic description of the approach of our proof to determine a distribution from its moments, and discuss the obstacles that arise.

In Section 5, we prove an inclusion-exclusion formula which expresses 
the number of surjections from a 3-manifold group to a fixed finite group $G$ that satisfy certain conditions (regarding not extending to other surjections)
 as a linear combination of the number of surjections from a 3-manifold groups to other finite groups $H$, without conditions on the surjections. 
Comparing the expectations of both sides of this formula is a crucial step in our construction of a distribution
(whose probabilities are essentially the expectation of the number of surjections to $G$ that don't extend to larger relevant groups) 
 from its moments (which are the expectation of the number of all surjections to $H$).

In Section 6, we show that the limiting $H$-moments of a sequence of random groups, when all such limits exist,  are equal to the $H$-moments of the limiting distribution. This convergence theorem for the moments of a random group is an essential step in our determination of a limiting distribution from the limits of moments.  Additionally, it shows that there is no escape of mass and that the limiting distribution we find is indeed a probability distribution.
We expect our convergence theorem will be useful in many other applications involving random groups. 

 In Section 7, we evaluate a linear combination of the limits of the moments of a random 3-manifold group by algebraic methods, in particular detailed analysis of the Lyndon-Hochschild-Serre spectral sequence.   This is where we find the particular formulas appearing in \cref{intro-prob}.
 This also completes all the ingredients for the proof of \cref{localized-counting}.

In Section~\ref{s-existence}, we  give  general criteria for the existence of $3$-manifold groups with certain finite quotients but not others
and prove \cref{closure-characterization}.  We also deduce from Theorems~\ref{pi1-properties} and \ref{localized-counting} several example results about the existence and non-existence of 3-manifolds with fundamental groups with specific prescribed conditions on their finite quotients.
Finally, we classify finite groups in the closure of the set of $3$-manifold groups in $\operatorname{Prof}$.

In Section 9, we prove the probabilistic results
that follow from \cref{localized-counting}, including \cref{intro-prob}.  
We also give formulas for the distribution of the first homology of a random 3-manifold, along with the torsion linking pairing, and 
the distribution of the maximal $p$-group or nilpotent class $s$ quotient of a random $3$-manifold group.
We show that for each finite $G$, the limiting probability of  a random $3$-manifold group having a $G$-cover with positive first Betti number is $0$.

In Section 10, we discuss questions for further research.

\subsection{Notation}\label{S:Notation}

{\bf Topology:}
We always assume manifolds to be connected. 
All our $3$-manifolds will be  oriented. 
For a $3$-manifold $M$ and a field $\kappa$, we denote the map $H^3 ( \pi_1(M), \kappa) \to H^3(M,\kappa) \to \kappa$ obtained by pullback and integrating against the fundamental class  by $\int_M$.  Moreover,  in the context of a specified map $\pi_1(M)\ra G$,  we also denote the composite map $H^3 (G, \kappa)\ra H^3 ( \pi_1(M), \kappa)\ra \kappa$ by $\int_M$.

Let $H_g$ be a genus $g$ handlebody and $\Sigma_g$ its boundary.  If $\sigma$ is an element of the mapping class group of $\Sigma_g$, it describes a closed, oriented $3$-manifold $M_\sigma$ given by gluing two copies $H_g^1,H_g^2$ of $H_g$ along their boundaries using the map $\sigma$, identifying $x$ on $H_g^1$ with $\sigma^{-1}(x)$ on $H_g^2$.  We take the orientation on $M$ to be the one that restricts to the orientation on $H_g^1$.

{\bf Vector spaces and fields:}
For a vector space $V$ over a field $\kappa$, 
we write $\wedge^2_\kappa V$ for the quotient of $V\tensor_\kappa V$ by the $\kappa$-subfield generated by $v\tensor v$ for each $v\in V$.  When the subscript is omitted, it means $\kappa=\F_p$ for some prime $p$.  
We write $\F_q$ for the finite field with $q$ elements.

{\bf Representations:}
In this paper, when we say $V$ is a representation of $G$, we always mean that $V$ is finite dimensional.

If $V$ is a representation  of a group $G$ over a field $\kappa$, the dual representation is $V^\vee:=\Hom_\kappa(V,\kappa).$
Note if $\kappa$ is a finite field of characteristic $p$, the trace map gives an isomorphism $\Hom_\kappa(V,\kappa)\ra \Hom_{\F_p}(V,\F_p)=
\Hom(V,\QZ)$.
We say $V$ is self-dual if we have an isomorphism of $G$-representations $V\cong V^\vee.$  We occasionally view such a $V$ as a vector space over a subfield $\kappa'$ of $\kappa$ as well, and so when there may be any confusion, we write \emph{$\kappa$-dual} and \emph{$\kappa$-self-dual}.  We say that $V$ is \emph{symplectic} or \emph{$\kappa$-symplectic} if there exists a $G$-invariant alternating, nondegenerate,  $\kappa$-bilinear form on $V$.%, or equivalently if $(\wedge^2_\kappa V)^G\ne 0$

If $V$ is a $\F_p$-self-dual irreducible representation of a group $G$ over $\F_p$, and we let $\kappa:=\Hom_G(V,V)$, then $V$ is either $\kappa$-symplectic,
% if $(\wedge^2_\kappa V)^G \ne 0$, 
\emph{symmetric} if $V$ is $\kappa$-self-dual but not $\kappa$-symplectic,
%$(V\tensor_\kappa V)^G\ne 0$ but $(\wedge^2_\kappa V)^G = 0$, and 
or \emph{unitary} if $V$ is not $\kappa$-self-dual.
%$(V\tensor_\kappa V)^G= 0.$

If $V$ is a representation of $G$ over $\kappa$ 
with a $G$-invariant symmetric,  nondegenerate,  $\kappa$-bilinear form, 
 then we say $V$ is \emph{symmetrically self-dual}.  In odd characteristic,
if $V$ is irreducible and $\kappa=\operatorname{End}_G(V),$ then symmetrically self-dual equivalent to symmetric.

{\bf Groups and homomorphisms:}
For any  abelian group $V$, we write $V^\vee:=\Hom(V,\QZ)$.  

By a homomorphism of profinite groups, we always mean a continuous homomorphism.

We write $\Surj(A,B)$ for the set of surjective morphisms from $A$ to $B$ (in whatever category we are considering $A$ and $B$).

{\bf Affine symplectic group:} 
Let $p$ be a prime and $V$  a vector space  over $\F_p.$
The map $\Phi: \Hom(V\tensor_{\F_p} V , \F_p) \ra \Hom(V\tensor_{\F_p}V , \F_p)$ sending $f\mapsto f-f^t$ (where $f^t(a\tensor b)=f(b\tensor a)$) 
gives a surjection onto the group of alternating (i.e. $f(v,v)=0$ for all $v\in V$) bilinear forms on $V$, and the kernel is the subgroup of symmetric forms.
There is also a map $h:\Hom(V\tensor_{\F_p} V , \F_p) \ra H^2(V,\Z/p^2\Z),$ using the bilinear form as a cochain and the map $\F_p \stackrel{\times p}{\ra} \Z/p^2\Z$.  
One can check that all symmetric forms are in $\ker(h)$ (even when $p=2$), and so $h\Phi^{-1}$ gives a map from alternating bilinear forms on $V$
to $H^2(V,\Z/p^2\Z).$  (Note this is not the same as using the alternating form as a cochain.  Rather, one writes an alternating form $\omega$ as $f-f^t$, and then uses $f$ as a cochain.)  This is the association that is used in the construction of the affine symplectic group (and works as written even if $V$ is also a vector space over a larger finite field).

{\bf Levels:}
For a set $\C$ of groups, we have defined the set of level-$\C$ groups to be what is also known as the \emph{formation} of groups generated by $\C$, i.e. the smallest set of isomorphism classes of groups that contains $\C$ and is closed under taking quotients and fiber products.  (Note fiber products are the same as subdirect products.)  What we call the the level-$\C$ completion, $G^{\C}$, is also known as the pro-$\hat{\C}$ completion, where $\hat{\C}$ is the set of level-$\mathcal{C}$ groups.

Note that for $G$ finitely generated,  $G^{\C}$ is finite 
 (\cite[Corollary 15.72]{Neumann1967}).
We will show that for profinite $G$ with finitely many open subgroups of each index,  $G^{\C}$ is finite (\cref{maximal-quotient-finite}). 
When $G^\C$ is finite,  it is a quotient of $G$ \cite[Lemma 3.2.1]{Ribes2010}, and it is the maximal  quotient of $G$ that is level-$\mathcal{C}$.  

Our definition of level-$\C$ groups is slightly more general than the definition used in \cite{Liu2019,Liu2020}. Previously, one said the level-$\C$ groups are the smallest set of groups containing $\C$ and closed under subgroups, quotients, and products (the variety of groups generated by $\mathcal{C}$). It's possible to check that the level-$\C$ groups in the old sense are the level-$\mathcal D$ groups in the new sense, where $\mathcal D$ consists of all subgroups of groups in $\C$.

Recalling that we define a basic open subset in $\Prof$ to be the set of all $X$ such that the maximal level-$\C$ quotient of $X$ is $G$ for some $\C$ and $G$, we see that every basic open subset with the old definition is a basic open subset in the new definition, but not vice versa (though they induce the same topology). Thus, \cref{intro-prob} giving a formula for the measure of each basic open subset is more powerful with the new definition, motivating the change of notation.

\subsection*{Acknowledgements} 
We would like to thank Ian Agol, Jordan Ellenberg, Ian Hambleton, Emmanuel Kowalski, Yuan Liu, Alan Reid, and Akshay Venkatesh for helpful conversations and comments on the manuscript.
The first author was supported by a Clay Research Fellowship and NSF grant DMS-2101491. The second author was partially supported by a Packard Fellowship for Science and Engineering,  NSF grants DMS-1652116 and DMS-2140043,  the Radcliffe Institute for Advanced Study at Harvard University, and a MacArthur Fellowship.

\section{Properties of the fundamental group of closed $3$-manifolds}\label{s-3-properties}

The goal of this section is to prove Theorem~\ref{pi1-properties}, i.e. to verify certain properties of the group cohomology of representations of 3-manifold groups over finite fields. We start with an easy lemma relating the cohomology of $M$ to that of $\pi_1(M)$. 

\begin{lemma}\label{pi1-to-M} Let $M$ be a manifold, and $V$ a representation of $\pi_1(M)$ over a field $\kappa$. The natural map $H^1 (\pi_1(M), V) \to H^1( M, V)$ is an isomorphism and the natural map $ H^2 ( \pi_1(M), V) \to H^2(M,V)$ is an injection. \end{lemma}

\begin{proof} Let $\tilde{M}$ be the universal cover of $M$. Then we have a Cartan-Leray spectral sequence whose second page is $H^p ( \pi_1(M), H^q ( \tilde{M}, V))$ converging to $H^{p+q} (M, V)$.  The five-term exact sequence
\[ \hspace{-.1in} 0 \to H^1 ( \pi_1(M), H^0 ( \tilde{M}, V)) \to H^1(M, V) \to H^0 ( \pi_1(M), H^1 ( \tilde{M}, V)) \to H^2 ( \pi_1(M), H^0  ( \tilde{M}, V)) \to H^2(M, V) \]
reduces to an exact sequence
\[ 0 \to H^1( \pi_1(M),V) \to H^1(M, V) \to 0 \to H^2( \pi_1(M), V) \to H^2(M, V)\]
because $H^0 (\tilde{M}, V) = V$ and $H^1(\tilde{M}, V)=0$. This gives both claims. \end{proof}

\cref{pi1-to-M}, together with Poincar\'e duality for $M$, gives \cref{pi1-properties}(2).  We can deduce \cref{pi1-properties}(1) by combining the following lemma with \cref{pi1-to-M}.

\begin{lemma}\label{M-duality} Let $M$ be a closed, oriented $3$-manifold and let $V$ be an irreducible representation of $\pi_1(M)$ over a field $\kappa$. We have \[ \dim H^1(M, V) = \dim H^1(M, V^\vee).\]

\end{lemma} 

\begin{proof} Because $V$ is irreducible, we have \[ \dim H^0( M, V) = \dim H^3(M,V) = \begin{cases} 1 & \textrm{if } V \cong \kappa \\ 0 & \textrm{otherwise} \end{cases}.\] Thus $\dim H^2(M, V) - \dim H^1(M,V)=\chi(M, V) =(\dim V)\chi(M) =0.$
Hence by Poincar\'e duality $\dim H^1(M, V) = \dim H^2(M, V)= \dim H^1(M, V^\vee) .$ \end{proof}

We now work towards \cref{pi1-properties}(3). We first prove the relatively straightforward characteristic zero analogue.

\begin{lemma}\label{characteristic-zero-parity} Let $M$ be a closed manifold.
 Let $V$ be a symplectic representation of $\pi_1(M)$ over a field $\kappa$ of characteristic $0$, for which the action of $\pi_1(M)$ factors through a finite group.
Then $\dim_{\kappa} H^1(M, V)$ is even. 
\end{lemma}

Our proof of \cref{characteristic-zero-parity}  works for group cohomology of a representation $V$ of a finitely-generated group, instead of twisted cohomology of a manifold. Which statement to use is only a matter of preference.

The fact that symplectic representations of finite groups over the complex numbers are quaternionic, crucial in the below proof, was earlier used, in a similar context but with completely different language and notation, by \cite{Davis1989} to control a semicharacteristic invariant. We discuss the relationship between semicharacteristics and our work in \cref{ss-semicharacteristics}.

\begin{proof} We may replace $\kappa$ with the field generated over $\mathbb Q$ by the matrix entries of generators of $\pi_1(M)$ acting on a basis of $V$, and then, embedding this field in $\mathbb C$, we may assume $\kappa = \mathbb C$.
Let $n =\dim V$. Because the action on $V$ factors through a finite group, some Hermitian form is preserved, and because $V$ has a nondegenerate alternating bilinear form, it is standard that if $V$ is irreducible then it has quaternionic structure.  It follows that $H^1(M, V)$ has a quaternionic structure. If $H^1(M,V)$ has dimension $k$ over the quaternions, it has dimension $2k$ over the complex numbers - in particular, this is always even.

If $V$ is not irreducible, then since $V$ has a nondegenerate invariant bilinear form, each irreducible representation $W$ must appear in $V$ the same number of times as $W^\vee$,  and if $W\cong W^\vee$ then either $W$ has an an invariant nondegenerate alternating bilinear form (and thus is quaternionic) or $W$ appears an even number of times in $V$. If $H: W \cong W^\vee$ is the anti-linear morphism given by the preserved Hermitian form on $W$, then $(v,f)\mapsto (-H^{-1}(f),H(v))$ gives a quaternionic structure on $W\times W^\vee$.
Thus $V$ is quaternionic in any case, and the lemma follows.  
\end{proof}

We now reinterpret \cref{characteristic-zero-parity} as a result about the mapping class group, using Heegaard splittings. We first show a suitable mapping class  exists, and then relate the parity to a certain determinant of the mapping class group element acting on a symmetrically self-dual representation. Because the determinant, which is always $\pm1$, is preserved by reduction mod $p$ for odd $p$, we will use this to deduce the odd characteristic case, i.e. Theorem~\ref{pi1-properties} (3).

\begin{lemma}\label{all-handlebodies-are-the-same} 
Let $Q$ be a finite group.  For sufficiently large $g$ the following holds. 
 Let $H_g$ be a handlebody with boundary $\Sigma_g$. Let $*$ be a base point of $\Sigma_g$.
 For two surjections $f_1,f_2\colon \pi_1( H_{g})\to Q$ there is a mapping class of $(\Sigma_g,*)$ that extends to a mapping class of $(H_g,*)$ and sends $f_1$ to $f_2$.\end{lemma}

\begin{proof}This was proven by Dunfield and Thurston in \cite[Proposition 6.25]{DunfieldThurston}, which shows that the outer automorphism group of $\pi_1( H_g)= F_g$, the free group on $g$ generators, acts transitively on surjections  $\pi_1( H_g)\to Q$ up to conjugacy, and the discussion on the same page, which explains that the mapping class group of $H_g$ surjects onto this outer automorphism group. It follows that the pointed mapping class group surjects onto the usual automorphism group of $\pi_1( H_g)$, and
this automorphism group acts transitively on surjections. \end{proof}

\begin{lemma}\label{mapping-class-exists} Let $M$ be a closed $3$-manifold with a fixed base point $*$. Let $Q$ be a finite group, and let $\pi_1(M) \to Q$ be a surjection. Let $M = H_{g}^1 \cup H_{g}^2$ be a Heegaard splitting of $M$ into genus $g$ handlebodies  $H_g^1, H_g^2$ whose intersection is equal to their boundary, a genus $g$ surface $\Sigma_g$ containing $*$.

For $g$ sufficiently large with respect to $Q$, there exists an element $\sigma$ in the pointed mapping class group of $(\Sigma_g,*)$ that preserves the induced surjection $\pi_1(\Sigma_g) \to \pi_1(M) \to Q$ and %extends to a homeomorphism $H_{g}^1 \to H_{g}^2$. 
such that $M$ is homeomorphic to $M_\sigma$, the Heegaard splitting associated to $\sigma$, via a homemorphism that is the identity on $\Sigma_g$
(which we have inside $M$ and $M_\sigma$ each by virtue of writing them as a Heegaard splitting).
\end{lemma}

\begin{proof} 

Let $\sigma'$ be a  homeomorphism of $(\Sigma_g,*)$ giving the Heegaard splitting.
This element $\sigma'$ may, however, send a surjection $f: \pi_1(\Sigma_g)\to Q$ ({that factors through $\pi_1(M)$}) to a different surjection $f(\sigma')^{-1} :\pi_1(\Sigma_g)\to Q$.  Both these surjections  factor through $\pi_1( H_{g}^2)$, so by \cref{all-handlebodies-are-the-same} there is a mapping class of $\Sigma_g$ that extends to a homeomorphism of $H_g^2$ and sends one to the other, and by composing $\sigma'$ with this mapping class, we obtain the desired $\sigma$.
\end{proof}

Conversely, given a finite group $Q$, a surface $\Sigma_g$ of genus $g$ with a base point $*$,  a surjection $f\colon \pi_1 (\Sigma_g) \to Q$, a handlebody $H_g$ with boundary $\Sigma_g$ such that $f$ factors through $\pi_1(H_g)$, and a mapping class $\sigma$ of $(\Sigma_g,*)$ preserving $f$, $H_g \cup_{\Sigma_g} \sigma(H_g)$ is a $3$-manifold with a surjection from its fundamental group to $Q$.

Given a surjection $f: \pi_1(\Sigma_g) \to Q$ and a representation $V$ of $Q$, mapping classes $\sigma$ of $(\Sigma_g, *)$ that preserve $f$ act on $H^i( \Sigma_g, V)$ for all $i$. 
Now we show the parity of $\dim_\kappa H^1(M, V) $ is determined by the sign of the determinant of a mapping class group element on $H^1(\Sigma_g, V)$.

\begin{lemma}\label{parity-is-sign} Let $M$ be a closed $3$-manifold, expressed as $H^1_g \cup_{\Sigma_g} H^2_g$ for $H^i_g$  copies of a handlebody $H_g$ of genus $g$ and $\sigma$ a mapping class of the boundary $\Sigma_g$ of $H_g$, fixing a base point $*$ in $\Sigma_g$.

 Let $f\colon \pi_1( \Sigma_g) \to Q$ be a surjection which factors through $\pi_1(H^1_g)$ and is preserved by $\sigma$, so that $f$ descends to a surjection $\pi_1(M) \to Q$.

Let $V$ be a symplectic representation
 of $Q$ over a field $\kappa$, of characteristic not equal to $2$. 

Then $\dim_\kappa H^1(M, V) + \dim_\kappa H^0(M, V) $ 
is even if  and only if $\det ( \sigma, H^1(\Sigma_g, V))$ is $+1$ and odd if  and only if $\det ( \sigma, H^1(\Sigma_g, V))$ is $-1$. 
\end{lemma}

\begin{proof} We have a non-degenerate bilinear form $H^1(\Sigma_g, V) \times H^1 (\Sigma_g, V)  \to \kappa$ by taking the cup product, using the symplectic structure of $V$, and integrating. Since the cup product on $H^1$ is alternating and $V$ is symplectic, this is a symmetric bilinear form, which also gives a quadratic form by evaluation on the diagonal. 
By Poincar\'e duality, it is non degenerate. Because $\sigma$ preserves this symmetric bilinear form, it must have determinant $\pm 1$. 

We can check that $H^1(H^i_g,V)$ is a subspace of $H^1(\Sigma_g, V)$ for $i=1,2$ using the long exact sequence of a pair.
Then $H^1( H^i_g, V)\subset H^1(\Sigma_g, V)$ is an isotropic subspace for this quadratic form because the cup product of two elements of $H^1(H^i_g, V)$ lies in $H^2(H^i_g,V^{\tensor 2})$, which vanishes. We have \[\dim H^1( H^i_g, V) =(g-1) \dim V + \dim H^0(H^i_g,V)= (g-1) \dim V+ \dim V^Q \] and \[\dim H^1(\Sigma_g, V) = (2g-2) \dim V + 2 \dim H^0(\Sigma_g, V) = (2g-2) \dim V+ 2 \dim V^Q\] by Euler characteristic computations, so $H^1( H^i_g, V)$ is a maximal isotropic subspace. The Mayer-Vietoris sequence gives a long exact sequence
\[  H^0 ( M, V) \to H^0(H^1_g, V)  \oplus H^0 (H^2_g, V) \to H^0 (\Sigma_g, V) \] \[\to H^1(M, V) \to H^1 ( H^1_g, V) \oplus H^1(H^2_g,V) \to H^1 (\Sigma_g, V).\]

The maps $H^0(M, V) \to H^0(H^i_g,V) \to H^0(\Sigma_g, V)$ are induced by the natural maps $V^{ \pi_1(\Sigma_g) } \to V^{ \pi_1(H^i_g)} \to V^{\pi_1(M)}$, which are isomorphisms because the maps $\pi_1(\Sigma_g) \to \pi_1(H^i_g) \to \pi_1(M)$ are surjective. Hence the initial part $ H^0 ( M, V) \to H^0(H^1_g, V)  \oplus H^0 (H^2_g, V) \to H^0 (\Sigma_g, V) $ is itself short exact.

Furthermore, we have $H^1 (H^2_g,V) =\sigma ( H^1(H^1_g,V))$ as a subspace of $H^1(\Sigma_g, V)$. So we have an exact sequence 
\[ 0 \to H^1(M, V) \to H^1 ( H^1_g, V) \oplus \sigma ( H^1(H^1_g,V)) \to H^1 (\Sigma_g, V).\] 
In other words, $H^1(M,V)$ is the intersection of the maximal isotropic subspace $H^1(H^1_g, V)$ with its image under $\sigma$. 
The result then follows from the fact that  $\dim H^1(\Sigma_g, V) = (2g-2) \dim V + 2\dim H^0(M,V) $ is congruent to $2 \dim H^0(M,V)$ modulo  $4$ 
and the general observation that for an even-dimensional quadratic space $W$, an orthogonal automorphism $\sigma$, and a maximal isotropic subspace $S$, we have $\dim (S \cap \sigma(S) ) \equiv \frac{ \dim (W)}{2} \mod 2$ if $\det ( \sigma, W)=1$ and $\dim (S \cap \sigma(S) ) \equiv 1+ \frac{ \dim (W)}{2} \mod 2$ if $\det ( \sigma, W)=-1$ \cite[Example T.3.5 (see also Corollary T.3.4 and Theorem L.3.1)]{ConradNotesGroups}.
\end{proof}

\begin{lemma}\label{characteristic-zero-sign}  Let $\Sigma_g$ be a Riemann surface of genus $g$. Let $Q$ be a finite group. Let $V$ be a symplectic representation of $Q$ over $\mathbb C$.
Let $f\colon \pi_1(\Sigma_g)\to Q$ be a surjection. Suppose that $f$ factors through $\pi_1(H_g)$ for some handlebody $H_g$ with boundary $\Sigma_g$.
 Let $\sigma$ be a mapping class of $\Sigma_g$, together with its base point $*$, that preserves $f$.
  Then \[ \det ( \sigma, H^1( \Sigma_g, V))=1 .\] 
  \end{lemma}
 
 \begin{proof}
Since $V$ is symplectic and over $\mathbb{C}$, we have that $\dim H^0(M,V)=\dim V^Q$ is even.
  The lemma follows from applying \cref{characteristic-zero-parity,parity-is-sign} to the manifold $M = H_g \cup \sigma(H_g)$.
 \end{proof} 
 
 We can now extend \cref{characteristic-zero-sign} from characteristic zero to odd characteristic:
 
 \begin{lemma}\label{odd-characteristic-sign}  Let $\Sigma_g$ be a Riemann surface of genus $g$. Let $Q$ be a finite group. Let $V$ be a 
 self-dual  representation of $Q$ over a finite field $\kappa$.
Let $f\colon \pi_1(\Sigma_g)\to Q$ be a surjection. Suppose that $f$ factors through $\pi_1(H_g)$ for some handlebody $H_g$ with boundary $\Sigma_g$.
 Let $\sigma$ be a mapping class of $\Sigma_g$, together with its base point $*$, that preserves $f$.
  Then \[ \det ( \sigma, H^1( \Sigma_g, V))=1 .\] 
  \end{lemma}
 
 \begin{proof}  We will use algebraic properties of $\det ( \sigma, H^1( \Sigma_g, W))$ that hold for a general representation $W$ of $Q$ over an arbitrary field $\kappa$. Note that $\det ( \sigma, H^1( \Sigma_g, W))=\det ( \sigma, H^1( \Sigma_g, W^\vee))^{-1}$
  by Poincar\'e duality (which completes the proof of the lemma in characteristic 2). We now assume the characteristic of $\kappa$ is  odd.

  First, for an exact sequence $0 \to W_1 \to W_2 \to W_3 \to 1$ of representations of $Q$, we have \[ \det ( \sigma, H^1( \Sigma_g, W_2))=  \det ( \sigma, H^1( \Sigma_g, W_1)) \det ( \sigma, H^1( \Sigma_g, W_3)).\] This follows from the long exact sequence on cohomology and the fact that $H^0 (\Sigma_g, W)= W^Q$ and $H^2(\Sigma_g,W)=W_Q$ are fixed by $\sigma$. 
 Thus, this determinant defines a homomorphism from the representation ring to $\kappa^\times$. 
 
 Second, note that $\det ( \sigma, H^1( \Sigma_g, W))$ is preserved when we reduce a representation from characteristic zero to characteristic $p$. More precisely, if $Q$ acts on a free module $M$ over a local ring $R$ with residue field $\kappa$ of characteristic $p$ and fraction field $K$ of characteristic $0$,  we have
$\det ( \sigma, H^1( \Sigma_g, M\tensor_R \kappa))= \det ( \sigma, H^1( \Sigma_g, M\tensor_R K))$.  This is because alternating products of determinants on cohomology are preserved by change of coefficients and, again, $H^0$ and $H^2$ don't contribute.
 % Thus, the determinant homomorphism is compatible with the reduction mod $p$ map on the representation ring.
 
We note that the determinant is preserved under extension of the field of scalars.  Thus we may assume $\kappa$ is a splitting field for  $Q$, and that $\kappa$ is a residue field at some prime $\wp$ of a number field $K$ which is  a splitting field for all subgroups and quotients of $Q$.  Further, since the composition series of a self-dual representation contains all non-self-dual irreducibles in equal multiplicity with their duals, it suffices to prove the theorem for irreducible self-dual representations.
 
 If $R$ is the ring of elements of $K$ with positive valuation at $\wp$,  every representation of $Q$ over $K$ has a $K$-basis such that the action of $Q$ is given by matrices over $R$, and thus can be reduced modulo $\wp$ to a representation of $Q$ over $\kappa$ \cite[Theorem 73.6]{CurtisReiner}.  
 Though the reduced representation over $\kappa$ is not unique, its multiset of isomorphism classes of composition factors is unique \cite[Theorem 82.1]{CurtisReiner}.
 It will thus suffice for us to find irreducible representations $W_i$ of $Q$ over $K$ and coefficients $c_i\in \Z$ such that this reduction sends $\prod_i W_i^{c_i}$ to 
a representation over $\kappa$ whose multiset of composition factors is an odd number 
of copies of $V$ and such that \[ \prod_i \det ( \sigma, H^1(\Sigma_g, W_i))^{c_i}= 1.\]
 
 To do this, we use some ideas from modular representation theory, specifically, the decomposition matrix and the Cartan matrix of $Q$. The decomposition matrix $D$ is defined as the matrix with one column for each irreducible representation over {$K$} and one row for each irreducible representation over $\kappa$, with the entry giving the multiplicity of the
 {representation over $\kappa$ in the composition series of the reduction (as above) of the representation over $K$}. The Cartan matrix $C$ is defined as $D D^T$. The key fact we need is that $\det C$ is a power of {the characteristic of $\kappa$}, and in particular is odd \cite[Theorem 84.17]{CurtisReiner}.  

Thus  \[ D^T C^{-1}  (\det C)\] is a matrix with integral entries. We let $W_i$ be all the irreducible representations of $Q$ over $K$ and let $c_i$ be the entry of  $D^T C^{-1}  (\det C)$ in the row corresponding to $W_i$ and the column corresponding to $V$.  Then the reduction of $\sum_i c_i[W_i]$ mod $p$ is $ (\det C)[V]$ because \[ D D^{T} C^{-1} (\det C) = (\det C) C C^{-1} = (\det C) I \] by definition. Because $(\det C)$ is odd, it remains to calculate
\begin{equation}\label{reduction-p-product} \prod_i \det ( \sigma, H^1(\Sigma_g, W_i))^{c_i}.\end{equation}
Because reduction mod $p$ commutes with duality, the matrix $D$ is preserved by swapping rows and columns with the rows and columns corresponding to dual representations. Because this holds for $D$, it holds for $C$, and thus for $D^T C^{-1}(\det C)$. Because $V$ is self-dual, it follows that $c_i =c_j$ if $W_i \cong 
\Hom(W_j,\kappa)$.  Because  $H^1(\Sigma_g, W_j)$ and $H^1(\Sigma_g, \Hom(W_j,\kappa))$ have inverse determinants, the contributions of each representation to this product cancels the contribution of its dual, leaving only the self-dual representations.

For the %orthogonal self-dual 
symmetrically self-dual
representations $W_i$, cup product and integration define a symplectic form on $H^1 ( \Sigma_g, W_i)$, which is preserved by $\sigma$, ensuring that the determinant of $\sigma$ is $1$.
For the symplectic representations, it follows from \cref{characteristic-zero-sign}.  Since $K$ is a splitting field of characteristic $0$, every self-dual irreducible representation of $Q$ is either symmetrically self-dual or symplectic.
Hence all the factors cancel, so the product \eqref{reduction-p-product} is $1$, and thus the determinant of $\sigma$ acting on $H^1(\Sigma_g, V)$ is $1$, as desired.  
 \end{proof}

 \begin{lemma}\label{odd-characteristic-parity} Let $M$ be a closed, oriented $3$-manifold. Let $V$ be a symplectic representation
  of $\pi_1(M)$ over a finite field $\kappa$ of odd characteristic. 
Then $\dim_{\kappa} H^0(M, V)+\dim_{\kappa} H^1(M, V)$ is even. \end{lemma}

\begin{proof} Because $\kappa$ is finite, $V$ necessarily factors through a finite group $Q$. Choose a Heegaard splitting of $M$ whose genus is sufficiently large that \cref{mapping-class-exists} can be applied. The result then follows from \cref{parity-is-sign} and \cref{odd-characteristic-sign}.
\end{proof}

Combining Lemmas \ref{pi1-to-M} and \ref{odd-characteristic-parity} and the fact that for an irreducible symplectic (and thus non-trivial) representation $V$ of $\pi_1(M)$ we have $H^0(M, V)=0$, 
we conclude \cref{pi1-properties}(3).

We now turn to the characteristic $2$ case of parity in order to prove \cref{pi1-properties}(4). Our proof of the parity condition in this case relies on cobordism. Rather than directly compute $\dim H^1(M,V) \mod 2$, we will show that $\dim H^1(M,V) \mod 2$ depends only on the class of $M$ in a certain bordism group. We then calculate this bordism group.

For a CW complex $X$, the \emph{oriented bordism group of $X$} in dimension $n$ is the group generated by classes $[M]$ for each $n$-dimensional oriented (smooth, closed) manifold with a map to $X$, subject to the relations $[M \cup N ] = [M]+[N]$ and $[M]=0$ if $M$ is the boundary of an $n+1$-dimensional closed, oriented manifold with boundary endowed with a map to $X$ \cite{AtiyahBordism}.
We will be interested in the case when $X = BG$ for a finite group $G$ (as described in \cite[\S2]{ConnerFloyd}). The same argument will work for infinite discrete groups. 

For the next lemma, we extend the definition of $\ASp_\kappa(V)$ from finite fields $\kappa$ of characteristic $2$ to arbitrary perfect fields $\kappa$ of characteristic $2$ by defining $\mathcal H$ to be the central extension of $V$ by $\mathbb Z/4$ with extension class corresponding to the composition of the symplectic form with a fixed nonzero group homomorphism $T \colon \kappa \to  2\Z/4\Z$ and $\ASp_{\kappa}(V)$ to be the group of automorphisms of $\mathcal H$ acting trivially on $\mathbb Z/4$ and $\kappa$-linearly on $V$.

Note if $\kappa$ is a finite field, then $T$ is the composition of multiplication by $\lambda\in \kappa$ and the trace map.  Then the isomorphism $\lambda^{1/2}:V\ra V$ gives an isomorphism between the two Heisenberg groups and thus an isomorphism between their automorphism groups. In particular, in the finite field case, we may as well assume $T$ is the trace.

\begin{lemma}\label{parity-is-bordism} Let $n$ be a natural number and let $V$ be an even-dimensional representation of a discrete group $G$ over a  field $\kappa$ which has an invariant nondegenerate $\kappa$-bilinear form  
that is symmetric if $n$ is even and is symplectic if $n$ is odd. If $\kappa$ has characteristic $2$, we also assume that $n=1$, and $\kappa$ is perfect, and the homomorphism $G \to \Sp_\kappa(V)$ lifts to the affine symplectic group $\ASp_\kappa(V)$. 

Then the map that sends a $2n+1$-dimensional closed oriented smooth manifold $M$ with a map to $BG$ to 
$\sum_{i=0}^n \dim_\kappa H^i ( M, V) \mod 2$ depends only on the class of $M$ in the oriented bordism group of $BG$, and 
\[ [M] \to \sum_{i=0}^n \dim_\kappa H^i ( M, V) \mod 2\] defines a homorphism from this oriented bordism group  of $BG$ to $\mathbb Z/2$. 

\end{lemma}

\begin{proof}

 Because $\sum_{i=0}^n \dim H^i ( M, V) \mod 2$ is certainly additive under disjoint unions of the manifold, if it is a well-defined function then it is automatically a homomorphism.

Let $M_1$ and $M_2$ be two manifolds that are cobordant, i.e. there is a manifold $M'$ whose boundary is $M_1 \cup M_2$, with the orientation on $M_2$ reversed. We follow the standard procedure to express a cobordism as a series of Dehn surgeries, checking that it works in the presence of a map to $BG$. We can define a function $f$ on $M'$ that takes the value $0$ on $M_1$, $1$ on $M_2$, and values in $(0,1)$ on all other points of $M'$. Perturbing $f$, we can assume that $f$ is a Morse function, i.e. has only simple critical points, and takes different values on these critical points.

The level sets of $f$ between the critical values are then manifolds, so by induction among these manifolds we may assume that $f$ has exactly one critical point. Then $M_1$ and $M_2$ can be related by Dehn surgery, i.e., removing a submanifold of $M_1$ of the form $S^k \times B^{2n+1-k}$ and then gluing a submanifold of the form $B^{k+1} \times S^{2n-k}$ onto the same $S^k \times S^{2n-k}$ boundary to obtain $M_2$. Furthermore, because the Dehn surgery arises from the local geometry of $M'$ near the critical point, we can trivialize the $G$-bundle near the critical point and so assume that it is trivial on $S^k \times S^{2n-k}$.

We can relate the cohomology of the manifold before and after removing $S^k \times B^{2n+1-k}$ using the Mayer-Vietoris  sequence. Letting $U$ be the complement 
of $S^k \times B^{2n+1-k}$ in $M_1$, we have a long exact sequence
\[ H^i (M_1, V) \to H^i ( U, V) \oplus H^i ( S^k \times B^{2n+1-k}, V) \to H^i (S^k \times S^{2n-k}, V) .\]
Because $V$ is trivial on $S^k \times B^{2n+1-k}$, we have $H^i ( S^k \times B^{2n+1-k}, V)= V$ if $i=0$ or $k$ and $0$ otherwise (unless $k=0$, in which case it is $V \oplus V$ if $k=0$). Similarly, $H^i (S^k \times S^{2n-k}, V) $ vanishes unless $i= 0,k,2n-k, 2n$. In particular, if $ k\neq n$ then $H^n ( S^k \times S^{2n-k},V)=0$, so truncating the long exact sequence in degree $n$ and taking Euler characteristics, we obtain
\[ \sum_{i=0}^n (-1)^i \dim H^i (M_1,V)  + \sum_{i=0}^{n-1} (-1)^i \dim H^i ( S^k \times S^{2n-k}, V) \] \[ = \sum_{i=0}^n (-1)^i \dim H^i (U, V) + \sum_{i=0}^n (-1)^i \dim H^i ( S^k \times B^{2n+1-k}, V).\]

Since $\dim H^i ( S^k \times S^{2n-k}, V) $ and $\sum_{i=0}^n (-1)^i \dim H^i ( S^k \times B^{2n+1-k}, V)$ are divisible by $\dim V$, which is divisible by $2$, we obtain 
\[ \sum_{i=0}^n (-1)^i \dim H^i (M_1,V)  \equiv  \sum_{i=0}^n (-1)^i \dim H^i (U, V) \mod 2.\]
Applying the same argument to $M_2$, swapping $k$ and $2n-k$, we obtain the equality mod $2$ in the case $k \neq n$.

In the more difficult case $k=n$, we obtain by the same argument the equality
\[ \sum_{i=0}^n (-1)^i \dim H^i (M_1,V)  + \sum_{i=0}^{n} (-1)^i \dim H^i ( S^k \times S^{2n-k}, V) \]
\[-  (-1)^n \dim \coker ( ( H^n (U, V) \oplus H^n ( S^n \times B^{n+1}, V) ) \to H^n (S^n \times S^n, V) ) \]
\[= \sum_{i=0}^n (-1)^i \dim H^i (U, V) + \sum_{i=0}^n (-1)^i \dim H^i ( S^k \times B^{2n+1-k}, V).\]
Subtracting the analogous formula for $M_2$, it suffices to check that 
\begin{equation}\label{parity-surgery} \begin{split} & \dim \coker ( ( H^n (U, V) \oplus H^n ( S^n \times B^{n+1}, V) ) \to H^n (S^n \times S^n, V) ) \\  &\equiv \dim \coker ( ( H^n (U, V) \oplus H^n ( B^{n+1} \times S^n , V) ) \to H^n (S^n \times S^n, V) ) \mod 2 .\end{split}\end{equation}

Let $\omega$ be the nondegenerate form from the hypothesis.
Under the identification $H^n(S^n\times S^n,V)\cong V\times V,$ there is a natural quadratic form $Q$ on 
$H^n(S^n\times S^n,V)$ given by $Q(v_1,v_2)=\omega(v_1,v_2)$. 
In particular, this quadratic form has associated bilinear form $B_Q( (v_1,v_2),(w_1,w_2))=\omega(v_1,w_2)+\omega(w_1,v_2)$. Since $\omega$ is nondegenerate, $B_Q$ is nondegenerate.  Thus, (by definition) the quadratic form $Q$ is nondegenerate.  

We note  that $H^n ( S^n \times B^{n+1}, V) $ and $ H^n ( B^{n+1} \times S^n , V) $ are complementary maximal isotropic subspaces of $H^n (S^n \times S^n, V)$ for the quadratic form $Q$.
If we knew that $\im (H^n (U, V) \ra H^n (S^n \times S^n, V))$ was also maximal isotropic, we would obtain the parity condition \eqref{parity-surgery}. Indeed, it would follow from the fact  \cite[Example T.3.5]{ConradNotesGroups} that  for $W$ a vector space with a nondegenerate quadratic form of dimension divisible by $4$, $W_1, W_2$ complementary maximal isotropic subspaces, and $W_3$ a third maximal isotropic subspace, \[ \dim W / ( W_1 +W_3) \equiv \dim W / (W_2 + W_3) \mod 2.\]

It remains to verify that $\im (H^n (U, V) \ra H^n (S^n \times S^n, V))$ is maximal isotropic.
To find the dimension of the  image of $H^n(U, V)$, we can use the long exact sequence on relative cohomology to obtain
\[ H^0_c(U,V) \to H^0(U,V) \to H^0(S^n \times S^n, V)\to H^1_c(U,V) \to \]
\[ \dots \to  H^{2n} (S^n \times S^n, V)\to H^{2n+1} _c(U, V) \to H^{2n+1}(U,V).\]
By Poincar\'e duality, the first term in this sequence is dual to the last, and the second term dual to the second-from last, and so on.
 Thus the rank of the first map in this sequence is equal to the rank of the last map, and inductively the rank of the $r$th map is equal to the rank of the $r$th from last map, so the two maps
\[ H^n(U, V)\to H^n(S^n\times S^n, V) \to H^{n+1}_c ( U,V) \] have equal rank, and thus the image of $H^n(U, V)$ in  $H^n(S^n\times S^n, V)$ has half the dimension of $H^n(S^n\times S^n, V)$.
It remains to check that the image of $H^n(U, V)$ is isotropic, 
 which we do separately in the even and odd characteristic cases by giving additional formulas for the quadratic form.

We can also check that the cup product map, followed by $\omega$, followed by integration:
$$
H^n(S^n\times S^n,V) \times H^n(S^n\times S^n,V) \stackrel{\cup}{\ra} H^{2n}(S^n\times S^n,V\otimes V) \stackrel{\omega}{\ra}
H^{2n}(S^n\times S^n,\kappa)\ra \kappa
$$
is exactly $B_Q$.  Since the map $H_{2n}(\delta U,\kappa)\ra H_{2n}(U,\kappa)$ is the zero map, we see that $B_Q$ vanishes on
$H^n(U,V)\times H^n(U,V)$, since the cup product of two classes in the image of $H^n(U, V)$ lies in the image of $H^{2n}(U, V\otimes V)$ and thus, after applying $\omega$, lives in the image of $H^{2n}(U,\kappa)$, which vanishes.  In characteristic not $2$, this implies that 
$H^n(U,V)$ is isotropic for $Q$.

When $n=1$, we can give another construction of the quadratic form as follows.  Let $\bar{\beta}$ be a bilinear form on $V$ such that $\bar{\beta}  - \bar{\beta}^T=\omega$, and let $\beta$ be the composite of $\bar{\beta}$ with a homomorphism $T: \kappa \ra R$ for any abelian group $R$.
Let $H$ be the group extension of $V$ by $R$ corresponding to the
element in $H^2(V,R)$ whose co-cycle is given by $\beta$.  Consider the map
$$
H^1(S^1\times S^1,V)\ra H^2(S^1\times S^1,R)\ra R,
$$
given by the extension $H$ followed by integration.  If we identify $H^1(S^1\times S^1,V)\cong V \times V$  and $H^2(S^1\times S^1,\kappa)\ra \kappa$ (via integration), then we can compute that  the above map  sends $(v,v')\in V\times V$ to $\beta(v_1,v_2)-\beta(v_2,v_1)$, which is $T\omega(v_1,v_2)$, i.e. the composite  map  $H^1(S^1\times S^1,V)\ra R$ is $TQ$.
If the action of $\pi_1(U)$ on $V$ can be extended to an action on $H$ that is trivial on $R$, then by the same consideration as above (integration on $U$ over $H_{2}(\delta U,R)$ is 0), this realization of the quadratic form allows us to see that $TQ$ is $0$ on $H^n(U,V)$.  
When the characteristic of $\kappa$ is $2$, we apply this with $R=\Z/4\Z$ and $T$ the fixed nonzero linear form $\kappa \to \mathbb Z/2 \cong 2\mathbb Z/4$, so the group $\ASp_\kappa(V)$ gives automorphisms of $H$ trivial on $R$. 
We have $TQ(x)=0$ for all $x\in H^n(U,V)$, and for $\lambda\in \kappa$, we have an element $\sqrt{\lambda}\in \kappa$, and $T(\lambda Q(x))=TQ(\sqrt{\lambda} x)=0$.   Since $T$ is nonzero, we have that $Q(x)=0$ for all $x\in H^n(U,V)$, as desired. 
\end{proof}

\begin{lemma}\label{bordism-is-homology} For any CW complex $X$, the natural map from the oriented bordism group of $X$ in dimension $3$ to the homology $H_3( X, \mathbb Z)$ that sends a closed, oriented $3$-manifold $M$ to the fundamental class of $M$ inside $X$ is an isomorphism.

\end{lemma}

\begin{proof} The Atiyah-Hirzebruch spectral sequence of oriented bordism is a homological spectral sequence whose second page is $H_p(X, \Sigma_q)$ where $\Sigma_q$ is the oriented cobordism group of $q$-manifolds, and which converges to a complex whose $p+q$th term is the $p+q$th bordism group of $X$ \cite[Theorem 1.2]{ConnerFloyd}.

Now $\Sigma_q=0$ for $q=1,2,3$, so the only nonvanishing terms with $p+q<3$ have $q=0$, all the differentials out of $H_3(X, \Sigma_0)$ on the second and higher pages vanish. There are no differentials into $H_3(X, \Sigma_0)$ on the second and higher pages, since the differentials in a homological spectral sequence decrease the first index and increase the second index.

So the third oriented bordism group is $H_3(X,\mathbb Z)$. The fact that the fundamental class map witnesses this isomorphism is \cite[third sentence after Theorem 1.2]{ConnerFloyd}.
\end{proof}

Finally, we can deduce  \cref{pi1-properties}(4) from the following corollary of Lemmas \ref{parity-is-bordism} and \ref{bordism-is-homology}.

\begin{cor}\label{universal-class-exists} For each vector space $V$ over a finite field $\kappa$ of characteristic $2$, endowed with a symplectic form, there is a unique $2$-torsion class $c_V$ in the group cohomology $H^3 (\ASp_\kappa(V), \QZ)$ such that for any closed, oriented 3-manifold $M$ with a homomorphism $\pi_1(M) \to \ASp_\kappa(V)$, twice the integral of $c_V$ over $M$ is congruent to $\dim_{\kappa} H^1(M,V) + \dim_{\kappa} H^0(M,V)$ modulo $2$. \end{cor}

Since $c_V$ is $2$-torsion, it necessarily arises from a class in $H^3 (\ASp_\kappa(V), \mathbb Z/2)$, although this class is not necessarily unique.  For the statement of \cref{pi1-properties}(4), we choose an arbitrary preimage, and, abusing notation, refer to it also as $c_V$, but in the body of the paper we will always take $c_V \in H^3 (\ASp_\kappa(V), \mathbb Q/\mathbb Z)$. 
Combining Corollary~\ref{universal-class-exists} and the fact that for an irreducible symplectic (and thus non-trivial) representation $V$ of $\pi_1(M)$ we have $H^0(M, V)=0$, 
we conclude \cref{pi1-properties}(4).

\begin{proof} This follows immediately from combining \cref{parity-is-bordism} and \cref{bordism-is-homology}, using the duality between homology with coefficients in $\mathbb Z$ and cohomology with coefficients in $\QZ$.
\end{proof}

\subsection{The affine symplectic group and the class $c_V$} \label{S:cv}
Throughout this subsection, we assume that $\kappa$ is a finite  field of characteristic $2$.  We will give an alternate description of the affine symplectic group, show some stability properties of the class 
$c_V\in H^3(\ASp_\kappa(V),\QZ)$, and prove $c_V$ is non-trivial. 

 If $V$ is a finite dimensional vector space of dimensional $2n$ over $\kappa$, with a symplectic form $\omega$, then we can choose a standard basis $e_i$ of $V$ (i.e. so $\omega(e_i,e_j)=0$ unless $j=i+n \pmod{2n}$, in  which case $\omega(e_i,e_j)=1$). 
 Let $W:=W_2(\kappa)$ be the ring of length two Witt vectors over $\kappa$. 
 Then we let $V_2=W^{2n}$, with basis $\tilde{e}_i$ so that $V_2/2V_2\ra V$ is an isomorphism of $\kappa$-vector spaces taking $\tilde{e}_i$ to $e_i$.  We define a symplectic form $\tilde{\omega}$ on $V_2$ by $\tilde{\omega}(\tilde{e}_i,\tilde{e}_j)=0$ unless $j=i+n \pmod{2n}$, in which case $\tilde{\omega}(\tilde{e}_i,\tilde{e}_j)=1$ if $1\leq i\leq n$ and $\tilde{\omega}(\tilde{e}_i,\tilde{e}_j)=-1$ if $n+1\leq i \leq 2n$.  We also view $\tilde{\omega}$ as a $2n\times 2n$ matrix over $W$  in the usual way.
%Hmmmm..
% Note that for $v,w\in V_2$ with images $\bar{v},\bar{w}\in V$, we have $\omega(\bar{v},\bar{w})=2\tilde{\omega}(v,w)$.
Let $\Sp(V_2)$ be the group of $W$-module automorphisms of $V_2$ that preserve $\tilde{\omega}$, and we can check that there is a surjection
$$
\Sp(V_2) \ra  \Sp(V)
$$
induced by the reduction map $V_2\ra V$.  
The kernel $K$ of the above map consists of matrices $M\in \Sp(V_2)$ such that $M=I+2A$,  or equivalently,
$2n\times 2n$ matrices $M$ over $W$ such that $M=I+2A$ and $2 \tilde{\omega} A$ is symmetric.   So we see that $K$ is isomorphic to the additive group of symmetric $2n \times 2n$ matrices over $\kappa$ via the map that sends $I+2A$ to the reduction of $\tilde{\omega} A$ mod $2$.   Symmetric forms in characteristic $2$ have a natural homomorphism, evaluation on the diagonal, whose kernel corresponds to matrices with $0$ diagonal.  
This gives an exact sequence $1\ra N \ra K \ra \Hom(V,\kappa)\ra 1 $, where $N$ is the set of $M=I+2A$ such that $2\tilde{\omega}(v,Av)=0$ for all $v\in V_2$.  
One can easily check that $N$ is normal in $\Sp(V_2)$, and we define $\ASp_\kappa(V)$ to be $\Sp(V_2)/N,$ and have the exact sequence
$$
1\ra \Hom_\kappa(V,\kappa) \ra \ASp_\kappa(V) \ra \Sp_\kappa(V)\ra 1.
$$

\begin{lemma}
For $\kappa$ a finite field, the definition above agrees with our definition of $\ASp_\kappa(V)$ from the introduction. 
\end{lemma}
\begin{proof}
Let $T$ be the trace map $W\ra \Z/4\Z$. 
Next, we let $\tilde{H}$ be the central extension of $V_2$ by $\Z/4\Z$ whose class in $H^2(V_2,\Z/4\Z)$ is given by the cocycle $T\tilde{\omega}$, i.e. as  a set 
$\tilde{H}$ is  $V_2\times \Z/4\Z$ and $(v,c)(w,d)=(v+w,c+d+T\tilde{\omega}(v,w))$.  Note that $2V_2\times 0$ is a normal subgroup of $\tilde{H}$, 
and we can define $H:=\tilde{H}/(2V_2\times 0)$.  We have an exact sequence
$$
1\ra \Z/4\Z\ra H \ra V \ra 1.
$$
Define the $W$-bilinear form $B: V_2\times V_2\ra W$ so that $B(e_i,e_j)$ is $1$ if $j=i+n$ and is $0$ otherwise.  Then $B(v,w)-B(w,v)=\tilde{\omega}(v,w)$.
We will now see that $H$ is the same as the group $\mathcal{H}$ in the definition of the affine symplectic group. 
For each element $v\in V$, we choose a lift $\tilde{v}\in V_2$, and we can check that the element $[(\tilde{v},-TB(\tilde{v},\tilde{v})]\in H$ does not depend on the choice of lift $\tilde{v}$.
For $v,w\in V_2$, we have the following equality in $H$,
$$
[(v,-TB(v,v)][(w,-TB(w,w)]=[(v+w,-TB(v+w,v+w)][(0,2TB(v,w))].
$$
Since $2B-2B^t$ gives the symplectic form on $V$, we see that the cocycle in $H^2(V,\Z/4\Z)$ defining the multiplication in $H$ agrees with the cocycle defining the group $\mathcal{H}$.

 We now give an action of $\Sp(V_2)$ on $H$ by $M \cdot [(v,a)] = [(M \cdot v, a)]$. 
 Because we defined the equivalence relation and group operation solely in terms of addition and the symplectic form, this action is well-defined.
Since for $M=I+2A\in N$ and $v\in V_2$ we have 
$[(v+2Av,0)]=[(v,T\tilde{\omega}(v,2Av))]=[(v,0)]$, we see that $\ASp_\kappa(V)$ acts on $H$.
Finally, we check that $\ASp_\kappa(V)$ is exactly the group of automorphisms of $H$ preserving the central $\mathbb Z/4$ which are $\kappa$-linear modulo $\mathbb Z/4$. 
Since $\ASp_\kappa(V)$ surjects onto $\Sp_\kappa(V)$, it suffices to check this when restricting to elements that act trivially on $V$.
One can see directly from the definition of multiplication on $H$ that the automorphisms of $H$ acting trivially on $V$ and on $\Z/4\Z$ are exactly the maps
$[(v,0)]\mapsto [(v,\alpha(v))],$ where $\alpha\in \Hom(V,\Z/4\Z)$, and that an element $\gamma\in \Hom_\kappa(V,\kappa)$ (viewed as an element of $\ASp_\kappa(V)$ as above) is an
automorphism of $H$ such that $\alpha=T\gamma$.  Thus we can conclude that this definition of $\ASp_\kappa(V)$ agrees with our original definition.
\end{proof}

This new description is convenient for making observations about the stability of the class $c_V$.  

\begin{enumerate}
\item If $V$ and $V'$ are symplectic vector spaces over $\kappa$, then the above description of $\ASp_\kappa$ in terms of matrices of Witt vectors allows us to  see 
that we have a map $\ASp_{\kappa}(V)\ra \ASp_{\kappa}(V\oplus V')$ (mapping $M\mapsto [\begin{smallmatrix}M & 0\\ 0 & I\end{smallmatrix}]$).
By the defining property of $c_V$, we see that the map $H^3(\ASp_{\kappa}(V\oplus V'),\QZ)\ra  H^3(\ASp_{\kappa}(V),\QZ)$ sends $c_{V\oplus V'}\mapsto c_V$.

\item If we view a $\kappa$-symplectic vector space $V$ as a representation over a subfield $\kappa'$, we can similarly use the descriptions in terms of matrices of Witt vectors to see we have a map
$\ASp_{\kappa}(V)\ra \ASp_{\kappa'}(V)$ and a corresponding map $H^3(\ASp_{\kappa'}(V),\QZ)\ra  H^3(\ASp_{\kappa}(V),\QZ)$.
By the defining property of $c_V$, we have $c_V \mapsto  [\kappa:\kappa'] c_V$ in this map, and in particular if  $[\kappa:\kappa'] $ is even then $c_V$ maps to $0$, and if $[\kappa:\kappa'] $  is odd, then $c_V$ is preserved.

\item If we have a symplectic vector space $V$ over $\kappa$, and $\kappa'$ is a finite extension of $\kappa$, then from the descriptions in terms of matrices of Witt vectors we can see we have a map
$\ASp_{\kappa}(V)\ra \ASp_{\kappa'}(V\tensor_\kappa \kappa')$ and a corresponding map $H^3(\ASp_{\kappa'}(V\tensor_\kappa \kappa'),\QZ)\ra  H^3(\ASp_{\kappa}(V),\QZ)$.
By the defining property of $c_V$, we have $c_{V \otimes_\kappa \kappa'} \mapsto c_V$ in this map.

\item Since $\dim_\kappa H^i(M,V)$ doesn't depend on the lift of $\pi_1(M)\ra \Sp_\kappa(V)$ to $\pi_1(M)\ra \ASp_\kappa(V)$, for any finite group $G$ with two maps $\phi_i : G \ra \ASp_\kappa(V)$ that agree in the quotient to $\Sp_\kappa(V)$, then $\phi_1^* (c_V)=\phi_2^*(c_V)$, where $\phi_i^* : H^3( \ASp_\kappa(V),\QZ)\ra H^3(G,\QZ)$.
\end{enumerate}

\begin{prop}\label{cvnonzero}
For any finite dimensional symplectic vector space $V$ over a finite field $\kappa$ of characteristic $2$, we have that $c_V\in H^3(\ASp_\kappa(V),\QZ)$ is non-zero.
\end{prop}
\begin{proof}
Let $G$ be the binary octahedral group, which has the quaternion group $Q_8$ as a normal subgroup with quotient $S_3$.
Consider $M= S^3/G$, a spherical 3-manifold whose fundamental group is $G$.
Let $W$ be the two-dimensional representation over $\mathbb F_2$ on which $\pi_1(M)$ acts by $G\to S_3 = GL_2(\mathbb F_2)$. 
We have $H^0(M,W)=W^{S_3}=0$.
We have $H^*(A_3,W)=0$ and hence $H^*(S_3,W)=0$ by the Lyndon-Hochschild-Serre spectral sequence.  This allows us to compute
$\dim H^1(M,W)=1$ from the Lyndon-Hochschild-Serre spectral sequence.  

We will check that $\ASp_{\F_2}(W)$ is $S_4$. 
  The group $\ASp_{\F_2}(W)$ is some extension of $S_3$ by $W^\vee\cong W$, compatible with the action of $S_3$ on $W$. Restricted to $A_3$, this extension must split as a semidirect product because both groups appearing have relatively prime orders. The semidirect product is $A_4$, so $\ASp_{\F_2}(W)$ contains $A_4$ as an index $2$ subgroup, but the only such group which also surjects onto $S_3$ is $S_4$.  
  %So $\pi_1(M) \to \Sp_{\F_2}(W)$ lifts to $\ASp_{\F_2}(W)$. 
  %In fact, 
  Because the map $S_4\to S_3$ admits a section, every surjection of any $\pi_1(M)$ to $\Sp_{\F_2}(W)$ automatically lifts onto $\ASp_{\F_2}(W)$.  
  
  Since $\dim H^0(M,W)+\dim H^1(M,W)=1$,  Corollary~\ref{universal-class-exists}
   shows that $c_W$ in $ H^3 ( \ASp_{\F_2}(W), \QZ)$ integrates nontrivially over $M$, and hence $c_W$ is non-zero. By stability properties (1) and (3), the class $c_V$ is nontrivial for any $V$.  
\end{proof}

\begin{remark}\label{R:fromS3}
Note that the identity map $S_4\ra S_4$ and the composite $S_4\ra S_3 \ra S_4$ of the section and the quotient are two maps $\ASp_{\F_2}(W)\ra \ASp_{\F_2}(W)$, that
agree in the quotient to $\Sp_{\F_2}(W)$.  Thus by stability property (4) above, we have that $c_W\in H^3 ( \ASp_{\F_2}(W), \QZ)$ pulls back from
$c_W\in H^3 ( \Sp_{\F_2}(W), \QZ),$ which has a unique non-zero $2$-torsion class. 
\end{remark}

\section{The Dunfield-Thurston random model and its moments}\label{S:DT}

In this section we describe the Dunfield-Thurston model for a random $3$-manifold and find the moments of the random groups given by the fundamental group of these random $3$-manifolds.

Dunfield and Thurston \cite{DunfieldThurston} proposed a model for random $3$-manifolds, defined using Heegaard splitting.
Recall that an \emph{Heegaard splitting} of a 3-manifold $M$ is an expression of it as a union of two 
copes of the
genus $g$ handlebody $H_g$ after identifying their boundaries, each the Riemann surface $\Sigma_g$ of genus $g$. 
It is easy to see that the resulting 3-manifold only depends on the choice of identification up to isotopy, i.e. on the mapping class. 
It is known that every 3-manifold $M$ has a Heegaard splitting of some genus $g$. Thus, generating a random 3-manifold from a random Heegaard splitting will not exclude without reason any class of 3-manifolds. 

The mapping class group of genus $g$, i.e. the group of (oriented) homeomorphisms from $\Sigma_g$ to itself, up to isotopy, is known to be finitely generated. Fix a finite set $T$ of generators, including the identity. (It would be natural to choose the set $T$ to be closed under inverses, but this is not necessary for our results. We require that $T$ includes the identity so that when we later apply the Perron-Frobenius theorem we are in the setting of an \emph{aperiodic} Markov chain on a finite state space.)

Let the random variable $\sigma_{g,L}$ in the mapping class group be a random word of length $L$ in the generators $T$, i.e. the product of $L$ independent, uniformly random elements of $T$. We define the Dunfield-Thurston random 3-manifold $M_{g,L}:=M_{\sigma_{g,L}}$ (as defined in the notation section, i.e. 
  the union of two copies of $H_g$ after identifying their boundary with the mapping class $\sigma_{g,L}$).
Our results will all cover the statistical properties of $M_{g,L}$ in the limit where, first, $L$ is sent to $\infty$, and then, $g$ is sent to $\infty$.

An \emph{{\Mdcorated/} group}
 is a group $G$ together with an element $s\in H_3(G, \mathbb Z)$, 
called the \emph{orientation}, and a morphism of \Mdcorated/ groups $(G_1, s_1) \to (G_2,s_2)$ is a group homomorphism $G_1 \to G_2$ such that the pushforward of $s_1$ is $s_2$. We will use the notation $\mathbf G$ to denote an \Mdcorated/ group with underlying group $G$.
We write $\tau_G$ (or $\tau$) for the morphism $H^3(G, \mathbb \QZ)\ra\QZ$ corresponding to $s$, using the isomorphism
$H^3(G,\Q/\Z)\cong \Hom(H_3(G,\Z),\Q/\Z)$ from the Universal Coefficient Theorem, and the evaluation on $s$ map
$\Hom(H_3(G,\Z),\Q/\Z)\ra \Q/\Z$.  We can also describe this map  $H^3(G, \mathbb \QZ)\ra\QZ$ as the map obtained by integrating along the homology class $s$.
For $M$ an oriented, closed $3$-manifold, let $\mathbf{\pi_1(M)}$ be the \Mdcorated/ group with underlying group $\pi_1(M)$ and with $s$ 
the image of the fundamental class in the map $H_3(M,\Z)\ra H_3(\pi_1(M),\Z)$.

We will find in \cref{P:decmom} the limiting moments of the random \Mdcorated/ groups $\mathbf{ \pi_1 (M_{g, L})}$.
We build off of work of Dunfield and Thurston, who found the moments in the un\Mdcorated/ analog  \cite[Theorem 6.21]{DunfieldThurston}.  Having the \Mdcorated/ version will be essential to our work in this paper.
To make things precise, we will need to also consider the \emph{pointed mapping class group} of $(\Sigma_g,*)$, i.e. oriented, pointed homeomorphisms of $\Sigma_g$, up to pointed isotopy.

\begin{lemma}\label{cobordism-class-additive} Let $\Sigma_g$ be a surface of genus $g$, the boundary of a handlebody $H_g$, and fix a base point $*$ on $\Sigma_g$. Let $Q$ be a finite group, and let $f\colon \pi_1(\Sigma_g ) \to Q$ be a homomorphism.
Let $\sigma_1$ and $\sigma_2$ be two pointed mapping classes of $\Sigma_g, *$. Assume that $  f, f\sigma_1,f\sigma_1\sigma_2 \colon \pi_1(\Sigma_g ) \to Q$ all factor through $\pi_1(\Sigma_g) \to \pi_1(H_g)$.
Then  $[M_{\sigma_1 \sigma_2} ] = [ M_{\sigma_1}] + [M_{\sigma_2}]$ in the oriented bordism group of $BQ$, where 
$M_{\sigma_1}$ and $M_{\sigma_1 \sigma_2}$ are given maps to $BQ$ using $f$, and $M_{\sigma_2}$ is given a map to $BQ$ using $f\sigma_1$.
\end{lemma}

\begin{proof}  We do this by finding an explicit cobordism between $M_{\sigma_1 \sigma_2}$ and a disjoint union of $M_{\sigma_1}$, $M_{\sigma_2}$, and a third manifold which we can separately check is cobordant to zero. Observe that each of $M_{\sigma_1 \sigma_2}$,$M_{\sigma_1}$, $M_{\sigma_2}$  is a union of two copies of $H_g$, which we will call $H_g$ and $H_g'$ to distinguish them.
The map to $BQ$ is given by $f$ on the copy of $H_g$ in $M_{\sigma_1\sigma_2}$ and $M_{\sigma_1}$,  by $f\sigma_1$ on the copy of $H_g'$ in $M_{\sigma_1}$ and $H_g$ in $M_{\sigma_2}$, and by $f\sigma_1\sigma_2$ on the copies of $H_g'$ in $M_{\sigma_2}$ and $M_{\sigma_1 \sigma_2}$.

We start with $M_{\sigma_1\sigma_2} \times [0,1]$, $M_{\sigma_1} \times [1,2]$, and $M_{\sigma_2} \times [1,2]$.  

We glue $M_{\sigma_1\sigma_2} \times [0,1]$ to $M_{\sigma_1} \times [1,2]$ by identifying $H_g \times 1$ with $H_g \times 1$ via the identity map.

We glue $M_{\sigma_1\sigma_2} \times [0,1]$ to $M_{\sigma_2} \times [1,2]$ by identifying $H_g' \times 1$ with $H_g' \times 1$ via the identity map.

This produces a connected four-manifold with boundary, whose boundary components are $M_{\sigma_1\sigma_2} \times 0$ with negative orientation, $M_{\sigma_1} \times 2$ with positive orientation, $M_{\sigma_2} \times 2$ with positive orientation, and a fourth component $M'$, the union of the $H_g' \times 1$  from $M_{\sigma_1} \times [1,2]$ and the $H_g \times 1$ from $M_{\sigma_2} \times [1,2]$.

We can map this 4-manifold to $BQ$ because our gluings were compatible with the maps to $BQ$. This gives a relation $[M_{\sigma_1 \sigma_2} ] = [ M_{\sigma_1}] + [M_{\sigma_2}] + [M']$ in the bordism group of $BQ$. It remains to check that $[M']=0$.

This manifold $M'$ is the union of two copies of $H_g$ glued along the identity map  of their boundaries. To check $[M']=0$, we can use \cref{bordism-is-homology} to reduce to showing that the fundamental class of $M'$ vanishes in $H_3( BQ,\Z)$, and note that this fundamental class lies in the image of $H_3 ( \pi_1(M'),\Z)$, but $\pi_1(M') = \pi_1(H_g) = F_g$ is the free group on $g$ generators and has no higher homology.
 \end{proof}

\begin{lemma}\label{cobordism-class-is-homomorphism} Let $\Sigma_g$ be a surface of genus $g$, the boundary of a handlebody $H_g$, and fix a base point $*$ on $\Sigma_g$. Let $Q$ be a finite group, and let $f\colon \pi_1(\Sigma_g ) \to Q$ be a homomorphism. Assume that $f$ factors through $\pi_1(\Sigma_g) \to \pi_1(H_g)$.
Let $\mathcal{M}_{f,g}$ be the subgroup of the pointed mapping class group of $(\Sigma_g,*)$ preserving $f$. 
For $\sigma\in\mathcal{M}_{f,g},$ we have that $f$ extends to a map $\pi_1 (M_{\sigma}) \to Q$, giving a map $M_\sigma \to BQ$.

The function  sending $\sigma \in \mathcal{M}_{f,g}$ to the class of $M_\sigma$ in the oriented bordism group of $BQ$ is a homomorphism. If $g$ is sufficiently large depending on $Q$, and $f$ is surjective, then this homomorphism is surjective. \end{lemma}

\begin{proof}To prove the map is a homomorphism, we must check that for two mapping classes  $\sigma_1, \sigma_2\in \mathcal M_{f,g}$, that $[M_{\sigma_1 \sigma_2} ] = [ M_{\sigma_1}] + [M_{\sigma_2}]$ in the oriented bordism group. This  is a special case of \cref{cobordism-class-additive}. Indeed, if $f$ factors through $\pi_1(\Sigma_g) \to \pi_1(H_g)$ and $\sigma_1$ and $\sigma_2$ preserve $f$ then $f\sigma_1$ and $f\sigma_1\sigma_2=f$ also factor through $\pi_1(\Sigma_g) \to \pi_1(H_g)$.

 Let us check surjectivity. By \cref{bordism-is-homology} the bordism group is finite, so it suffices to show that each bordism class can arise for sufficiently large $g$. Each bordism class arises from a $3$-manifold $M$ with a homomorphism $\pi_1(M) \to Q$. 
We can take the connect sum of $M$ with many manifolds of the form $S^2 \times S^1$,
with homomorphisms from their fundamental group to $Q$, 
and without changing the class in the bordism group we can thus assume that we have $\pi_1(M) \to Q$ surjective.  By \cref{mapping-class-exists}, for $g$ sufficiently large we can split the 3-manifold $M$ into a union of two handlebodies $H_g$ glued by a mapping class $\sigma$ of the boundary $\Sigma_g $, in such a way that the homomorphism $F\colon \pi_1(\Sigma_g) \to \pi_1 ( M) \to Q$ is preserved by $\sigma$.
Thus the bordism class we are considering is the image of $\sigma\in \mathcal{M}_{F,g}.$

Furthermore, by \cref{all-handlebodies-are-the-same}, we can see that $f=F\sigma_0$ for some mapping class $\sigma_0$ of $(H_g,*).$
It follows that the bordism class we are considering is the image of $\sigma_0^{-1}\sigma\sigma_0\in \mathcal{M}_{f,g}.$
Thus, for any surjection $f: \pi_1(H_g) \to Q$, for $g$ sufficiently large, any fixed bordism class arises from some element of $\mathcal{M}_{f,g},$ so the group homomorphism is surjective, as desired.

Thus the lemma follows for $g$ large enough that each class in the bordism group is represented by a $3$-manifold that has a Heegaard genus at most $g$ and also large enough for  Lemmas \ref{all-handlebodies-are-the-same} and \ref{mapping-class-exists} to hold for $Q$. 
\end{proof}

\begin{prop}\label{P:decmom} For each $g, L$, let $M_{g,L}$ be a random variable valued in $3$-manifolds obtained by taking the genus $g$ Heegaard splitting arising from the mapping class of a uniform random length $L$ word $\sigma_{g,L}$ in a fixed generating set (including the identity) for the  genus $g$ mapping class group. For $\mathbf H$ a finite \Mdcorated/ group, we have
 \[ \lim_{ g\to \infty} \lim_{L \to \infty} \mathbb E \left[ \Surj( \mathbf{ \pi_1 (M_{g, L})}, {\mathbf H} ) \right] =  \frac{  \abs{H}  \abs{ H_2 (H, \mathbb Z ) } }{  \abs{ H_1(H, \mathbb Z) } \abs{ H_3(H, \mathbb Z) }} .\] 
\end{prop}

\begin{proof} 
First fix $g$ and $L$. 
Let $\Sigma_g$ be a surface of genus $g$, with base point $*$, and let $H_g$ be a handlebody with boundary $\Sigma_g$. 
 Now, we will refine things slightly and consider lifts of our generators of the mapping class group to the pointed mapping class group of $(\Sigma_g,*)$
so that we may lift $\sigma_{g,L}$ to a pointed mapping class.  Since
the pointed mapping class group surjects onto the usual mapping class group, this will not affect the distribution of $M_{g,L}$.
For $\sigma$ a mapping class of $\Sigma_g$ preserving $*$, %and $M_\sigma$ the manifold obtained by gluing two copies of $H_g$ along $\sigma$
 a surjection $\pi_1(M_\sigma) \to H$ is a surjection $f: \pi_1(\Sigma_g) \to H$, factoring through $\pi_1(H_g)$, whose pullback by $\sigma$ factors also through $\pi_1(H_g)$. 

The expectation of the number of  surjections 
${ \pi_1 (M_{g,L})} \ra {H} $
 is the sum over surjections $f \colon \pi_1( \Sigma_g )  \to Q$ factoring through $\pi_1(H_g)$ of the fraction of words $\sigma$ of length $L$ which send $f$ to another surjection factoring through $\pi_1(H_g)$.  As the word length grows, $f\sigma$ will equidistribute in its mapping class group orbit by the Perron-Frobenius theorem, and so the limit as $L$ goes to $\infty$ of this expectation is equal to the sum over $f$ of the fraction of elements in the orbit of $f$ that factor through $\pi_1(H_g)$.
Dunfield and Thurston  \cite[Theorem 6.21]{DunfieldThurston} showed that the limit (in $L$) converges to $  \frac{  \abs{H}  \abs{ H_2 (H, \mathbb Z ) } }{  \abs{ H_1(H, \mathbb Z) }}$ as $g$ goes to $\infty$.

The expectation of the number of \Mdcorated/ surjections  is the sum over $f$ of the fraction of $\sigma$ in the pointed mapping class group such that (1) $f\sigma$ factors through $\pi_1(H_g)$ and 
(2) the fundamental class $f_*[M_\sigma]$ is equal to $s\in H_3(H,\Z)$.  When $g$ is sufficiently large, by \cref{cobordism-class-is-homomorphism} and \cref{bordism-is-homology}, for $\sigma$ in the stabilizer of $f$, taking the fundamental class of $f_*[M_\sigma]$  gives a surjective  homomorphism from the stabilizer of $f$ to $H_3(H, \mathbb Z)$. Thus $\sigma$ which stabilize $f$ and give fundamental class mapping to $s$ form a coset for the kernel of $\mathcal{M}_{f,g}\ra H_3(H,\Z)$.  

In fact, the $\sigma$ that send $f$ to any other fixed $f' \colon \pi_1(\Sigma_g) \to H$ factoring through $\pi_1(H_g)$ and give fundamental class mapping to $s$
form a coset for the kernel of the homomorphism $\mathcal{M}_{f,g}\rightarrow H_3(H,\Z)$. This follows from the previous claim after composing with a mapping class of $\Sigma_g$ sending $f$ to $f'$ and extending to a mapping class of $H_g$, whose existence is guaranteed by \cref{all-handlebodies-are-the-same}.

So we fix a surjection $f:\pi_1(\Sigma_g)\ra H$.
For sufficiently large $g$, the limit as $L$ goes to $\infty$ of the fraction  of $\sigma$ of length $L$ satisfying this condition for which $f$ can be extended to an \Mdcorated/ surjection $\pi_1(M_\sigma)\ra H$  is
 the number of surjections $\pi_1(\Sigma_g)\ra H$ factoring through $\pi_1(H_g)$ in the mapping class group orbit of $f$ divided by 
the product of the size of the orbit of $f$ and $\abs{H_3(H, \mathbb Z)}$. So the limit as $L$ goes to $\infty$ is the fraction of surjections in the orbit factoring through $\pi_1(H_g)$ divided by $\abs{H_3(H, \mathbb Z)}$. Thus the limit as $L$ goes to $\infty$ of the expected number of \Mdcorated/ surjections is the limit as $L$ goes to $\infty$ of the expected number of surjections, divided by $\abs{H_3(H, \mathbb Z)}$.
Using Dunfield and Thurston's result, this converges as $g$ goes to $\infty$ to $  \frac{  \abs{H}  \abs{ H_2 (H, \mathbb Z ) } }{  \abs{ H_1(H, \mathbb Z) } \abs{ H_3(H, \mathbb Z) }} $.
\end{proof}

\section{The main theorem on the distribution}\label{s-main-lc}

Now we turn to determining the distribution of the (\Mdcorated/) group $\mathbf{ \pi_1 (M_{g, L})}$ from its moments determined in \cref{P:decmom}.
In this section we will state our main technical theorem on the distribution of $\mathbf{ \pi_1 (M_{g, L})}$ and set up the notation for the proof.
We will first describe our approach informally.  
Given a random group $\pi$, suppose one wants to determine the probability that $\pi\cong G$ for a fixed group $G$. Certainly in such situations there is a surjection $\pi \ra G$, in fact $|\Aut(G)|$ of them, so $\E[|\Surj(\pi,G)|]/|\Aut(G)|$ provides an upper bound on this probability.  However, this is likely an overestimate.  If a surjection $\pi \ra G$ is not an isomorphism, it factors through a surjection $\pi \ra E$, where $E$ is a minimal non-trivial extension of $G$.  Thus the moments  $\E[|\Surj(\pi,E)|]$ over
all $E$ tell us about the extent to which this is an overestimate, and 
$$
\frac{\E[|\Surj(\pi,G)|]}{|\Aut(G)|} -\sum_E \frac{\E[|\Surj(\pi,E)|]}{|\Aut(E)|}\frac{|\Surj(E,G)|}{|\Aut(G)|}
$$
is our next estimate for $\Prob(\pi\cong G)$, which would be correct if $\pi$ was supported only on $G$ and minimal non-trivial extensions of $G$.
More generally one can work out that this second estimate gives
 a lower bound on $\Prob(\pi\cong G)$.  The undercounting now is because $\pi$ might surject onto more than one minimal non-trivial extension of $G$,
and one could add another term to account for this, and continue on analogous to inclusion-exclusion, leading to an infinite sum.
There are two major obstacles to such an approach, the first algebraic and the second analytic.  The algebraic obstacle is that is it not at all clear how to evaluate an infinite sum involving a group, its extensions, surjections between them, and automorphisms and group cohomology (appearing, for us, in the moments) of the group and its extensions, \emph{for an arbitrary finite group} $G$.   The second obstacle is that a priori it is not clear that this infinite sum converges, and indeed in general it will not.  

In this paper we overcome both obstacles.  On the algebraic side, we relate the group cohomology of $G$ and its minimal extensions with precise formulas structured in such a way that we can evaluate the necessary infinite sum, and indeed express it as a product.  Of course, the group cohomology of $G$ is obviously related to the group cohomology of its extensions, but the work is in finding precise formulas that allow us to evaluate the infinite sum.  This requirement for workable formulas has necessitated  our considering \Mdcorated/ groups.

On the analytic side,  we must confront the fact that the infinite sums in some cases truly fail to converge.  This is where we use the topological input we have proven in Section \ref{s-3-properties}, which shows that the fundamental group of a 3-manifold is not an arbitrary group but rather has certain parity restrictions on its group cohomology.  These restrictions allow us to do the inclusion-exclusion in a smaller category of $G$-extensions where the sum will actually converge.

As an example, let $G$ be a finite group and $V$ an absolutely irreducible symplectic representation of $G$ over an odd characteristic finite field $\kappa$ such that $\dim_{\kappa} H^1(G,V)$ is even.   If $\pi_1(M)$ surjects onto $V\rtimes G$, then using  \cref{lem-even_ei} to compute $H^1(\pi_1(M),V)$  and  using \cref{pi1-properties} to show $H^1(\pi_1(M),V)$ is even dimensional, it follows that
 $\pi_1(M)$ also surjects onto $V^2\rtimes G$.  
 One can imagine how this kind of result allows us to skip steps in the inclusion-exclusion sketched above to obtain a sum that might converge even if the original one did not.

\subsection{Definitions}
Let $G$ be a finite group.  A \emph{$[G]$-group} is a group $H$ together with a homomorphism from $G$ to $\Out(H)$.  A morphism of $[G]$-groups is a homomorphism $f: H\ra H'$ such that for each element $g\in G$, for each lift $\sigma_1$ of the image of $g$ from $\Out(H)$ to $\Aut(H)$, there is a lift $\sigma_2$ of the image of $g$ from $\Out(H')$ to $\Aut(H')$ such that $\sigma_2\circ f=f\circ \sigma_1$.  
Note that $G$ acts on the set of normal subgroups of a $[G]$-group $H$, and we say a nontrivial $[G]$-group is \emph{simple} if it has no nontrivial proper fixed points for this action.  
If we have an exact sequence $1\ra N \ra H \ra G\ra 1$, then $N$ is naturally a $[G]$-group.  A normal subgroup $N'$ of $N$ is fixed by the $\Out(G)$ action if and only if it is normal in $H$.

\subsection{Setup}\label{ss-lc-notation}
Fix a finite \Mdcorated/ group $\mathbf G$. A minimal non-trivial extension of $G$ (in the sense that its map to $G$ does not factor through any non-trivial quotients)
 is either by an irreducible representation $V$ of $G$ over $\F_p$ for some prime $p$ (a finite simple abelian $[G]$-group) or a finite simple non-abelian $[G]$-group $N$.   In the first case, $H^2(G, V)$ classifies the different extensions, and in the second case there is a unique extension up to isomorphism given by the fiber product $\Aut(N)\times_{\Out(N)} G$ (where $\Aut(N)$ and $\Out(N)$ here are automorphisms and outer automorphisms of $N$ in the category of groups). 
   When we make our inclusion-exclusion argument, we will fix a finite set of these minimal extensions, and determine the expected number of surjections from a random $3$-manifold group to $G$ that don't extend to our chosen extensions.  Now we will choose and name those extensions.

 Fix a tuple $\underline{V}= (V_1,\dots, V_n)$ of irreducible representations of $G$ over fields $\F_{p_i}$ for primes $p_i$.
  Write $\kappa_i = \operatorname{End}_{{G}}(V_i)$, a finite field, and $q_i = |\kappa_i|$, a prime power. For each $i$, fix also a $\kappa_i$-subspace $W_i \subseteq H^2(G, V_i)$, forming a tuple $\underline{W}$.  We will only be avoiding extensions by $V_i$ whose extension class is in $W_i$. 
 Fix a tuple $\underline{N}= (N_1,\dots, N_m)$ of non-abelian finite simple $[G]$-groups.

We say the following (isomorphism classes of) extensions of $G$ are \emph{minimally material}. 
\begin{itemize}

\item  For $1\leq i \leq n$,  every extension $1 \to V_i \to H \to G \to 1$ whose extension class lies in $W_i$.

\item For $1\leq i \leq m$, the extension $\Aut(N_i)\times_{ \Out(N_i)} G\to {G}$.

\end{itemize}

For an \Mdcorated/ group  $\mathbf K$, define $L_{\mathbf G, \underline{V}, \underline{W}, \underline{N}} (\mathbf K )$ to be the number of surjective \Mdcorated/ morphisms 
$f\colon \mathbf K \to \mathbf G$  that do not 
%lift to surjections $\mathbf K\to \mathbf H$ to 
factor through any minimally material extension $\mathbf H\to \mathbf G$. The main case to keep in mind is the following.

\begin{lemma}\label{L:Linnicecase}
Let $\mathcal{C}$ be a finite set of finite groups and $\bG$ a level-$\mathcal{C}$ oriented group. 
If the   $\{V_i\}$ are the irreducible $G$-representations such that $V_i\rtimes G$ is level-$\mathcal{C}$, and $W_i$ is the set of all the 
level-$\mathcal{C}$ extensions of $G$ by $V_i$, and the
$\{N_i\}_i$ are the finite simple non-abelian $[G]$-groups $N$ such that $\Aut(N)\times_{ \Out(N)} G$ is level-$\mathcal{C}$, 
then (1)  $\{V_i\}_i$ and $\{N_i\}_i$ are finite sets, and (2) if $\bK$ is a finitely generated oriented group, then
 ${L_{\bG, \underline{V}, \underline{W}, \underline{N}} (\bK)}{}$ is $|\Aut(\bG)|$ when $\bK^\C\cong \bG$ and $0$ otherwise. 
\end{lemma}

\begin{proof}
The first claim is shown in \cite[Proof of Theorem 4.12]{Liu2019}.
The key feature required for the second claim is that level-$\mathcal{C}$ groups are closed under fiber products and quotients.  
Note that this implies the $W_i$ as defined in the lemma are $\kappa_i$-subspaces.
Since $K$ is finitely generated, $K^\C$ is finite and hence level-$\C$.
If $K^\C \ra G$ is a surjection that is not an isomorphism, then it factors through some minimal non-trivial extension of $G$, and that extension, since it is a quotient of $K^\C$, is level-$\C$, which proves the lemma.
\end{proof}

Our goal will now be to determine the asymptotics of $\E[L_{\mathbf G, \underline{V}, \underline{W}, \underline{N}} (\mathbf K )]$, which includes
computing the asymptotics of $\Prob({\bf \pi_1(M_{g,L})}^\C\cong \bG)$.  

Let $\tau: H^3 ( G, \mathbb Q/\mathbb Z) \to \mathbb Q/\mathbb Z$ be the map induced by integrating against the homology class of $\mathbf G$.
Let $\delta_{N_i}$ be the differential $ d^{0,2}_3 \colon H^2 ( N_i, \QZ)^{ {G} } \to H^3 ({G},\QZ)$  appearing in the Lyndon-Hochschild-Serre spectral sequence  computing $H^{p+q} (  {G} \times_{ \Out(N_i)} \Aut(N_i), \QZ) $ from $H^p ( {G}, H^q ( N_i, \QZ))$.
We define weights  $w_{N_i}=w_{N_i}(\tau)$  to be positive numbers depending on the above data, as in the following table.
\begin{center}
{Table 1: Definition of the $w_{N_i}$}\\
\begin{tabular}{ |c|c| } 
 \hline
 Condition &$w_{N_i}$  \\ 
  \hline  \hline
 $\tau \circ \delta_{N_i}=0$ &   $  e^{ - \frac{  \left| H^2( N_i, \mathbb Q/ \mathbb Z )^{{G}} \right| }{ | Z_{ \Out(N_i)} ( {G} ) |} }  $  \\   \hline
 $\tau \circ \delta_{N_i}\ne 0$ &  $1$ \\ 
 \hline
\end{tabular}
\end{center}
Here $Z_{ \Out(N_i)} ( {G} )$ is the centralizer of the image of $G$ in $\Out(N_i)$ (the outer automorphism group of $N_i$ as a group).

Next we will define analogous weights for the $V_i$.
For any $i$, let $W_i^{\tau}$ consist of all those $\alpha \in W_i$ such that $\tau ( \alpha \cup \beta) =0 $ for all $\beta \in H^1 ( {G}, V_i^\vee)$.
If $V_i$ has odd characteristic $p$ and is $\mathbb F_p$-self-dual, let $\epsilon_i$ be the Frobenius-Schur indicator, which is $1$ if $V_i$ is symmetric, $0$ if $V_i$ is unitary, and $-1$ if $V_i$ is $\kappa_i$-symplectic (see Section~\ref{S:Notation} for definitions).
If $V_i$ has even characteristic and is $\mathbb F_2$-self-dual, then either $V_i$ is trivial, in which case we set $\epsilon_i=1$, or $V_i$ is $\F_2$-symplectic, in which case we 
let $\epsilon_i$ be $-1$ if the action of $G$ on $V_i$ lifts to 
$\ASp_{\kappa_i}(V_i)$, $0$ if the action lifts to
$\ASp_{\F_2}(V_i)$ but not $\ASp_{\kappa_i}(V_i)$, and $1$ if the action doesn't lift to $\ASp_{\F_2}(V_i)$. 

Regardless of characteristic, we will say that $V_i$ is \emph{A-symplectic} if $\epsilon_i=-1$: in other words, in odd characteristic, ``A-symplectic" means the same thing as ``$\kappa_i$-symplectic", and in even characteristic, it refers to representations $V_i$ lifting to $\ASp_{\kappa_i}(V_i)$. Note that whether $V_i$ is A-symplectic in even characteristic does not depend on a choice of symplectic form. Since $V_i$ is irreducible, it is only possible to change the symplectic form by multiplication by a scalar $a\in \kappa_i$, and this is equivalent to multiplying each vector by $\sqrt{a}\in \kappa_i$ and thus does not change whether the action of $G$ lifts to $\ASp_{\kappa_i}(V_i)$.

If $\epsilon_i=-1$, define $c_{V_i}\in H^3(\ASp_{\kappa_i}(V),\QZ)$ as in \cref{universal-class-exists} if the characteristic is even and $c_{V_i}=0$ if the characteristic is odd.  Then the weights $w_{V_i}=w_{V_i}({\tau})$ are defined in the following table.

\begin{center}
{Table 2: Definition of the $w_{V_i}$}\\
\begin{tabular}{|c|c|c|} 
 \hline
  $V_i$ $\F_{p_i}$-self-dual? & Conditions & $w_{V_i}$  \\ 
  \hline  \hline
yes & $W_i^{\tau}\neq0 $  & $0$ \\
 \hline
 yes,  $\epsilon_i > -1$ & $W_i^{\tau}=0 $  &  $\prod_{j=1}^{\infty} ( 1+ q_i^{- j- \frac{\epsilon_{i}-1}{2}} )^{-1}$ \\
 \hline
 yes,  $\epsilon_i = -1$ & $W_i^{\tau}=0 $, $2 \nmid \dim_{\kappa_i} H^1(G, V_i) - 2 \tau (c_{V_i}) $&  $0$ \\
 \hline
  yes,  $\epsilon_i = -1$ & $W_i^{\tau}=0 $, $2 \mid \dim_{\kappa_i} H^1(G, V_i) - 2 \tau (c_{V_i})$ &  $\prod_{j=1}^{\infty} ( 1+ q_i^{- j })^{-1}$ \\
 \hline
 no,   &   & $\displaystyle{\prod_{j=1}^{\infty}  
 %   (1- q_i ^{ -j + \dim W_i^\tau + \dim H^1 (G, V_i^\vee) - \dim H^1 (\overline{G}, V_i) } )
\left(1-q_i^{-j} \frac{\abs{W_i^\tau}\abs{H^1(G, V_i^\vee)}}{\abs{H^1(G, V_i)}}\right) }$     
    \\
 $V_i^\vee\not\cong V_j$ any $j$  &   &  \\
    \hline
 no, $V_i^\vee\cong V_j$   & 
$W_i^\tau = W_j^\tau =0 $,  
       & $\prod_{k=1}^{\infty} (1 -q_i^{-k})^{1/2}$ \\
             for $j\ne i$ & 
 $\dim H^1( G, V_i) =\dim H^1(G, V_i^\vee) $ &
     \\
 \hline
   no, $V_i^\vee\cong V_j$   & $W_i^{\tau} \ne0$ or $W_i^\tau \neq 0$ or  & $0$ \\
               for $j\ne i$ &  $\dim H^1( {G}, V_i) \neq \dim H^1({G}, V_i^\vee) $ &  \\
               \hline
\end{tabular}
\end{center}

Note that $w_{V_i}$ vanishes in many cases. Notably, it vanishes whenever $W_i^\tau \neq 0$ and $V_i$ is dual to $V_j$ for some $j$ (regardless of whether $i=j$). Less obviously, if $V_i^\vee\not\cong V_j$ for any $j$, then $w_{V_i}=0 $ if $\dim H^1({G}, V_i) < \dim W_i^\tau + \dim H^1({G}, V_i^\vee)$.

Now with these weights defined we can state our main technical theorem on the distribution of $\mathbf{\pi_1 (M_{g,L} ) }$.
Let $\mu_{g,L, \mathcal C}$ be the probability measure of the random variable $\mathbf{\pi_1^{\mathcal C}(M_{g,L})}$, which is just slightly more convenient notation for $(\mathbf{\pi_1(M_{g,L})})^{\mathcal C}$.  Note that the level-$\mathcal{C}$ completion of an oriented group is naturally oriented.

\begin{theorem}\label{localized-counting} For each $g, L$, let $M_{g,L}$ be the Dunfield-Thurston random $3$-manifold as defined in Section~\ref{S:DT}. For every $\bG, \underline{V}, \underline{W}, \underline{N}$ as above,
\[ \lim_{ g\to \infty} \lim_{L \to \infty} \mathbb E \left[ L_{\bG, \underline{V}, \underline{W}, \underline{N}} (\mathbf{\pi_1 (M_{g,L} ) })\right]=  \frac{ 
 \abs{G}\abs{H_2 (G, \mathbb Z) }  }{ \abs{H_1(G,\mathbb Z)} \abs{H_3(G, \mathbb Z)}} \prod_{i=1}^n w_{V_i} \prod_{i=1}^m w_{N_i} .\]
In particular,  for $\C$ a finite set of finite groups, and $\bG, \underline{V}, \underline{W}, \underline{N}$ as in \cref{L:Linnicecase},
$$
\lim_{ g\to \infty} \lim_{L \to \infty}\mu_{g,L, \mathcal C}(\bG)
=  \frac{  \abs{G}\abs{H_2 (G, \mathbb Z) }  }{\abs{\Aut(\bG)} \abs{H_1(G,\mathbb Z)} \abs{H_3(G, \mathbb Z)}} \prod_{i=1}^n w_{V_i} \prod_{i=1}^m w_{N_i}.
$$
\end{theorem}
 In the case that  we take no $V_i$'s and $N_i$'s, then \cref{localized-counting} is just \cref{P:decmom}.  Taking all possible
$V_i$'s and $N_i$'s (relevant to a level-$\C$) gives the second statement of the theorem.  The first statement is a general flexible result that allows one to interpolate between these two extremes.

\subsection{Proof of Theorem \ref{localized-counting} from the major inputs}
Now we will state the main results that go into the proof of \cref{localized-counting} and show how the theorem follows from these inputs.

We first define the class of extensions of $G$ that will arise in our inclusion-exclusion formula.
We call a $G$-extension $H\ra G$ \emph{material} if it is a finite fiber product over $G$ of finitely many minimally material extensions, and 
a ${\bG}$-extension $\bH\ra \bG$
\emph{attainable} if  for each $i$ such that  $V_i$ is A-symplectic, we have $\dim_{\kappa_i}  H^1(H, V_i) \equiv 2\tau ( c_{V_i}) \mod 2 $ (motivated by Theorem~\ref{pi1-properties}). 

Next we will define the coefficients that will appear in our inclusion-exclusion formula.
Let $I$ be a set of finite \Mdcorated/ groups that includes exactly one from each isomorphism class.
For $\bH \in I$, we define a path $P$ from $\bH$ to $\bG$ to be sequence $\bH_s, \bH_{s-1},
\dots, \bH_0$ for some $s\geq 0$ with $\bH_i \in I$, with $\bH_s=\bH$ and $\bH_0=\bG$, along 
with choices $f_i \colon \bH_i \to \bH_{i-1}$ for each $1\leq i \leq s$ of surjective \Mdcorated/ morphisms that are not 
isomorphisms, such that each composite map $H_i\ra G$ is material. 
We write $\Path(\bH,\bG)$ for the set of such paths.
We write $\abs{P}=s$ for the length of the path and define
\begin{equation*}
 \alpha_P : = \prod_{i=0}^{|P|-1} \frac{1}{ \abs{ \Aut(\bH_i)} }  \quad\quad\quad
 \beta_P: = \prod_{\substack{j \\ V_j \textrm{ A-symplectic} \\ \dim_{\kappa_j} H^1 (\bH_i ,V_j) = \dim_{\kappa_j} H^1( \bG, V_j)+1 \textrm{ for some } i}}  \frac{1}{q_j}.
 \end{equation*}
 There is a path of length $0$ from $\bH$ to $\bG$ if and only if $\bH=\bG$, and there is one path of length $0$ from $\bG$ to $\bG$.

For $\bH\in I$, we define  \[ T_{\bH} : =\sum_{P\in\Path(\bH,\bG)} (-1)^{\abs{P}} \alpha_P \beta_P.\]
 Here the precise $\alpha_P$ factor is a necessary normalization. The exact value of $\beta_P$ factor is somewhat arbitrary, but some such factor is necessary to improve the rate of convergence of certain sums in the case where some $V_j$ is A-symplectic (and without it these sums will not converge). One can see many of the main ideas on a first read while ignoring the $\beta_P$ factor.

The following lemma is our basic inclusion-exclusion formula, which gives the number of surjections from a $3$-manifold group to $\bG$, not lifting to any minimally material extension, in terms of total numbers of surjections to various extensions of $\bG$.  We will prove Lemma~\ref{L:ie} in Section~\ref{S:ie} using group theory.

\begin{lemma}\label{L:ie}
Let $\bG, \underline{V}, \underline{W}, \underline{N}$ be as above. 
Assume $\bG\in I $ and $\bG$ is an attainable $\bG$-extension.
% i.e. for all $i$ such that $V_i$ is A-symplectic, we have $\dim_{\kappa_i}  H^1(G, V_i)  \equiv 2\tau ( c_{V_i}) \mod 2 $.
 %For any \Mdcorated/ group $\bK$, 
 If $\bK$ is the (oriented) %(\Mdcorated/) 
 fundamental group of a $3$-manifold, then we have
\[ \frac{L_{\bG, \underline{V}, \underline{W}, \underline{N}} (\bK)}{|\Aut(\bG)|} = \sum_{\bH \in I} \frac{ T_{\bH} }{ \abs{ \Aut(\bH)} } \Surj(\bK,\bH).\]
\end{lemma}

Given \cref{L:ie} and \cref{P:decmom}, to find $\E \left[ L_{\bG, \underline{V}, \underline{W}, \underline{N}} (\mathbf{\pi_1 (M_{g,L} ) })\right]$ we naturally seek to evaluate the  sum in the following proposition.  This is the most difficult part of the argument, and occurs in Section \ref{S:evalsum}, where we do a detailed spectral sequence analysis.  
\begin{prop}\label{P:evalsum}
If $\bG\in I$ is an attainable $\bG$-extension, and if for any  $i$ such that $V_i^\vee \cong V_j$ for some $j$ we have $W_i^\tau = 0$, then
$$\sum_{\bH \in I} \frac{ T_{\bH} }{ \abs{ \Aut(\bH)} } \frac{  \abs{H}  \abs{ H_2 (H, \mathbb Z ) } }{  \abs{ H_1(H, \mathbb Z) } \abs{ H_3(H, \mathbb Z) }}
= \frac{ \abs{G}  \abs{ H_2(G, \Z) } }{ \abs{\Aut(\bG)} \abs{ H_1(G, \Z)}  \abs{  H_3(G , \Z )} } \prod_{i=1}^r w_{V_i} \prod_{i=1}^s w_{N_i},$$
and the sum is absolutely convergent.
\end{prop}

However, to apply \cref{P:evalsum} to prove Theorem~\ref{localized-counting}, we need to handle several analytic questions of the existence of limits and whether they can be interchanged with infinite sums in our particular situation.  
For this, we use  \cref{L:llim} and \cref{P:robustness-exchange}, which are proven in Section \ref{S:switch}.
\cref{L:llim} is relatively straightforward as it is only about limiting behavior in $L$, but \cref{P:robustness-exchange} involves
some intricate group theory arguments along with the analytic arguments.

\begin{lemma}\label{L:llim}
Let $\bG, \underline{V}, \underline{W}, \underline{N}$ be as above,  let $\mathcal{C}$ be any finite set of finite  groups, and let $g\geq 1$.
The limit 
$$
\lim_{L\ra\infty} \mu_{g,L,\mathcal{C}}(\mathbf{G}),
$$ 
exists.   We define $\mu_{g,\infty,\mathcal{C}}(\mathbf{G})$ to be the limit above. 

\end{lemma} 

We say an oriented group is level-$\mathcal C$ if and only if the underlying group is level-$\C$.
 Let $I_{\mathcal C}$ be a set consisting of one representative of each isomorphism class of finite level-$\C$ \Mdcorated/ groups.
 
\begin{prop}\label{P:robustness-exchange}
For some finite set of  finite groups $\C$, for each $\bK \in I_\C$, let $p_{\bK}^n$ be a sequence of nonnegative real numbers such that  $\lim_{n \to \infty} p_{\bK}^n$ exists. Suppose that, for every $\bH\in I_\C$, we have
\begin{equation*}
{\sup_n} \sum_{\bK\in I_\C} |\Surj(\bK,\bH)| p_{\bK}^n <\infty.
\end{equation*}
Then for every $\bH\in I_\C$, we have that 
\begin{equation*}
\lim_{n \to \infty} \sum_{\bK\in I_\C}  |\Surj(\bK,\bH)| p_{\bK}^n= \sum_{\bK\in I_\C}  |\Surj(\bK,\bH)|  \lim_{n \to \infty} p_{\bK}^n.
\end{equation*}
\end{prop}

Finally, our input results Lemma~\ref{L:ie} and Proposition~\ref{P:evalsum} require certain hypotheses on $\bG$, but the following result, which follows from Lemmas~\ref{maximal-quotient-attainable} and \ref{lem:to_alpha_zero}, and is much easier than the rest of the argument,  shows that these are the only $\bG$ relevant for our purposes.
\begin{lemma}\label{L:zerocases}
If $\bG$ is not an attainable $\bG$-extension, or if for some $i, j$ we have $V_i^\vee \cong V_j$ and $W_i^\tau \neq 0$, then
$$
L_{\bG, \underline{V}, \underline{W}, \underline{N}} (\mathbf{\pi_1 (M_{g,L} ) }) =0.
$$
\end{lemma}

\begin{proof}[Proof of Theorem~\ref{localized-counting}]
When the hypothesis of Lemma~\ref{L:zerocases} is satisfied, we can see from the definition of the $w_{V_i}$ that the right-hand side of 
 Theorem~\ref{localized-counting} is $0$ as well, concluding the theorem in those cases.

Now we may assume $\bG$ is an attainable $\bG$-extension and that for any  $i$ such that $V_i^\vee = V_j$ for some $j$ we have $W_i^\tau \neq 0$.
Let $\mathcal{C}$ be any finite set of finite groups such that all minimally material extensions of $G$ are level-$\C$.
For the limit in $g$, it is not so clear that a limiting distribution of $\mu_{g,\infty, \mathcal C}$ even exists.  However, by a diagonal argument, we can always consider a weak limit.
 Let $\mu_{ \infty, \infty, \mathcal C}$ be a weak limit of $\mu_{g,\infty, \mathcal C}$ over a convergent sequence of $g$, i.e. a sequence $g_s$ chosen so that for all $\bK\in I_\mathcal{C}$, the limit $\lim_{s\rightarrow\infty} \mu_{g_s,\infty,\mathcal{C}}(\bK)$ exists (and $g_s\ra\infty$).

Since $\pi_1 (M_{g,L} )^\C $ is a quotient of $\pi_1^{\mathcal C}(\Sigma_g)$,  which is finite,
  $\pi_1 (M_{g,L} )^\C $ takes finitely many possible values, independently of $L$.
Thus, given $g_s$ and $\C$,  there is  finite subset of $I^\mathcal C$ containing the support of $\mu_{g_s,L,\mathcal C}$ for all $L$, and 
$$
\lim_{L\ra\infty} \sum_{\bK\in I_\C} |\Surj(\bK,\bH)| \mu_{g_s,L, \mathcal C}(\bK)=
 \sum_{\bK\in I_\C} |\Surj(\bK,\bH)| \mu_{g_s,\infty, \mathcal C}(\bK).
$$
We next apply \cref{P:robustness-exchange} with $p_{\bK}^s=\mu_{g_s,\infty, \mathcal C}(\bK)$, using the above equality and \cref{P:decmom} to check the hypothesis, and obtain, for any $\bH\in I$,
\begin{equation}\label{lm-ml}
\lim_{s\ra\infty} \sum_{\bK\in I_\C} |\Surj(\bK,\bH)| \mu_{g_s,\infty, \mathcal C}(\bK)=
 \sum_{\bK\in I_\C} |\Surj(\bK,\bH)| \mu_{\infty,\infty, \mathcal C}(\bK).
 \end{equation}
Combining the above two equations and \cref{P:decmom} 
we have the following,
\begin{align*}
\sum_{\bH \in I} \frac{ T_{\bH} }{ \abs{ \Aut(\bH)} } \sum_{\bK\in I_\C} |\Surj(\bK,\bH)| \mu_{\infty,\infty, \mathcal C}(\bK)
&=\sum_{\bH \in I} \frac{ T_{\bH} }{ \abs{ \Aut(\bH)} } \lim_{s\ra\infty} \lim_{L\ra\infty} \sum_{\bK\in I_\C} |\Surj(\bK,\bH)| \mu_{g_s,L, \mathcal C}(\bK)\\
&=\sum_{\bH \in I} \frac{ T_{\bH} }{ \abs{ \Aut(\bH)} } \frac{  \abs{H}  \abs{ H_2 (H, \mathbb Z ) } }{  \abs{ H_1(H, \mathbb Z) } \abs{ H_3(H, \mathbb Z) }},
\end{align*}
and moreover by \cref{P:evalsum}, all these sums are absolutely convergent.  So we can exchange the order of summation and obtain
\begin{align*}
\sum_{\bK\in I_\C} \mu_{\infty,\infty, \mathcal C}(\bK) \sum_{\bH \in I} \frac{ T_{\bH} }{ \abs{ \Aut(\bH)} } |\Surj(\bK,\bH)| 
&=\sum_{\bH \in I} \frac{ T_{\bH} }{ \abs{ \Aut(\bH)} } \frac{  \abs{H}  \abs{ H_2 (H, \mathbb Z ) } }{  \abs{ H_1(H, \mathbb Z) } \abs{ H_3(H, \mathbb Z) }}.
\end{align*}
Thus by  \cref{L:ie} we have
\begin{equation}\label{S2}
\sum_{\bK\in I_\C} \mu_{\infty,\infty, \mathcal C}(\bK) \frac{L_{\bG, \underline{V}, \underline{W}, \underline{N}} (\bK)}{|\Aut(\bG)|}
=\sum_{\bH \in I} \frac{ T_{\bH} }{ \abs{ \Aut(\bH)} } \frac{  \abs{H}  \abs{ H_2 (H, \mathbb Z ) } }{  \abs{ H_1(H, \mathbb Z) } \abs{ H_3(H, \mathbb Z) }}.
\end{equation}

In particular, in the case of main interest when the hypothesis of Lemma~\ref{L:Linnicecase} is satisfied, we use that lemma to see that Equation \eqref{S2} says
\begin{equation}\label{E:findprob}
 \mu_{\infty,\infty, \mathcal C}(\bG) 
=\sum_{\bH \in I} \frac{ T_{\bH} }{ \abs{ \Aut(\bH)} } \frac{  \abs{H}  \abs{ H_2 (H, \mathbb Z ) } }{  \abs{ H_1(H, \mathbb Z) } \abs{ H_3(H, \mathbb Z) }}.
\end{equation}
Since every weak limit of $ \mu_{g,\infty, \mathcal C}(\bG)$ is the same, it follows that $\lim_{g\ra\infty}\mu_{g,\infty, \mathcal C}(\bG)$ exists and is given as  above, which can be combined with Proposition~\ref{P:evalsum} to prove the second statement of the theorem.

For general $V_i,W_i,N_i$,  we must exchange our two limits with one final sum.
%Finally, we need to move the (implicit) limits from the inside of the sum above to the outside.
We have
\begin{equation}\label{justLfromfin}
\sum_{\bK\in I_\C} \mu_{g,\infty, \mathcal C}(\bK) \frac{L_{\bG, \underline{V}, \underline{W}, \underline{N}} (\bK)}{|\Aut(\bG)|}=
 \lim_{L\ra\infty}  \sum_{\bK\in I_\C} \mu_{g,L, \mathcal C}(\bK) \frac{L_{\bG, \underline{V}, \underline{W}, \underline{N}} (\bK)}{|\Aut(\bG)|}.
\end{equation}
because the only $\bK$ which give nonzero terms in the sum on each side are those level-$\C$ groups that can be generated by $2g$ elements, a finite set, and we may exchange finite sums with limits. 
For the limit in $s$,  Fatou's lemma gives
\begin{equation}\label{fatou-inf} \sum_{\bK\in I_\C}   \mu_{\infty,\infty, \mathcal C}(\bK) \frac{L_{\bG, \underline{V}, \underline{W}, \underline{N}} (\bK)}{|\Aut(\bG)|}  \leq \liminf_{s\ra\infty} \sum_{\bK\in I_\C} \mu_{g_s,\infty, \mathcal C}(\bK) \frac{L_{\bG, \underline{V}, \underline{W}, \underline{N}} (\bK)}{|\Aut(\bG)|}.\end{equation}
Since $L_{\bG, \underline{V}, \underline{W}, \underline{N}}(\bK) \leq \abs{\Surj (\bK , \bG)} $, Fatou's lemma also gives
$$\sum_{\bK\in I_\C}   \mu_{\infty,\infty, \mathcal C}(\bK) \frac{ \Surj (\bK , \bG)- L_{\bG, \underline{V}, \underline{W}, \underline{N}} (\bK)}{|\Aut(\bG)|}  \leq \liminf_{s\ra\infty} \sum_{\bK\in I_\C} \mu_{g_s,\infty, \mathcal C}(\bK) \frac{\Surj (\bK , \bG) -  L_{\bG, \underline{V}, \underline{W}, \underline{N}} (\bK)}{|\Aut(\bG)|}$$
which subtracted from \eqref{lm-ml} gives 
\begin{equation}\label{fatou-sup}\sum_{\bK\in I_\C}   \mu_{\infty,\infty, \mathcal C}(\bK) \frac{L_{\bG, \underline{V}, \underline{W}, \underline{N}} (\bK)}{|\Aut(\bG)|}  \geq \limsup_{s\ra\infty} \sum_{\bK\in I_\C} \mu_{g_s,\infty, \mathcal C}(\bK) \frac{ L_{\bG, \underline{V}, \underline{W}, \underline{N}} (\bK)}{|\Aut(\bG)|}.\end{equation}

Combining \eqref{fatou-inf} and \eqref{fatou-sup}, we have
$$\sum_{\bK\in I_\C}   \mu_{\infty,\infty, \mathcal C}(\bK) \frac{L_{\bG, \underline{V}, \underline{W}, \underline{N}} (\bK)}{|\Aut(\bG)|}  = \lim_{s\ra\infty} \sum_{\bK\in I_\C} \mu_{g_s,\infty, \mathcal C}(\bK) \frac{ L_{\bG, \underline{V}, \underline{W}, \underline{N}} (\bK)}{|\Aut(\bG)|},$$
and then using Equation~\eqref{justLfromfin}, we have
\begin{equation}\label{Elaststep}\sum_{\bK\in I_\C}   \mu_{\infty,\infty, \mathcal C}(\bK) \frac{L_{\bG, \underline{V}, \underline{W}, \underline{N}} (\bK)}{|\Aut(\bG)|}  = \lim_{s\ra\infty} \lim_{L\ra\infty}  \sum_{\bK\in I_\C} \mu_{g_s,L, \mathcal C}(\bK) \frac{L_{\bG, \underline{V}, \underline{W}, \underline{N}} (\bK)}{|\Aut(\bG)|},\end{equation}
Because \eqref{Elaststep} holds for any subsequence $g_s$ such that $\mu_{g_s,\infty,\mathcal C}$ converges weakly, we have
\begin{equation}\label{contradiction-equation}\sum_{\bK\in I_\C}   \mu_{\infty,\infty, \mathcal C}(\bK) \frac{L_{\bG, \underline{V}, \underline{W}, \underline{N}} (\bK)}{|\Aut(\bG)|}  = \lim_{g\ra\infty} \lim_{L\ra\infty}  \sum_{\bK\in I_\C} \mu_{g,L, \mathcal C}(\bK) \frac{L_{\bG, \underline{V}, \underline{W}, \underline{N}} (\bK)}{|\Aut(\bG)|}.\end{equation}
Indeed, if \eqref{contradiction-equation} is false, we can pass to a subsequence $g_s$ on which the right side either converges to a different value or diverges to $\infty$, then pass to a further subsequence on which $\mu_{g_s,\infty,\mathcal C}$ converges, obtaining a contradiction with \eqref{Elaststep}. This gives \eqref{contradiction-equation}, which together with \eqref{S2} and Proposition~\ref{P:evalsum}
handles the general case.
\end{proof}

\section{Inclusion-exclusion lemma}\label{S:ie}

The goal of this section is to prove Lemma~\ref{L:ie}, the identity we use for inclusion-exclusion. 
We will first need one preliminary result, which is also
used in the proof of Lemma~\ref{L:zerocases}, to settle the non-attainable case of Theorem~\ref{localized-counting}.
When $\bK$ is a $3$-manifold (\Mdcorated/) group and $\rho: \bK \ra \bG$ a surjection, there is a maximal quotient of $\bK$ that sees the {material extensions of $G$}. 

\begin{lemma}\label{maximal-quotient-attainable}
Let $\bK$ be a $3$-manifold (\Mdcorated/) group, $\bG$ a finite \Mdcorated/ group, and $\rho: \bK \ra \bG$ a surjection.
 Let ${\bf Q_\rho}$ be the quotient of $\bK$ by the intersection of the kernels of all surjective lifts of $\rho$ to minimally material extensions of $G$.  Then 
 ${\bf Q_\rho}$ is a finite, attainable, material extension of $\bG$, and any lift of $\rho$ to a material extension factors through $K\ra {Q_\rho}$.  
\end{lemma}

Before giving the proof of Lemma~\ref{maximal-quotient-attainable}, we record a basic fact of group cohomology that we will use repeatedly.
\begin{lemma}\label{lem-even_ei} 
Let $\pi: H\ra G$ be a surjection of groups.  Non-zero morphisms of $G$-groups in $\Hom_G((\ker\pi)^{ab}, V_i)$
correspond exactly to surjections from $H$ to extensions of $G$ by $V_i$ that are compatible with the map to $G$. 
If $S$ is the subset of those morphisms that correspond to trivial extensions (along with the $0$ morphism), then we have an exact sequence
$$
1\ra  H^1(G,V_i) \ra  H^1(H,V_i) \ra S\ra 1.
$$
Also, the kernel of $H^2(G,V_i)\ra H^2(H,V_i)$ is the set of those extensions of $G$ by $V$ that occur as quotients of $H$, compatibly with the map to $G$ (along with $0$).
\end{lemma} 

\begin{proof}
We can use the Lyndon-Hochschild-Serre spectral sequence to compute $H^*(H,V_i)$. 
From the edge maps, we have that $H^1(G,V_i)\ra H^1(H,V_i)$ is an injection whose
cokernel is %a subgroup of $H^1(\ker\pi,V_i)^G=\Hom_G(\ker\pi,V_i)$, in particular
 the kernel of $d_2^{0,1}: \Hom_G(\ker\pi,V_i)\ra H^2(G,V)$.
The map $d_2^{0,1}$ is the transgression \cite[Theorem 2.4.3]{Neukirch2000}, and we can check that
$d_2^{0,1}(\phi)=\phi_*(\alpha)$, where $\alpha\in H^2(G,\ker\pi^{ab})$ is the class of the extension $H$. 
From this it follows that $S=\ker d_2^{0,1}$. 
Further, the edge map gives that the kernel of $H^2(G,V_i)\ra H^2(H,V_i)$ is $\im d_2^{0,1}$, and the second claim follows.
 \end{proof}

\begin{proof}[Proof of Lemma~\ref{maximal-quotient-attainable}]

First, we must check that $Q_\rho$ is finite.  Because $K$ is a 3-manifold group, $K$ is finitely generated, and thus there are finitely many surjections from $K$ to each minimally material extension of $G$. Since there are finitely many minimally material extensions, there are finitely many lifts of $\rho$ to minimally material extensions of $G$. Thus the quotient $Q_\rho$ of $K$ by the intersection of the kernels of these lifts is finite. 

Second, we must check that $Q_\rho$ is a fiber product over $G$ of minimally material extensions.  More generally, one can prove that any subgroup of a fiber product of minimal non-trivial extensions that surjects onto each factor must be a fiber product of a subset of the extensions.  The argument is analogous to that in \cite[Lemma 5.3]{Liu2020}.

Finally, we must check the conditions for attainability.  
Because $\bK$ is isomorphic to the fundamental group of a $3$-manifold,
% and $\rho$ is an \Mdcorated/ map, 
by Theorem~\ref{pi1-properties} we have
that $\bK\ra \bG$ is attainable.  We apply \cref{lem-even_ei} to $K\ra Q_\rho$.  Since $K\ra Q_\rho$ does not lift to $K\ra V_i\rtimes Q_\rho$ for an minimally material extension $V_i \rtimes G$ of $G$, we have that $\dim H^1({Q}_\rho , V_i)  =\dim  H^1({K} , V_i)$ for all such $V_i$ and thus ${\bf Q_\rho}\ra \bG$ is attainable.
\end{proof}

\begin{proof}[Proof of Lemma~\ref{L:ie}] 
We have
\[  \sum_{\bH \in I} \frac{ T_{\bH} }{ \abs{ \Aut(\bH)} } \Surj(\bK,\bH) = \sum_{\bH \in I} \sum_{\phi\in \Surj(\bK,\bH)} \sum_{ P\in\Path(\bH,\bG)}  (-1)^{\abs{P}} \alpha_P \beta_P \frac{1}{  \abs{ \Aut(\bH)} } .\]
Each term of the sum on the right defines a surjection $\rho \colon \bK \to \bG$, the composition of the map $\phi$ with the maps $f_i$ in the path $P$. 
By \cref{maximal-quotient-attainable}, we have that $\phi$ factors through $\mathbf{Q_\rho}$.  
Note that $\bK$ has only finitely many surjections to $\bG$, and the corresponding $\mathbf{Q_\rho}$ have only finitely many quotients.  Thus in the sum on the right, there are only finitely many $\bH$ for which the sum over $\phi,P$ is non-empty.

Now there are two possibilities for a term in the sum on the right. Either $\mathbf{Q_\rho}$ is isomorphic to $\bH_s$, or it is not.  Using the terms where $\mathbf{Q_\rho}$ is not isomorphic to $\bH_s$, we will cancel all the terms where $\mathbf{Q_\rho}$ is isomorphic to $\bH_s$ except for those where $s=0$ and $\mathbf{Q_\rho}\cong\bG$, which will contribute $L_{\bG, \underline{V}, \underline{W}, \underline{N}} (K)$.

Consider a path where $\bH_s$ is not isomorphic to $\mathbf{Q_\rho}$. We can adjust the path by adding the unique member $\bH_{s+1}$ of $I$ isomorphic to $\mathbf{Q_\rho}$. For morphisms, we replace $\phi \colon \bK \to \bH_s$ with a morphism $\phi' \colon \bK \to \bH_{s+1}$ obtained by composing the projection $\bK \to \mathbf{Q_\rho}$ with one of the $\abs{\Aut(\bH_{s+1})}$ isomorphisms $\mathbf{Q_\rho} \to \bH_{s+1}$. We then take $f_{s+1}$ to be the unique surjection $\bH_{s+1} \to \bH_s$ whose composition with $\phi'$ is $\phi$, which exists because the kernel of $\phi'$ is the intersection of the kernels of all surjections from $\bK$ to material extension of $G$, and therefore is contained in the kernel of $\phi$.
Because $\bG$ is attainable by assumption and $\bH_{s+1}$ is attainable by \cref{maximal-quotient-attainable}, we have \[\dim H^1 ( H_{s+1},V_i) \equiv 2 \tau (c_{V_i})\equiv \dim H^1 ( G,V_i) \] and thus their difference cannot be $1$. Hence no additional factors of $\frac{1}{q_i}$ are added to $\beta_P$ by this adjustment.

This adjustment has the effect of raising $\abs{P}$ by $1$, multiplying $\alpha_P$ by $\frac{1}{ \abs{\Aut(\bH_{s})}}$, and fixing $\beta_P$. Thus, the terms corresponding to the $\abs{\Aut(\bH_{s+1})}$ new paths exactly cancel the term corresponding to the original path. Each term where $\bH_s$ is isomorphic to $\mathbf{Q_\rho}$ arises exactly once from this construction, via the truncated path obtained by removing $\bH_s$ and replacing $\phi$ with $f_{s} \circ \phi$, except for the terms where $s=0$.

The remaining terms in the sum are those with $s=0$ and $\mathbf{Q_\rho} \cong \bH_0 = \bG$. Such terms are simply given by \Mdcorated/ maps $\rho  \colon \bK \to \bG$  that induce an isomorphism $\mathbf{Q_\rho} \cong \bG$ and have $(-1)^{\abs{P}} \alpha_P \beta_P = 1 \cdot 1= 1$.  The condition $\mathbf{Q_\rho} \cong \bG$ means that every lift of $\rho \colon K \to G$ to a material extension of $G$ 
%on our list in fact factors through $K \to G$. Such a lift factors through $K \to G$ if and only if it 
fails to be surjective, so the number of surjections $\rho$ with $\mathbf{Q_\rho} \cong \bG$ is  $L_{\bG, \underline{V}, \underline{W}, \underline{N}} (K) $.
\end{proof}

\begin{remark}
By using a parity hypothesis in \cref{L:ie}, we had the flexibility to introduce the $\beta_P$ term into the sum, and indeed the lemma would hold if we replaced $1/q_j$, in the definition of $\beta_P$ with anything else.  The particular choice will only matter in \cref{calculating-th}, where it causes $T_\bH$ to take a smaller value in the affine symplectic case. In \cref{self-dual-w} this smaller term gives a convergent sum, which without the $\beta_P$ factor wouldn't converge.
\end{remark}

\section{Convergence theorem for the moments}\label{S:switch}
\cref{P:decmom} found the limiting moments of our distributions of interest, and in this section, we will show that 
the limiting moments agree with the moments of the limiting distribution (assuming the limiting distribution exists), proving Proposition~\ref{P:robustness-exchange}.  This is a non-trivial analytic question, as limits and infinite sums don't always commute.
%To prove the desired agreement, we will use the dominated convergence theorem (many times). 
 The main challenge is to express our group-theoretic sums in terms of something whose analytic behavior we can control.
 Also, in Section~\ref{S:llim}, we will prove Lemma~\ref{L:llim}.

%   In the next subsection, we will calculate   $\sum_{\bK \in I_{\mathcal C}}  S_{\bH}(\bK)  \mu_{\infty,\infty,\mathcal C}(\bK) $ by showing that the moments %of the limit of a sequence of measures are the limit of the moments.
 
\subsection{Definitions} 
A \emph{$G$-group} $H$ is a group with an action of $G$.  A $G$-group is \emph{simple} if it contains no non-trivial, proper, normal subgroups that are fixed (setwise) by $G$.  A $G$-group is \emph{semisimple} if it is a finite direct product of simple $G$-groups.
A \emph{semisimple} $[G]$-group is a finite direct product of simple $[G]$-groups.  
If $\pi:E \ra G$ is a group homomorphism, we call $\pi$ \emph{semisimple} (resp. \emph{simple}) if $\pi$ is surjective and $\ker(\pi)$ is a semisimple (resp. simple) $[G]$-group (or equivalently, a semisimple (resp. simple) $E$-group).
If $\pi: E \ra G$ is a surjective group homomorphism, we say that a surjective group homomorphism $\phi:E\ra R$ is a \emph{radical} of $\pi$ if  if $\ker \phi$ is the intersection of all maximal proper $E$-normal subgroups of $\ker \pi$.  Note in this case $\pi$ factors through $\phi$ and the resulting map $R\ra G$ is semisimple.  Indeed, in this case every intermediate quotient $E\ra Q\ra G$ such that $Q\ra G$ is semisimple factors through $E\ra R$.
 
\subsection{Limit of moments is moments of limit for (unoriented) groups}
We will first give a result for unoriented groups, which we expect to be useful in other contexts, and then we will show Proposition~\ref{P:robustness-exchange}, for oriented groups, follows from it. 
For any two groups $G,H$, let $S_{G,H}:=|\Surj(G,H)|$.
For a  set $\C$  of finite groups, let $J_\mathcal{C}$ be a set of groups consisting of one  
from each isomorphism class of finite level-$\mathcal{C}$ groups. 
%Consider for each positive integer $n$, non-negative values $p^n_G$ defined for each $G\in I_\mathcal{C}$.  
The main result of this subsection is the following.

\begin{theorem}\label{robustness-exchange-un\Mdcorated/}
For some finite set of finite groups $\C$, for each $K \in J_\C$,  let $p_{K}^n $  be a sequence of nonnegative real numbers such that  $\lim_{n \to \infty} p_{K}^n$ exists. Suppose that, for every $H\in J_\C$ we have
\begin{equation}\label{max-moment}
{\sup_n} \sum_{K\in J_\C} S_{K,H} p_{K}^n <\infty.
\end{equation}
Then for every $H\in J_\C$ we have that 
\begin{equation*}
\lim_{n \to \infty} \sum_{K\in J_\C}  S_{K,H} p_{K}^n= \sum_{K\in J_\C} S_{K,H} \lim_{n \to \infty} p_{K}^n.
\end{equation*}
\end{theorem}

%We consider an expression like
%$$
%\lim_{n\ra\infty} \sum_{G \in J_\C}  S_{G,F}  p^n_{G},
%$$
%and assuming these limits exists, and the limits $\lim_{n\ra\infty}  p^n_{G}$ exist, we would like to move the limit inside the sum.
To prove Theorem~\ref{robustness-exchange-un\Mdcorated/} we need to exchange the sum with the limit.
We will do this by breaking it up into a series of sums and exchanging them with the limit one at a time. Each step will be proven using the 
Fatou-Lebesgue theorem. %, which is a version of the dominated convergence theorem for $\limsup$.
%, as we don't want to assume limits exist in the intermediate steps. 
%Let $Q_\mathcal{C}$ be a set of finite groups consisting of one  from each isomorphism class of a quotient of a finite level-$\mathcal{C}$ group.
The first step in breaking up our sum into a series of sums is the following identity, for any finite group $H$,
 \begin{equation}\label{surjection-radical-splitting}  \sum_{K \in J_\C}  S_{K,H}  p^n_{K} = \sum_{ R \in J_\C} \sum_{ \substack{ a \colon R \ra H\\ \textrm{semisimple}}} \sum_{K \in J_\C} \sum_{ \substack{ b\colon K \ra R \\ b = \operatorname{radical}(a \circ b)}} \frac{p^n_K}{ \abs{\Aut(R)}}\end{equation}

  We will use the following result to check the hypothesis when we use the Fatou-Lebesgue theorem.
   \begin{lemma}\label{dominance} 
 Let $\mathcal{C}$ be a finite set of finite groups.
For all $K \in J_\C$ and $n \in \mathbb N$  let $p_{K}^n \geq 0$  be a real number. Suppose that, for every $H\in J_\C$, Equation \eqref{max-moment} holds.
Then, for every $H\in J_\C$, we have
$$
\sum_{ R \in J_\C} \sum_{ \substack{ a \colon R \ra H\\ \textrm{semisimple}}}  \sup_{n \in \mathbb N} \Bigl( \sum_{K \in J_\C} \sum_{ \substack{ b\colon K \ra R \\ b = \operatorname{radical}(a \circ b)}} \frac{p_{K}^n}{ \abs{\Aut(R)} } \bigr) < \infty.$$
\end{lemma}

\begin{proof} 
Given $H,$ let $(C_1, a_1),\dots (C_m,a_m)$ be pairs of a member $C_i$ of $J_{\C}$ together with a simple morphism $a_i: C_i \to H$. This set is finite by  \cite[Lemmas 6.1, 6.11]{Liu2020}.
For any semisimple morphism $a \colon R \ra H$, 
%we can express the kernel of $a$ as a product of simple $R$-groups, hence 
we can express $a$ as a fiber product of simple morphisms 
%(the morphisms from the quotients by all but one of these simple $R$-groups) 
and thus can write $R = \prod_{i=1}^m C_i^{e_i}$ for some natural numbers $e_i$, with the products taken over $H$.  Call this fiber product $R_{\mathbf e}$ and let $a_{\mathbf e}$ be its projection to $H$. 
%This expression may or may not be unique. Here we use the assumption that the decoration is trivial to make this well-defined, otherwise we would  have to %make a choice of decoration on the fiber product. 
 %Summing \eqref{max-moment} over various values of $F$, we obtain
%
%$$\sup_{n \in \mathbb N}  \sum_{ \mathbf d\in \{0,1,2\}^m} \sum_{s \in  \mathcal F ( R_{\mathbf d})}   \sum_{G \in I_\C} S_{G, (R_{\mathbf d}, s)}  p_{G,%\C}^n < \infty.$$
Let $$C = \sup_{n \in \mathbb N}  \sum_{ \mathbf d\in \{0,1,2\}^m}   \sum_{G \in J_\C} S_{G, R_{\mathbf d}}   {p_{G}^n}.$$

Given a group $K \in J_\C$ and a map $b \colon  K \ra R_{\mathbf e}$ such that $b = \operatorname{radical} ( a_{ \mathbf e} \circ b)$, we obtain $\prod_{i=1}^m {\binom{e_i}{d_i}}$ distinct homomorphisms $K \to R_{\mathbf d}$, by composing $b$ with the projections of $R_{\mathbf e} \ra R_{\mathbf d}$ onto $d_i$ of the $e_i$ factors of type $C_i$, for all $i$.
 If two surjections $b,b'$ give the same map $c: K \ra R_{\mathbf d}$ then $b,b'$ must be equal up to multiplication with an element of $\Aut(R_{\mathbf e})$, since we can recover $\ker(b)$ by composing $c$ with $a_{\mathbf d}$ and taking the radical. This implies, for fixed ${\mathbf e},{\mathbf d}$,
$$ \sum_{K \in J_\C} \sum_{ \substack{ b\colon K \ra R_{\mathbf e} \\ b = \operatorname{radical}(a_{\mathbf e}  \circ b)}} \frac{p_{K}^n}{ \abs{\Aut(R_{\mathbf e} )} }  \cdot \prod_{i=1}^m \binom{e_i}{d_i}  \leq     \sum_{K \in  J_\C} S_{K, R_{\mathbf d}}  {p_{K}^n} .$$

Summing over $\mathbf d$, we obtain
$$ \sum_{K \in J_\C} \sum_{ \substack{ b\colon K \ra R_{\mathbf e} \\ b = \operatorname{radical}(a_{\mathbf e}  \circ b)}} \frac{p_{K}^n}{ \abs{\Aut(R_{\mathbf e} )} }  \cdot \prod_{i=1}^m  \sum_{d=0}^2 \binom{e_i}{d}  \leq   \sum_{ \mathbf d\in \{0,1,2\}^m}    \sum_{K \in J_\C} S_{K, R_{\mathbf d}}  {p_{K}^n}  \leq C. $$
Thus
$$ \sum_{K \in J_\C} \sum_{ \substack{ b\colon K \ra R_{\mathbf e} \\ b = \operatorname{radical}(a_{\mathbf e}  \circ b)}} \frac{p_{K}^n}{ \abs{\Aut(R_{\mathbf e} )} }  \leq \frac{C}{ \prod_{i=1}^m  \sum_{d=0}^2 \binom{e_i}{d} }.$$
Now since every pair $R$ and $a \colon R \ra H$ semisimple is isomorphic to $ (R_{\mathbf e}, a_{\mathbf e})$ for some $\bf e$, we have
\begin{align*}\sum_{ R \in J_\C} \sum_{ \substack{ a \colon R \ra H\\ \textrm{semisimple}}}  \sup_{n \in \mathbb N} \Bigl( \sum_{K \in J_\C} \sum_{ \substack{ b\colon K \ra R \\ b = \operatorname{radical}(a \circ b)}} \frac{p_{K}^n}{ \abs{\Aut(R)} } \bigr) 
& \leq \sum_{ \mathbf e \in \mathbb N^m}  \sup_{n \in \mathbb N}  \Bigl( \sum_{K \in J_\C} \sum_{ \substack{ b\colon K \ra R_{\mathbf e} \\ b = \operatorname{radical}(a_{\mathbf e} \circ b)}} \frac{p_{K}^n}{ \abs{\Aut(R_{ \mathbf e} )} } \bigr) \\
&\leq \sum_{ \mathbf e \in \mathbb N^m} \frac{C}{ \prod_{i=1}^m  \sum_{d=0}^2 \binom{e_i}{d} } = C \prod_{i=1}^m \sum_{e=0}^{\infty} \frac{1}{ \sum_{d=0}^2 \binom{e}{d}}  < \infty,
\end{align*}
as desired.\end{proof}

Given a sequence of maps $K\stackrel{b}{\ra} R_k \stackrel{a_k}{\ra} R_{k-1} \cdots \stackrel{a_1}{\ra} R_0$, we say $b,a_k,\dots,a_1$ is a \emph{radical sequence} if
for all $1\leq i\leq k$, we have that $a_{i+1} \circ \dots \circ a_k \circ b: K\ra R_i$ is a radical of $ a_{i} \circ a_{i+1} \circ \dots \circ a_k \circ b: K\ra R_{i-1}.$
Lemma~\ref{dominance} will allow us to inductively prove the following result using the  Fatou-Lebesgue theorem, moving the limit further and further past the sum as $k$ increases.
\begin{lemma}\label{radical-exchange-limit}  
 Let $\mathcal{C}$ be a finite set of finite groups.
For all $K \in J_\C$ and $n \in \mathbb N$,  let $p_{K}^n \geq 0$  be a real number. 
 Suppose that Equation~\eqref{max-moment} holds for every finite group $H$.
Then for every $H\in I_\C$ and all natural numbers $k$ we have
$$ \limsup_{n \to \infty} \sum_{K\in J_\C} S_{K,H} p_{K}^n \leq  \sum_{\substack{ R_0,R_1, \dots, R_k \in J_\C \\ R_0=H}} \sum_{ \substack{a_i \colon R_i \to R_{i-1} \\ i=1,\dots k \\ \textrm{semisimple}} }  \limsup_{n \to \infty} \sum_{ K \in J_\C} \sum_{ \substack { b \colon K \ra R_k \\  
b,a_k,\dots,a_1 \textrm{ rad. seq.} 
 }} \frac{ p_{K}^n }{\prod_{i=1}^k \abs{ \Aut(R_i) }} .$$

\end{lemma} 

\begin{proof} The proof is by induction on $k$.
The case $k=0$ is trivial.  Now we assume the lemma is true for $k$.
We have
$$\sum_{ K \in J_\C} \sum_{ \substack { b \colon K \ra R_k \\   b,a_k,\dots,a_1 \textrm{ rad. seq.} }} \frac{ p_{K}^n }{  \prod_{i=1}^k \abs{ \Aut(R_i) }}=  \sum_{R_{k+1} \in J_\C} \sum_{ \substack{ a_{k+1} \colon R_{k+1} \to R_k \\ \textrm{semisimple}}}  \sum_{ K \in J_\C} \sum_{ \substack { b \colon K \ra R_{k+1} \\  b,a_{k+1},\dots,a_1 \textrm{ rad. seq.} }} \frac{ p_{K}^n }{  \prod_{i=1}^{k+1} \abs{ \Aut(R_i) }} .$$
To complete the induction, it suffices to show that 
$$
\limsup_{n \to \infty}\sum_{R_{k+1},a_{k+1}} \sum_{ K,b }  \frac{ p_{K}^n }{  \prod_{i=1}^{k+1} \abs{ \Aut(R_i) }}
\leq \sum_{R_{k+1},a_{k+1}} \limsup_{n \to \infty} \sum_{ K,b }  \frac{ p_{K}^n }{  \prod_{i=1}^{k+1} \abs{ \Aut(R_i) }},
$$
(where the sums are over the same sets as in the previous equation).
This follows from the Fatou-Lebesgue theorem, using \cref{dominance} with $H=R_k$ to check the hypothesis.  (Our sum in
$b$ is over a smaller set than in \cref{dominance} since we require $b,a_{k+1},\dots,a_1$ to be a radical sequence and not just $b,a_{k+1}$, but this only improves the upper bound.)
\end{proof}

Finally, we will now show that, given $\mathcal{C}$, for some $k$ eventually the inner sums over $K,b$ on the right-hand side of Lemma~\ref{radical-exchange-limit}
become trivial and we have fully exchanged the sum and limsup.
%Recall that given two groups $A,B$, a \emph{subdirect product} of $A$ and $B$ is a subgroup of $A\times B$ that projects surjectively onto each factor.
\begin{lemma}\label{L:subdirectss}
 Let $\mathcal{C}$ be a finite set of finite groups.  If $G_1,G_2$ are finite groups such that $G_i\ra G_i^{\C}$ are semisimple for $i=1,2$, then if $S$ is a subdirect product of $G_1,G_2$, then $S\ra S^{\C}$ is semisimple.
\end{lemma}
\begin{proof}
Let $K_i=\ker (G_i\ra G_i^{\C})$.  We have a commutative diagram
$$
\begin{tikzcd} 1 \arrow[r] \arrow[d] & K \arrow[r] \arrow[d]& S \arrow[r]\arrow[d] & S^{\C}  \arrow[d] &  \\ 
1 \arrow[r] & K_1\times K_2 \arrow[r] & G_1\times G_2 \arrow[r] & G_1^{\C} \times G_2^{\C}.  \end{tikzcd}
$$
Since $S^{\C}\ra G_i^{\C}$ is surjective, through that morphism, $K_i$ is also a semisimple $[S^{\C}]$-group, and thus $K_i$ is also a semisimple $S$-group (under conjugation by elements of $S$).  We have that $K$ is a $S$-subgroup of the semisimple $S$-group $K_1\times K_2$, 
and since $S\ra G_i$ is surjective, this means that the projection of $K$ to each $K_i$ is normal in $K_i$.  It follows that $K$ is a semisimple $S$-group (e.g. see \cite[Lemma 5.3]{Liu2020}), as desired.  
\end{proof}

\begin{lemma}\label{L:quotientss}
 Let $\mathcal{C}$ be a finite set of finite groups.  Let $G$ be a finite group such that $G\ra G^{\C}$ is semisimple and let $Q$ be a quotient of $G$. Then $Q \to Q^{\C}$ is semisimple.\end{lemma}
 
 \begin{proof} Let $K = \ker (G \ra G^\C)$ and $N= \ker (G \ra Q)$. Then $G/(KN)$ is a quotient of $G^{\C}$ and hence is level-$\C$, and is also a quotient of $Q$, so $(G/KN)$ is a quotient of $Q^{\C}$. Furthermore, since $KN\subset \ker(G\ra Q^{\C})$, we have that $Q^{\C}=G/(KN)$ and 
the kernel of $Q \ra Q^{\C}$ is $KN/N \cong K / (K \cap N)$.  Since $K$ is a semisimple $G$-group, $K/ (K \cap N)$ is a semisimple $G$-group, thus, because the action of $G$ on it factors through $Q$, a semisimple $Q$-group.
 \end{proof}

\begin{lemma}\label{one-lower-semisimple}
 Let $\mathcal{C}$ be a finite set of finite groups, and let $\C'$ be a set of groups that contains all proper quotients of groups in $\C$.
 Then for a finite group $G\in J_\C$, we have that $G\ra G^{\C'}$ is semisimple.
\end{lemma}
\begin{proof}
For $G\in \C$, if $G$ is a simple group, then clearly $G\ra G^{\C'}$ is semisimple.  Otherwise, let $N$ be a minimal non-trivial normal subgroup of $G$.
We have $G\ra G^{\C'} \ra G/N$, and so $G^{\C'}$ is either $G$ or $G/N$ and in either case $G\ra G^{\C'}$ is semisimple.
Since any group of level-$\C$ is contained in the closure of $\C$ under taking subdirect products and quotients, the conclusion follows from Lemma~\ref{L:subdirectss} and Lemma~\ref{L:quotientss}.
\end{proof}

\begin{lemma}\label{L:radfactor}
Let $G$ be a finite group and $N_1, N_2\subset N_3$ be normal subgroups of $G$ such that $G/N_2\ra G/N_3$ is semisimple.
Then $G/(N_2\cap N_1)\ra G/N_1$ is semisimple, and so if $G\ra G/N_4$ is the radical of $G/N_1$, then 
$G\ra G/(N_2\cap N_1)$ and $G\ra G/N_2$ factor though $G\ra G/N_4$.
\end{lemma}
\begin{proof}
Since  $G/N_2\ra G/N_3$ is semisimple, 
$N_3/N_2$ is a semisimple $G$-group.  Since $N_1/(N_2 \cap N_1)$ is a normal, $G$-invariant subgroup of 
$N_3/N_2$, we have that $N_1/(N_2 \cap N_1)$ is a semisimple $G$-group and hence $G/(N_2\cap N_1)\ra G/N_1$ is semisimple.
\end{proof}

\begin{prop}\label{radical-tower-bound} 
Let $\mathcal{C}$ be a set of finite groups each of order at most $k$.  Then if $K\in I_C$ and $b:K\ra R_k$ and $a_{i}: R_i \ra R_{i-1}$ (for $1\leq i\leq k$)
are group homomorphisms such that $b,a_k,\dots,a_1$ is a radical sequence, then $b$ is an isomorphism. 
\end{prop}

\begin{proof}
Let $\C_i$ be the set of all quotients of groups in $\mathcal{C}$ of order at most $i$.    Let $K_i$ be the image of $K$ in $K^{ \C_i} \times R_0$.
We apply Lemma~\ref{L:radfactor} with $G=K$ and $G/N_2=K^{\C_{i+1}}$ and  $G/N_3=K^{\C_{i}}$ and $G/N_1=K_{i}$, using
\cref{one-lower-semisimple} to see that $K^{\C_{i+1}}\ra K^{\C_{i}}$ is semisimple.
Then $G/(N_2\cap N_1)=K_{i+1}$ and so
$K_{i+1}\ra K_{i}$ is semisimple by the first part of Lemma~\ref{L:radfactor}. 

We will show by induction that $K\ra K_i$ factors through 
$a_{i+1} \circ \dots a_k \circ b : K \ra R_i.$  When $i=1$, we have $K_1=R_0,$ and this is automatic.
For the induction step, assuming $K \ra K_i$ factors through $K \ra R_i$, we apply the second part of Lemma~\ref{L:radfactor}
with $G=K$, $G/N_2=K_{i+1}$, $G/N_3=K_i$, and $G/N_1=K_i$, and conclude that since $K\ra R_{i+1}$ is the radical of $K\ra R_i$, it factors through
$K_{i+1}$.

Thus we have that $K\ra K_k=K$ factors through $b: K\ra R_k$, and we conclude the lemma. 
\end{proof}

Putting this all together  we can prove \cref{robustness-exchange-un\Mdcorated/}.

\begin{proof}[Proof of \cref{robustness-exchange-un\Mdcorated/}] 
Let $k$ be the maximal order of a group in $\mathcal{C}$.
By Proposition \ref{radical-tower-bound}, the sums over $G,b$ on the right-hand side of Lemma \ref{radical-exchange-limit} are finite, and can be exchanged with the limit and so we have
\begin{align*}
 \limsup_{n \to \infty} \sum_{K\in J_\C} S_{K,H} p_{K}^n &\leq  \sum_{\substack{ R_0,R_1, \dots, R_k \in J_\C \\ R_0=H}} \sum_{ \substack{a_i \colon R_i \to R_{i-1} \\ i=1,\dots k \\ \textrm{semisimple}} }   \sum_{ K \in J_\C} \sum_{ \substack { b \colon K \ra R_k \\  
b,a_k,\dots,a_1 \textrm{ rad. seq.} 
 }}  \frac{ \lim_{n \to \infty} p_{K}^n }{\prod_{i=1}^k \abs{ \Aut(R_i) }}\\
&=  \sum_{K\in J_\C} S_{K,H}  \lim_{n \to \infty} p_{K}^n.
\end{align*}
Fatou's lemma gives 
\begin{align*}
 \liminf_{n \to \infty} \sum_{K\in J_\C} S_{K,H} p_{K}^n &\geq \sum_{K\in J_\C} S_{K,H}  \lim_{n \to \infty} p_{K}^n,
\end{align*}
and the theorem follows.
\end{proof}

\subsection{Limit of moments is moments of limit for oriented groups}
Now we will see that \cref{P:robustness-exchange}, i.e. a version of Theorem~\ref{robustness-exchange-un\Mdcorated/} for oriented groups, follows directly from Theorem~\ref{robustness-exchange-un\Mdcorated/} (as Fatou's lemma gives one inequality, and we have equality after a finite sum over orientations).  For oriented groups $\bK,\bH$ let $S_{\bK,\bH}$ denote the number of oriented surjections $\Surj(\bK,\bH)$.

\begin{proof}[Proof of \cref{P:robustness-exchange}]
 For $K \in J_{\C}$ and $s \in H_3(K,\Z)$, let $(K,s)$ denote the oriented group, and let
  $$p_{K}^n = \sum_{ s \in H_3(K,\Z)} \frac{|\Aut(K,s)|}{|\Aut(K)|} p_{ (K,s)}^n.$$
%Note that the automorphism factors are chosen so that $p_G^n$ is the sum of $p_{\bH}^n$ over each isomorphism class of decorated group $\bH$ with %underlying group $G$.
Then for $H\in J_{\C}$,  
\begin{align*} \sum_{K \in J_\C}  S_{K,H} p_{K}^n  &=  \sum_{K \in J_\C}   \sum_{ s \in H_3(K,\Z)} S_{K,H}  \frac{|\Aut(K,s)|}{|\Aut(K)|}
 p_{(K,s) }^n \\
  &=  \sum_{K \in J_\C}   \sum_{ s \in H_3(K,\Z)} \sum_{ t \in H_3(H,\Z)} S_{(K,s),(H,t)}  \frac{|\Aut(K,s)|}{|\Aut(K)|}
 p_{(K,s) }^n \\ 
  & = \sum_{ t \in H_3(H,\Z)} \sum_{\bK \in I_\C} S_{\bK, (H,t)} p_{\bK}^n   
\end{align*} 
  because, given $s$, each surjection $\pi :K \to K$ is a surjection $(K,s)\to (H,t)$ of \Mdcorated/ groups for exactly one $t$,
  and by the orbit-stabilizer theorem, each isomorphism class of $(K,s)$ appears $\frac{|\Aut(K)|}{|\Aut(K,s)|}$ times in the sum.
In particular, the hypothesis \eqref{max-moment} holds by summing the hypothesis of \cref{P:robustness-exchange} over $t$ and we have
\begin{align*}
\lim_{n\ra\infty} \sum_{K \in J_\C}  S_{K,H} p_{K}^n = \sum_{K \in J_\C}  \lim_{n\ra\infty} S_{K,H} p_{K}^n
\end{align*} 
Also, 
\[ \sum_{ t \in H_3(H,\Z)}  \liminf_{n \to \infty}  \sum_{\bK \in I_\C} S_{\bK, (H,t)} p_{\bK}^n  \leq 
\lim_{n \to \infty}  \sum_{ t \in H_3(H,\Z)} \sum_{\bK \in I_\C} S_{\bK, (H,t)} p_{\bK}^n   =
\lim_{n\to\infty} \sum_{K \in J_\C}  S_{K,H} p_{K}^n.
\]
So
\begin{align}\label{W1}
\sum_{ t \in H_3(H,\Z)}  \liminf_{n \to \infty}  \sum_{\bK \in I_\C} S_{\bK, (H,t)} p_{\bK}^n  \leq & 
\sum_{K \in J_\C}  \lim_{n\ra\infty} S_{K,H} p_{K}^n\\
=& \sum_{K \in J_\C} S_{K,H} \sum_{ s \in H_3(K,\Z)}
\frac{|\Aut(K,s)|}{|\Aut(K)|} \lim_{n\to\infty} p_{ (K,s)}^n\notag\\
=&
\sum_{t \in H_3(H,\Z)} \sum_{\bK \in I_\C} S_{\bK, (H,t)}  \lim_{n \to \infty} p_{\bK}^n\notag.
\end{align}
  By Fatou's lemma, we have, for each $t \in H_3(H,\Z)$
\begin{align}\label{W2}
  \liminf_{n \to \infty}  \sum_{\bK \in I_\C} S_{\bK, (H,t)} p_{\bK}^n  \geq & 
 \sum_{\bK \in I_\C} S_{\bK, (H,t)}  \lim_{n \to \infty} p_{\bK}^n.
\end{align}  
Since the sum over $t$ of the inequalities in \eqref{W2} is the opposite of the inequality in \eqref{W1}, all of these inequalities must be equalities and we have,
for each $t \in H_3(H,\Z)$,
\begin{align*}
  \liminf_{n \to \infty}  \sum_{\bK \in I_\C} S_{\bK, (H,t)} p_{\bK}^n =
 \sum_{\bK \in I_\C} S_{\bK, (H,t)}  \lim_{n \to \infty} p_{\bK}^n.
 \end{align*}  
 Since the same statement holds for any subsequence of $n$, the proposition follows.
\end{proof}

 \subsection{Convergence in $L$: Proof of \cref{L:llim}}\label{S:llim}

\begin{proof}[Proof of \cref{L:llim}]
Let $K_g$ be the kernel of the map $\pi_1(\Sigma_g)\ra \pi_1(H_g)$ .
We have that $\pi_1(M_{g,L})$ is the quotient of $\pi_1(\Sigma_g)$ by $K_g$ and $\sigma_{g,L}(K)$, where $\sigma_{g,L}$ is the random element of the mapping class group that we used to define $M_{g,L}$.
Note that the mapping class group of $\Sigma_g$ acts on $\pi_1^{\mathcal{C}}(\Sigma_g)$. Let $\bar{K}_g$ denote the image of $K_g$ in $\pi_1^{\mathcal{C}}(\Sigma_g)$.
From the definition of level-$\mathcal{C}$ completion, we have
$$\pi_1^{\mathcal{C}}(M_{g,L})=\left(  \pi_1^{\mathcal{C}}(\Sigma_g)/( \bar{K_g}, \sigma_{g,L}(\bar{K_g})) \right)^{\mathcal{C}}.$$  

Let us first check that the limit of the probability that $\pi_1^{\mathcal{C}}(M_{g,L})$ is isomorphic to $G$ as an unoriented group exists. To do this, we observe that $\pi_1^{\mathcal{C}}(\Sigma_g)$ is finite and the action of the mapping class group on this group factors through a finite group. By the above equation, the isomorphism class of $\pi_1^{\mathcal{C}}(M_{g,L})$ depends only on the image of $\sigma_{g,L}$ in this finite group. That image equidistributes by the Perron-Frobenius theorem, showing that a limiting probability exists.

We now consider oriented groups. Let $H_{g,\mathcal C} $ be the subgroup of the mapping class group that fixes every element of $\pi_1^{\mathcal{C}}(\Sigma_g)$. By \cref{cobordism-class-is-homomorphism}, there is a homomorphism from $H_{g, \mathcal C}$ to the bordism group of $B \pi_1^{\mathcal{C}}(\Sigma_g)$ that sends a mapping class $\sigma$ to the bordism class of the associated 3-manifold. Let $J_{g,\mathcal C}$ be the kernel of this homomorphism.

We claim that the isomorphism class of $\pi_1^{\mathcal{C}}(M_{g,L})$, together with its orientation, depends only on $\sigma_{g, L}$ modulo $J_{g,\mathcal C}$. Having checked this, the Perron-Frobenius theorem will again imply that a limiting measure exists (with equal mass placed on the isomorphism class arising from each coset of $J_{g,\mathcal C}$).

To check this claim, let $\sigma_1$ be any mapping class and let $\sigma_2$ be a mapping class in $J_{g,\mathcal C}$. The identity $$\pi_1^{\mathcal{C}}(M_{\sigma })=\left(  \pi_1^{\mathcal{C}}(\Sigma_g)/( \bar{K_g}, \sigma(\bar{K_g})) \right)^{\mathcal{C}} ,$$ together with the fact that $\sigma_2 $ acts trivially on $\pi_1^{\mathcal C}(\Sigma_g)$,  gives an isomorphism between  $\pi_1^{\mathcal C} (M_{\sigma_1})$ and $\pi_2^{\mathcal C}(M_{\sigma_1\sigma_2})$, showing that both are isomorphic to $\left(  \pi_1^{\mathcal{C}}(\Sigma_g)/( \bar{K_g}, \sigma_1(\bar{K_g})) \right)^{\mathcal{C}} $. Call this group $Q$. It remains to check that this isomorphism preserves the orientation. Since the orientation of a quotient $Q$ of the fundamental group of a 3-manifold is determined by the bordism class of that manifold in $BQ$, it suffices to check that $[M_{\sigma_1}] = [M_{\sigma_1 \sigma_2}]$ in the third bordism group of $BQ$. Now we apply \cref{cobordism-class-additive} to the group $Q$ and the homomorphism $f \colon \pi_1(\Sigma_g) \to \pi_1^{\mathcal C}(\Sigma_g) \to Q$. By definition of $Q$, this factors through $\pi_1(H_g)$, as does its pullback under $\sigma_1$, and every homomorphism $\pi_1(\Sigma_g) \to \pi_1^{\mathcal C}(\Sigma_g) \to Q$ is preserved by $\sigma_2$, so  \cref{cobordism-class-additive} implies that \[[ M_{\sigma_1 \sigma_2} ]= [M_{\sigma_1}]+ [M_{\sigma_2}].\]

By construction of $J_{g,\mathcal C}$, the class of $M_{\sigma_2}$ in the bordism group $B\pi_1^{\mathcal C}(\Sigma_g)$ vanishes. Since $Q$ is a quotient of $\pi_1^{\mathcal C}(\Sigma_g)$, the map to $BQ$ factors through the map to $B\pi_1^{\mathcal C}(\Sigma_g)$, so the class $[M_{\sigma_2}]$ in the bordism group of $BQ$ vanishes. This proves that $[M_{\sigma_1}] = [M_{\sigma_1 \sigma_2}]$, as desired.
\end{proof}

\section{Evaluation of the main group-theoretic sum}\label{S:evalsum}
The goal of this section is to prove Proposition~\ref{P:evalsum}.
We do this by partially evaluating the $T_{\bH}$, using detailed spectral sequence analysis to express the remaining sum as a $q$-series, and then applying  $q$-series identities.
Theorem~\ref{pi1-properties} places certain restrictions on the fundamental group of a $3$-manifold, 
and first we will see that we can avoid certain groups in our analysis, motivating the hypothesis of Proposition~\ref{P:evalsum}.

\begin{lemma} \label{lem:to_alpha_zero}
For any $i$, if $V_i^\vee \cong V_j$ for some $j$ and $W_i^\tau \neq 0$, then  \[ L_{\bG, \underline{V}, \underline{W}, \underline{N}} (
{\bf \pi_1 (M)}
 ) =0 \] for any $3$-manifold $M$. \end{lemma}

\begin{proof}
Consider a surjection $f: {\bf \pi_1(M)}\ra \bG$.   Fix  a nonzero class $\alpha \in W_i^\tau \subseteq H^2( G, V_i)$. 
By \cref{lem-even_ei}, if $f^* \alpha  \in H^2( \pi_1(M), V_i)$ vanishes, then $f$ lifts to an extension of $G$ corresponding to $\alpha$, which is minimally material. 

If $f^* \alpha \in H^2(\pi_1(M), V_i) \neq 0$, by \cref{pi1-properties}(2), there exists $\beta \in H^1( \pi_1(M), V_j)$ such that $\tau ( f^*\alpha \cup \beta) \neq 0$.  Cochains representing $H^1( \pi_1(M),V_j)$ exactly give splittings of $ V_j\rtimes \pi_1(M) \ra \pi_1(M)$, so $\beta$ gives a splitting, which composes with $f$ to give $f':\pi_1(M)\ra V_j\rtimes G$.  By the irreducibility of $V_j$, the image of $f'$ is either surjective or isomorphic to $G$.  If $\operatorname{im} f'$ is isomorphic to $G$,
that implies that $\beta$ came from a splitting given by a $\beta' \in H^1( G, V_j)$.  
 Then, because $f$ is a map of oriented groups, \[ 0 \neq \int_M ( f^*\alpha \cup \beta) = \int_M ( f^*\alpha \cup f^* \beta') = \tau(\alpha \cup \beta') =0 \] because $\alpha \in W_i^\tau$, giving a contradiction.
So, we conclude that $\pi_1(M)\ra V_j\rtimes G$ is a surjection. 

Thus, in either case, the surjection $f$ lifts to a surjection to a minimally material extension and thus is not counted in $L_{\bG, \underline{V}, \underline{W}, \underline{N}}$.
\end{proof} 

\begin{proof}[Proof of Lemma \ref{L:zerocases}] The case that for some $i,j$ we have  $V_i^\vee \cong V_j$ and $W_i^\tau \neq 0$ is Lemma \ref{lem:to_alpha_zero}. In the case that $\bG$ is not an attainable $\bG$-extension, Lemma \ref{maximal-quotient-attainable} gives for any surjection $f\colon {\mathbf \pi_1(M)}\to \bG$ the existence of a quotient ${\mathbf Q}_\rho$ of ${\mathbf \pi_1(M)}$ that is a finite attainable material extension of $\bG$. By definition of material, ${Q}_\rho$ is a fiber product of minimally material extensions of $G$, each a quotient of $\pi_1(M)$. Since $\bG$ is not attainable but ${\mathbf Q}_\rho$ is, ${\mathbf Q}_\rho\not\cong \bG$, thus this set of minimally material extensions is nonempty, hence there is a minimally material extension of $G$ that $f$ lifts to, therefore $f$ does not contribute to $L_{\bG, \underline{V}, \underline{W}, \underline{N}}$. \end{proof}

Next we give a lemma about when a surjection of groups can lift to an oriented map.

\begin{lemma}\label{lem-tau}
%Consider the spectral sequence above when $F=V_i^{e_i}$ (or $V_i^{e_i}\times (V_{i'})^{e_{i'}}$ for $V_i$ and $V_{i'}$ dual) and the extension class of $H$ %is $\alpha\in H^2(G,F)$.  
Let $1\ra F \ra H \stackrel{\pi}{\ra} G\ra 1$ be an extension of groups, where $\bG$ is oriented with $\tau:H^3(G,\QZ)\ra \QZ$ corresponding to the orientation.  
We have $\tau(\ker \pi^*)=0$, i.e. there is an orientation on $H$ compatible with $\pi$, 
 if and only if $\tau \circ d_2^{1,1}=0$ and $\tau \circ d_3^{0,2}=0$, where the $d_r^{p,q}$ are the differentials in the 
Lyndon-Hochschild-Serre spectral sequence to compute $H^3(H,\QZ)$ (from $F$ and $G$).
Further, we have $\tau \circ d_2^{1,1}=0$ if and only if, for $\alpha \in H^2(G,F^{ab})$ the extension class of $H$, we have
$\tau(\alpha\cup \beta)=0$ for all $\beta\in H^1(G, \Hom(F,\QZ))$.
%$\alpha\in (W_i^{\tau})^{e_i}$ ($\alpha\in (W_i^{\tau})^{e_i}\times (W_{i'}^{\tau})^{e_{i'}}$, respectively).
\end{lemma}  

\begin{proof}
The first claim follows from the edge map of the spectral sequence, and the second because $d_2^{1,1}: H^1(G,F^\vee)\ra H^3(G,\QZ)$ is (up to sign)
 the cup product with  $\alpha$ 
  \cite[Theorem 2.4.4]{Neukirch2000}.
\end{proof}

In the case  $W_i^\tau=0$, we can deduce a useful consequence.
\begin{cor}\label{C:dimH1}
 If $\pi: \bH\ra \bG$ is a material \Mdcorated/ surjection and %$\bH$ is attainable  and 
  for some $i$ we have $W_i^\tau=0$, then  
$\dim H^1(H,V_i)=\dim H^1(G,V_i)+\dim \Hom_G (\ker \pi, V_i)$.  
\end{cor}

\begin{proof}
Let $F=\ker\pi$ and $\alpha\in H^2(G,F^{ab})$ be the extension class of $H$.
Since $W_i^\tau=0$ and $\pi$ is a product of extensions with classes in the $W_i$, it follows 
from \cref{lem-tau}
that %$\alpha=0$.
the image of $\alpha$ in $H^2(G,V_i^{e_i})$ is $0$.  % (where ${e_i}$ is the multiplicity of $V_i$ in $F$).
Then in \cref{lem-even_ei} we have $S=\Hom_G (\ker \pi, V_i)$.
%
% If 
% $\pi: H\ra G$ is a material  surjection with extension class $\alpha\in H^2(G,(\ker \pi)^{ab})$
%having trivial image in    $H^2(G,V_i^{e_i})$ (where ${e_i}$ is the multiplicity of $V_i$ in $\ker \pi$),
%  then
%$\dim H^1(H,V_i)=\dim H^1(G,V_i)+\dim \Hom_G (\ker \pi, V_i)$.  
%Thus if $V_i$ is an $A$-symplectic representation of $G$ over $\kappa_i$ and $\bH$ and $\bG$ are attainable, then
%$ \dim_{\kappa_i}\Hom_G (\ker \pi, V_i)$ is even.
%
%
%In particular, under our assumption,  we have $d_2^{0,1}=0$ and the lemma follows.
\end{proof}

Next, we will partially evaluate $T_{\bH}$.
To do this, first define, for  $\pi \colon \bH \to \bG$ a material $\bG$-extension,  \[ T_\pi   : =\sum_{ \substack{ P\in\Path(\bH,\bG) \\ \textrm{composite} ( P) = \pi }}  (-1)^{\abs{P}} \alpha_P \beta_P.\]
(The composite of a path of length $0$ is the identity map.)
Furthermore, for such $\pi$, we can write the $[G]$-group $\ker \pi$ as a product $\prod_{i=1}^r V_i^{e_i} \times \prod_{i=1}^s N_i^{f_i}$ for tuples $\underline{e}, \underline{f}$. 
 In this setting, we say that $\pi$ has \emph{type} $(\underline{e},\underline{f})$ (which in particular implies $\pi$ is material).  It will turn out that $T_\pi$ depends only on the type of $\pi$.

 For any $i$ from $1$ to $r$, if $V_i$ is not A-symplectic, define $Q_i(e_i)$ to be $q_i^{ \binom{e_i}{  2}} $.
 
 If $V_i$ is A-symplectic, we define $ Q_i (e_i)$ to be $ q_i^{ \binom{e_i}{2}- e_i} = q_i^{ \frac{e_i (e_i-3)}{2}}$.

\begin{lemma}\label{calculating-th} Assume for each $i$ such that $V_i$ is A-symplectic that $W_i^\tau=0$.

Then, for a material non-trivial $\bG$-extension $\pi$ of type $(\underline{e}, \underline{f})$, we have  \[ T_\pi = \frac{1}{ \abs{\Aut(\bG)}} (-1)^{ \sum_i e_i + \sum_i f_i} \prod_i Q_i (e_i).\]

 \end{lemma}
  
 \begin{proof} A path $P \in \Path(\bH, \bG)$ of length $\pl$ with composite $\pi$ gives us a sequence of $K_i = \ker ( H \to H_i)$, where the $K_i$ are normal subgroups of $H$. The isomorphism type of each $H_i$ is determined by the $K_i$, but given the $K_i$, there are $\prod_{i=0}^{\pl-1} \abs{\Aut(\bH_i)}$ choices of path with these $H_i$ that give the appropriate kernels in $H$, and exactly $\prod_{i=1}^{\pl-1} \abs{ \Aut (\bH_i)}$ of these choices have composite $\pi$.
 
 Furthermore, by \cref{C:dimH1}, the factor $\frac{1}{q_i}$ appears in $\beta_P$, if and only if, for some $j$, the multiplicity of $V_i$ in $K_j$ is $e_i-1$. Thus letting \[ \beta_{K_0,\dots, K_\pl} = \prod_{\substack{ i \\  V_i \textrm{ A-symplectic}\\ \operatorname{mult}_{V_i} (K_j) = \operatorname{mult}_{V_i} (K_0) -1 \textrm{ for some } j}} \frac{1}{q_i} ,\] we have
  \begin{equation}\label{calculating-th-reduction} \abs{\Aut(\bG)} \sum_{ \substack{ P\in\Path(\bH,\bG) \\ \textrm{composite} ( P) = \pi }}  (-1)^{\abs{P}} \alpha_P \beta_P   = \sum_{ 1 = K_\pl \subsetneq K_{\pl-1} \subsetneq  \dots \subsetneq K_0 = \ker \pi} (-1)^\pl   \beta_{K_0,\dots,K_s} 
 \end{equation}
   where the latter sum is over chains of normal $H$-subgroups $K_i$ of $\ker \pi$.
   We prove, by induction on the number of simple factors of a semisimple $H$-group $F$ that 
  \begin{equation}\label{calculating-th-induction} \sum_{ 1 = K_\pl \subsetneq K_{\pl-1} \subsetneq  \dots \subsetneq K_0 = F} (-1)^\pl  \beta_{K_0,\dots,K_s} = (-1)^{ \sum_i e_i + \sum_i f_i} \prod_i Q_i(e_i)\end{equation}
 where $F \cong \prod_{i} V_i ^{e_i} \times \prod_i N_i^{f_i}$. This is certainly true for trivial $F$. Now let $F$ be nontrivial, so any chain has $\pl \geq 1$.

We have $K_{\pl-1} = \prod_i V_i^{e_i'} \times \prod_i N_i^{f_i'}$ for some $e_i' \leq e_i$ and $f_i' \leq f_i$.  Given such $\underline{e}'=(e_1',e_2',\dots)$ and $\underline{f}'=(f_1',f_2',\dots)$ there are $\prod_{i} \binom{e_i}{e_i'}_{q_i} \times \prod_i \binom{f_i}{f_i'} $ ways of choosing a $K_\pl$ with these multiplicities, where $\binom{e_i}{e_i'}_{q_i}$
denotes the $q$-binomial coefficient (see \cite[Section 5]{Liu2020} for the basics of semisimple $H$-groups).  So 
\begin{align*} &\sum_{ \substack{1 = K_\pl \subsetneq K_{\pl-1} \subsetneq  \dots \subsetneq K_0 = F
 }} (-1)^\pl \beta_{K_0,\dots, K_s} \\
 = &\sum_{ \substack{ \underline{e}', \underline{f}' \\ e'_{i}\leq e_i, f_i' \leq f_i \\ \textrm{ not all }0}} \prod_i \binom{e_i}{e_i'}_{q_i} \prod_i \binom{f_i}{f_i'}   \sum_{\substack{ \prod_i V_i^{e'_i} \times \prod_i G_i^{f_i'}    \subsetneq K_{\pl-2} \subsetneq \dots \subsetneq K_0 = \prod_i V_i^{e_i} \times \prod_i G_i^{f_i} }} (-1)^s \beta_{K_0,\dots,K_\pl} \\
 = &\sum_{ \substack{ \underline{e}', \underline{f}' \\ e_{i}'\leq e_i, f_i' \leq f_i \\ \textrm{ not all }0 }}  \prod_i \binom{e_i}{e'_i}_{q_i} \prod_i \binom{f_i}{f_i'}   \sum_{0  \subsetneq K'_{\pl-2} \subsetneq  \dots \subsetneq K'_0 = \prod_i V_i^{e_i- e'_i} \times \prod_i G_i^{f_i - f_i'}} (-1)^s  \beta_{K_0,\dots,K_\pl} ,
 \end{align*}
 where $K'_i=K_i/K_{\pl-1}$. We have $\operatorname{mult}_{V_i} (K_j) - \operatorname{mult}_{V_i} (K_0)= \operatorname{mult}_{V_i} (K_j') - \operatorname{mult}_{V_i} (K_0')$ for all $j$. Thus \[ \frac{ \beta_{K_0,\dots, K_{\pl}}}{\beta_{K'_0,\dots, K'_{\pl-1}} }  =\prod_{ \substack{i \\ V_i \textrm{ A-symplectic} \\ e_i =e'_{i}=1 } } \frac{1}{q_i}  \] since the contributions of $i$ to $\beta_{K_0,\dots, K_{\pl}}$ and $\beta_{K'_0,\dots, K'_{\pl-1}}$ agree unless $\operatorname{mult}_{V_i} (K_\pl ) - \operatorname{mult}_{V_i} (K_0) =1$ but  $\operatorname{mult}_{V_i} (K_j ) - \operatorname{mult}_{V_i} (K_0) =0$ for all $j< \pl$ which happens exactly when $e_i = e'_i=1$. Thus by induction we have
\begin{align*} 
&\sum_{ \substack{1 = K_\pl \subsetneq K_{\pl-1} \subsetneq  \dots \subsetneq K_0 = F}} (-1)^s \beta_{K_0,\dots, K_\pl} \\
= &\sum_{ \substack{ \underline{e}', \underline{f}' \\ e_{i}'\leq e_i, f_i' \leq f_i \\ \textrm{ not all }0 }}  \prod_i \binom{e_i}{e'_i}_{q_i} \prod_i \binom{f_i}{f_i'}  \prod_{ \substack{i \\ V_i \textrm{ A-symplectic} \\ e_i =e'_{i}=1 } } \frac{1}{q_i}  \sum_{0  \subsetneq K'_{\pl-2} \subsetneq  \dots \subsetneq K'_0 = \prod_i V_i^{e_i- e'_i} \times \prod_i G_i^{f_i - f_i'}}(-1)^s\beta_{K'_0,\dots, K'_{\pl-1}}  \\ 
  \end{align*}
  \begin{align*}
= &\sum_{ \substack{ \underline{e}', \underline{f}' \\ e'_{i}\leq e_i, f_i' \leq f_i \\ \textrm{ not all }0 }} \prod_i \binom{e_i}{e_i'}_{q_i} \prod_i \binom{f_i}{f_i'}   \prod_{ \substack{i \\ V_i \textrm{ A-symplectic} \\ e_i =e'_{i}=1 }} \frac{1}{q_i} (-1)   (-1)^{ \sum_i (e_i-e_i') + \sum_i (f_i-f_i')}  \prod_i Q_i(e_i - e_i')  \\=& - (-1) (-1)^{ \sum_i e_i + \sum_i f_i} \prod_i Q_i(e_i)  \\
  + &\sum_{ \substack{ \underline{e}', \underline{f}' \\ e'_{i}\leq e_i, f_i' \leq f_i } } \prod_i \binom{e_i}{e_i'}_{q_i} \prod_i \binom{f_i}{f_i'} (-1)  (-1)^{ \sum_i (e_i -e_i') + \sum_i (f_i-f_i')} \prod_i Q_i(e_i-e'_i)\prod_{ \substack{i \\ V_i \textrm{ A-symplectic} \\ e_i =e'_{i}=1 }} \frac{1}{q_i} .
  \end{align*}
 The final sum in the equation above factors over $i$.
   Let us check that each factor is $0$ unless the corresponding $e_i$ or $f_i$ is $0$. For the $f_i$ factors, it follows from the binomial theorem that $\sum_{f_i' \leq f_i} \binom{f_i}{f_i'} (-1)^{f_i-f_i'}=0$ if $f_i>0$. For the factors with $V_i$ not A-symplectic, it follows from the $q$-binomial theorem that $\sum_{e_i' \leq e_i} \binom{e_i}{e_i'}_{q_i} (-1)^{e_i-e_i'} q_i^{ \binom{e_i - e_i'}{2}}=0$ if $e_i>0$. Finally, for the A-symplectic factors,  if $e_i>1$ we have 
 \[ \sum_{e_i' \leq e_i} \binom{e_i}{e_i'}_{q_i} (-1)^{e_i- e_i'} q_i^{ \binom{e_i - e_i'}{2} - (e_i-e'_{i}) }= \prod_{j=0}^{e_i-1} ( 1 - q_i^{j-1}) =0 \] by the $q$-binomial theorem, and if 
 $e_i=1$ we have 
 \[ \binom{1}{0}_{q_i} q_i^{ \binom{1}{2} -1 } -  \binom{1}{1}_{q_i} q_i^{\binom{0}{2} - 0 } \frac{1}{q_i} = q_i^{-1} -q_i^{-1} = 0.\]
  Thus we have proven \eqref{calculating-th-induction} by induction.  Combined with \eqref{calculating-th-reduction}, this proves the lemma.
  \end{proof}

To prove Proposition~\ref{P:evalsum}, we need to evaluate 
$$
\sum_{\bH \in I} \frac{ T_{\bH} }{ \abs{ \Aut(\bH)} } \frac{  \abs{H}  \abs{ H_2 (H, \mathbb Z ) } }{  \abs{ H_1(H, \mathbb Z) } \abs{ H_3(H, \mathbb Z) }}.$$
Using $T_{\bH}=\sum_{\pi: \bH \ra \bG} T_\pi$, and \cref{calculating-th} we will calculate this via the following sums
  \[ M (\underline{e}, \underline{f}) = \sum_{ \bH \in I } \sum_{ \substack{\pi: \bH\ra \bG \\ \textrm{type } \underline{e}, \underline{f}} } \frac{ \abs{H}  \abs{ H_2 (H, \mathbb Z ) } }{ \abs{ \Aut(\bH)}   \abs{ H_1(H, \mathbb Z) } \abs{ H_3(H, \mathbb Z) }},\]
  which will occupy most of the rest of this section. 

 Fix for now $\underline{e}$ and $\underline{f}$, and let $F$ be the $[G]$-group $\prod_{i=1}^r V_i^{e_i} \times \prod_{i=1}^{s} N_i^{f_i}$. 
 It will be convenient to dualize, observing that $\abs{H_i (H,\mathbb Z) } = \abs{ H^i (H, \QZ)}$.
 %, as group cohomology is easier to work with.
 For an exact sequence $ 1 \to  F \to H \to G \to 1$ (inducing the given $[G]$-action on $F$), consider the Lyndon-Hochschild-Serre spectral sequence calculating $H^{p+q} ( H, \QZ) $. Its second page satisfies $E_2^{p,q} = H^p ( G, H^q (F, \QZ))$.  The key differentials for us in this spectral sequence are:
\begin{align*}
d_{2}^{0,1} \colon E_2^{0,1} \to E_2^{2,0} \quad  \quad \quad 
d_{2}^{1,1} \colon E_2^{1,1} \to E_2^{3,0}  \quad  \quad \quad 
d_{2}^{0,2} \colon E_2^{0,2} \to E_2^{2,1} \quad \quad \quad 
d_3^{0,2} \colon E_3^{0,2} \to E_3^{3,0}. 
\end{align*}
Given such an exact sequence, let $\Aut_{F,G}(H)$ be the group of automorphisms of $H$ that are the identity on $F$ and fix the map $H\ra G$. 
The next lemma explains how to calculate $M (\underline{e}, \underline{f})$ using information about this spectral sequence.

\begin{lemma}\label{Mef-formula}  For 
$F = \prod_{i=1}^r V_i^{e_i} \times \prod_{i=1}^{s} N_i^{f_i}$, 
we have
\begin{equation}\label{moment-spectral-formula} {|\Aut_{[G]}(F)|} M( \underline{e}, \underline{f}) = \frac{ \abs{G}  \abs{ H^2(G, \QZ) } }{  \abs{ H^1(G, \QZ)}  \abs{  H^3(G, \QZ )}  }       \frac{ \abs{H^1 (G, H^1(F, \QZ))  } } {    \abs{ H^1 (F, \QZ)^G}  }\sum_{ \substack{ 1\to F \to H \stackrel{\pi}{\to} G \to 1\\ \tau ( \ker \pi^*) =0 } }\frac{|F|}{\abs{\Aut_{F,G}(H)}}    \abs{ E_3^{0,2} } .\end{equation}
\end{lemma}

In the sum in Lemma~\ref{Mef-formula}, $F$ and $G$ are fixed, and the sum is over isomorphism classes of  material extensions $H$ of $G$ by $F$
compatible with the given $[G]$-structure on $F$, or equivalently, classes $\alpha \in H^2(G,F^{ab})$, such that for $\pi^*:H^3(G,\QZ)\ra H^3(H,\QZ)$, we have $\tau ( \ker \pi^*) =0$.   
% (More precisely, a  extension  $H$ of $G$ by $F$ is an exact sequence of groups
% $1\ra F \ra H \ra G\ra 1$ such that the outer action of $G$ on $F$ from the exact sequence agrees with the given outer action of $G$ on $F$, and an isomorphism %of such extensions is a isomorphism of exact sequences restricting to the identity map of $F$ and $G$.)   
 Throughout the rest of the section, we will have similar sums over exact sequences, and they will always mean the analogous thing, in particular requiring that the extensions are material.

\begin{proof}Recall $\tau\in H^3(G,\QZ)\ra \QZ$ is the map given by the orientation of $\bG$.
Let $J$ be a set of finite groups containing exactly one group from each isomorphism class.  
Given an $H \in J$,  we have that $\Aut(H)$ acts on  choices of $\tau_H \in H^3(H,\QZ)^\vee$,  with orbits corresponding to $\bH \in I$ and stabilizers 
$\Aut(\bH)$.
So $$
 M (\underline{e}, \underline{f}) = \sum_{ H \in J } \sum_{\tau_H \in H^3(H,\QZ)}
  \sum_{ \substack{\pi: H\ra G \\ \textrm{type } \underline{e}, \underline{f} \\ \tau_H\circ \pi^*=\tau}} \frac{ \abs{H}  \abs{ H_2 (H, \mathbb Z ) } }{ \abs{ \Aut(H)}   \abs{ H_1(H, \mathbb Z) } \abs{ H_3(H, \mathbb Z)}}.
$$
We can extend a $\pi:H\ra G$ of type $ \underline{e}, \underline{f}$ to an exact sequence
$1\ra F \ra H\ra G \ra 1$  in $|\Aut_{[G]}(F)|$ ways (compatible with the $[G]$ structure on $F$).   Also, $\Aut(H)$ acts on these exact sequences with orbits corresponding to isomorphism classes of
extensions $H$ of $G$ by $F$ (compatible with the $[G]$ structure on $F$)
and stabilizers $\Aut_{F,G}(H)$.

Thus we can rewrite
 \[ M(\underline{e}, \underline{f}) = \frac{1}{|\Aut_{[G]}(F)|}\sum_{\substack{  1\to F \to H \stackrel{\pi}{\to} G \to 1 } }\frac{1 }{|\Aut_{F,G}(H)|}  \sum_{\substack{ \tau_H \colon H^3( H, \QZ ) \to \QZ \\  \tau_H \circ \pi^* = \tau }}   \frac{  \abs{H}  \abs{ H^2 (H, \QZ ) } }{  \abs{ H^1(H, \QZ) } \abs{ H^3(H, \QZ) }} .\] 

 Because none of the terms in the sum over $\tau_H$ depend on $\tau_H$, we can replace this sum with the count of $\tau_H \colon H^3 (H, \QZ) \to \QZ$ that satisfy $\tau_H \circ \pi^* = \tau$.  Because $H^3 (H, \QZ)$ is a finite group, this number is $\frac{ \abs{ H^3 (H, \QZ)}}{ \abs{ \Img \pi^*}}$ if $\tau$ is trivial on $\ker \pi^*$ and $0$ otherwise.
 \[ M(\underline{e}, \underline{f}) = \frac{1}{|\Aut_{[G]}(F)|}\sum_{\substack{  1\to F \to H \to G \to 1\\ \tau (\ker \pi^*) =0  } }\frac{1}{\abs{\Aut_{F,G}(H)}}    \frac{  \abs{H}  \abs{ H^2 (H, \QZ ) } }{ \abs{ \Img \pi^*}    \abs{ H^1(H, \QZ) } } .\]  
Now 
\[ \abs{H^1( H, \QZ) }= \abs{ E_{\infty}^{0,1}} \abs{ E_{\infty}^{1,0} } = \abs{ \ker d_2^{0,1} } \abs{ E_2^{1,0}} =  \frac{ \abs{ H^1 (F, \QZ)^G}}{ \abs{ \Img d_2^{0,1}}}  \abs{ H^1( G, \QZ)} .\]
Furthermore
\begin{align*}
 \abs{ E_\infty^{2,0}} = \abs {E_3^{2,0}} = \abs{ \coker d_2^{0,1}} = \frac{ \abs {H^2(G, \QZ)}}{  \abs{\Img d_2^{0,1}}},  
\quad \quad \abs{E_{\infty}^{1,1}} =   \abs{E_3^{1,1}} = \abs{\ker d_2^{1,1} }=  \frac{ \abs{E_2^{1,1}} }{\abs{ \Img d_2^{1,1}}},  
\\ E_3^{0,2} = \ker d_2^{0,2} , \textrm{ and} 
\quad\quad \abs{E_{\infty}^{0,2} }= \abs{E_4^{0,2} }=\abs{ \ker d_3^{0,2} }=\frac{ \abs{E_3^{0,2} } }{\abs{ \Img d_3^{0,2}} }. 
\end{align*}
We therefore have \[ \abs{ H^2( H, \QZ) } =  \abs{ E_{\infty}^{2,0}} \cdot \abs{ E_{\infty}^{1,1} } \cdot  \abs{E_{\infty}^{0,2}} = \frac{ \abs {H^2(G, \QZ)}}{  \abs{\Img d_2^{0,1}}}   \cdot \frac{ \abs{E_2^{1,1}} }{\abs{ \Img d_2^{1,1}}}   \cdot \frac{ \abs{E_3^{0,2} } }{\abs{ \Img d_3^{0,2}} } \] 
%For $\pi^*:H^3(G,\QZ)\ra H^3(H,\QZ)$, 
We have $E_{\infty}^{3,0} = \Img \pi^* $ from the edge map of the spectral sequence. This means $\abs{ \Img d_2^{1,1}} \abs{ \Img d_3^{0,2}} \abs{ \Img \pi_* } = \abs{E_2^{3,0} } = \abs{ H^3 (G, \QZ) }$. 
This gives
\begin{align*}&\frac{   \abs{ H^2 (H, \QZ ) } }{ \abs{ \Img \pi^*}    \abs{ H^1(H, \QZ) } } =\frac{ \abs{ H^2(G, \QZ) }  \abs{E_2^{1,1}}  \abs{E_3^{0,2} }  \abs{ \Img d_2^{0,1}}  }{    \abs{ \Img \pi^*}  \abs{\Img d_2^{0,1}} \abs{ \Img d_2^{1,1}} \abs{ \Img d_3^{0,2}} \abs{ H^1 (F, \QZ)^G}  \abs{ H^1( G, \QZ)} } \\ &=  \frac{ \abs{ H^2(G, \QZ)}} {\abs{H^1(G, \QZ)} \abs{ H^3(G, \QZ)} }   \frac{\abs{E_2^{1,1}}  \abs{E_3^{0,2} }    } {   \abs{ H^1 (F, \QZ)^G}  }  
=  \frac{ \abs{ H^2(G, \QZ)}} {\abs{H^1(G, \QZ)} \abs{ H^3(G, \QZ)} }   \frac{ \abs{H^1 (G, H^1(F, \QZ)) }  } {    \abs{ H^1 (F, \QZ)^G}  }   \abs{E_3^{0,2} }  .
\end{align*}
Thus
\[{|\Aut_{[G]}(F)|} M(\underline{e}, \underline{f}) = \sum_{\substack{  1\to F \to H \to G \to 1\\ \tau (\ker \pi^*) =0  } }  \frac{\abs{H} }{ \abs{\Aut_{F,G}(H)} }  \frac{ \abs{ H^2(G, \QZ)}} {\abs{H^1(G, \QZ)} \abs{ H^3(G, \QZ)} }   \frac{ \abs{H^1 (G, H^1(F, \QZ))  } } {    \abs{ H^1 (F, \QZ)^G}  }   \abs{E_3^{0,2} }\] \[ = \frac{ \abs{G} \abs{ H^2(G, \QZ)}} {\abs{H^1(G, \QZ)} \abs{ H^3(G, \QZ)} }  \frac{ \abs{H^1 (G, H^1(F, \QZ))   } }{    \abs{ H^1 (F, \QZ)^G}  }  \sum_{\substack{  1\to F \to H \to G \to 1\\ \tau (\ker \pi^*) =0  } }   \frac{\abs{F} }{ {\abs{\Aut_{F,G}(H)}}}  \abs{E_3^{0,2} }  \] where we use $\abs{H} = \abs{F} \abs{G}$. \end{proof}

The next few lemmas let us write $M(\underline{e},\underline{f})$ as a product of local factors by showing a multiplicativity property.
%, which will culminate in writing  as a product of local factors.

 \begin{lemma}\label{automorphism-of-extension} For 
$F = \prod_{i=1}^r V_i^{e_i} \times \prod_{i=1}^{s} N_i^{f_i}$, and $H$ as in \cref{Mef-formula},
  \[ \frac{|F|}{\abs{ \Aut_{F,G}(H)}}= \prod_{i=1}^r
  \frac{\abs{ V_i^{G} }^{e_i}}{   \abs{ H^1 ( G, V_i)^{e_i}}  }
 \prod_{i=1}^s |N_i|^{f_i}  .\] \end{lemma}
 
 \begin{proof}The group $H$ is a fiber product over $G$ of the extensions by the $V_i^{e_i}$ and 
 $N_i^{f_i}$ separately, and any element  of $ \Aut_{F,G}(H)$ acts separately on the factors.  The 
 factors for $N_i^{f_i}$ have no automorphisms fixing $N_i^{f_i}$ and $H$ since the extension is 
 canonically $H\times_{\Out(N_i^{f_i})}\Aut(N_i^{f_i})$. An automorphism of the extensions of $H$ 
 by $V_i^{e_i}$ is a cocycle in the standard presentation for $H^1( G, V_i^{e_i})$, and it is a 
 coboundary if and only if it acts as conjugation by an element of $V_i^{e_i}$.  Conjugation by an 
 element is a trivial automorphism if and only if the element is central, which happens exactly if it is $G$-invariant, so $|\Aut_{F,G}(H)|$ is $\abs{ H^1 ( G, V_i)^{e_i}}   \abs{ V_i^{e_i} } /  
 \abs{ V_i^{G} }^{e_i} $. \end{proof}

 \begin{lemma}\label{automorphism-of-F}  For 
$F = \prod_{i=1}^r V_i^{e_i} \times \prod_{i=1}^{s} N_i^{f_i}$,\[\abs{\Aut_{[G]}(F)} = \prod_{i=1}^r \abs{GL_{e_i}(\kappa_i) }\prod_{i=1}^s  \abs{N_i}^{f_i} \abs{ Z_{\Out(N_i)} (G)}^{f_i} f_i! .\]
\end{lemma}
\begin{proof} An automorphism of $F$ acts separately on each factor, so 
\[ \abs{\Aut_{[G]}(F)} =  \prod_{i=1}^r \abs{\Aut_{[G]} (V_i^{e_i} )}\times \prod_{i=1}^{s} \abs{ \Aut_{[G]}. ( N_i^{f_i})}.\]
A $G$-endomorphism of $V_i^{e_i}$ is given by a $e_i \times e_i$ matrix over $\kappa_i$ and it is an automorphism if and only if the matrix is invertible.
A automorphism of $N_i^{f_i}$ is a $[G]$-automorphism if and only if its image in $\Out(N_i^{f_i})$ commutes with $G$, so
\[ \abs{ \Aut_{[G]} ( N_i^{f_i}) }= \abs{N_i}^{f_i}  \abs{ Z_{\Out(N_i^{f_i}) } (G) } = \abs{N_i}^{f_i} \abs{ Z_{\Out(N_i)} (G)}^{f_i} f_i! .\]
\end{proof}

\begin{lemma} \label{pullback-differential} Let \[ \begin{tikzcd} 1 \arrow[r] \arrow[d] & F \arrow[r] \arrow[d,"\rho"]& H \arrow[r]\arrow[d] & G \arrow[r] \arrow[d,"id"] & 1\arrow[d] \\ 
1 \arrow[r] & F'\arrow[r] & H' \arrow[r] & G \arrow[r] & 1 \\ \end{tikzcd}\] be a commutative diagram of groups with both rows exact.  Then the differentials in the Lyndon-Hochschild-Serre spectral sequences $(E,d),(E',d')$ computing $H^{p+q} (H , \QZ)$ and $H^{p+q} (H' , \QZ)$, respectively, from $H^p ( G, H^q ( F, \QZ))$ and $H^p ( G, H^q ( F', \QZ))$, respectively, are compatible with the pullback map 
\[ \rho^*: H^p ( G, H^q ( F', \QZ)) \to H^p ( G, H^q ( F, \QZ)),\]
i.e. for all $r\geq 2$,
$$
\rho^* (d')_r^{p,q} =d_r^{p,q} \rho^*,
$$  
where pullback maps $\rho^*$ on pages past the second page are well-defined by the commutativity of these diagrams on previous pages.
%induced by the pullback $H^q ( F_b, \QZ ) \to  H^q ( F_a, \QZ)$ in the sense that, on each page, every diagram of the form ``pullback composed with a differential equals differential composed with pullback" commutes, 

\end{lemma}

\begin{proof} 
This follows from the definition of the spectral sequence as the spectral sequence of the bicomplex $K^{p,q}:=\Hom_G(B_p(G,\Hom_F(B_q(H,\QZ) )))$, where $B_*$ denotes the bar resolution and the differentials of the bicomplex come from the differentials on the bar resolutions
(see \cite[Section XI.10]{MacLane1995}, \cite[Section 3]{Huebschmann1981}).  The pullback map $(K')^{p,q} \ra K^{p,q}$ induces the pullback maps on all the pages of the spectral sequence and is compatible with the differentials.
\end{proof}

\begin{lemma}\label{inner-sum-product} Let $F_a$ and $F_b$ be  $[G]$-groups that are each finite products of the $V_i$ and $N_i$. Assume that, for each irreducible representation $V_i$ over $\F_p$  of $G$ appearing in $F_a$, the dual representation $V_i^\vee:=\Hom(V_i,\F_p)$  does not appear in $F_b$. 
Then 
\[ \sum_{ \substack{ 1\to (F_a \times F_b) \to H \stackrel{\pi_{ab}}{\to} G \to 1\\ \tau ( \ker \pi_{ab}^*) =0 } }  \abs{E_3^{0,2} }  = \Biggl( \sum_{ \substack{ 1\to F_a\to H \stackrel{\pi_a}{\to} G \to 1\\ \tau ( \ker \pi_a^*) =0 } }  \abs{E_3^{0,2} } \Biggr)  \Biggl( \sum_{ \substack{ 1\to F_b \to H \stackrel{\pi_b}{\to} G \to 1 \\ \tau ( \ker \pi_b^*) =0 } }  \abs{ E_3^{0,2} }\Biggr)   .\]

\end{lemma}

\begin{proof} Every extension $H_{ab}$ of $G$ by $F_a \times F_b$ is the fiber product of an extension $H_a$ of $G$ by $F_a$ and an extension $H_b$ of $G$ by $F_b$. Thus, matching terms on both sides, it suffices to show that 
\begin{equation}\label{E3-product} E_{3, H_{ab}}^{0,2}=  E_{3, H_a}^{0,2} \times E_{3, H_b}^{0,2}\end{equation} and 
$\tau ( \ker \pi_{ab}^*) =0 $ if and only if $\tau ( \ker \pi_{a}^*) =0 $ and $\tau ( \ker \pi_{b}^*) =0 $ .

In each of the spectral sequences we have $E_{3}^{0,2}=\ker(d_2^{0,2})$. We first consider the $E_2^{0,2}$ terms and will  check that the product of natural pullback maps
\begin{equation}\label{E02splitting} H^0 (G , H^2(F_{a}, \QZ)) \times H^0 (G , H^2(F_{b}, \QZ))\to H^0 (G , H^2(F_{ab}, \QZ))\end{equation} is an isomorphism. 
For any finite group $F$, the exact sequence $0 \to \mathbb Z \to \mathbb Q \to \QZ\to 0$ gives an isomorphism $H^q ( F, \mathbb Q/\mathbb Z) \cong H^{q+1} (F, \mathbb Z)$ for $q>0$. 
Because $F_a$ and $F_b$ are finite, $H^1(F_a,\mathbb Z)= H^1(F_b, \mathbb Z)=0$, so, by the K\"unneth formula with principal ideal domain coefficients, and we have an exact sequence 
\[ 
1\ra 
 H^3 (F_a, \mathbb Z) \times H^3 (F_b, \mathbb Z) \ra H^3 (F_a \times F_b, \mathbb Z)\ra \operatorname{Tor}^1 ( H^2 (F_a, \mathbb Z), H^2(F_b, \mathbb Z))\ra 1.\]

Thus, to prove \eqref{E02splitting}, it suffices to prove $H^0(G, \operatorname{Tor}^1 ( H^2 (F_a, \mathbb Z), H^2(F_b, \mathbb Z)))=0$.  We have $H^2(F_a, \mathbb Z)= H^1(F_a, \mathbb Q/\mathbb Z) = \Hom (F_a, \mathbb Q/\mathbb Z)$ is a product of vector spaces over finite fields, and the same for $H^2(F_b, \mathbb Z)$.  For such finite abelian groups $A,B$, there is a isomorphism
$\operatorname{Tor}^1(A,B)\ra A \otimes B$ functorial in both $A$ and $B$. It suffices to check that $\Hom (F_a, \mathbb Q/\mathbb Z)\otimes \Hom (F_b, \mathbb Q/\mathbb Z)$ contains no nontrivial element that  $G$-invariant. Such an element would give a nontrivial $G$-invariant $\QZ$-valued bilinear form on $F_a \times F_b$ (again using that $F_a,F_b$ are products of vector spaces over finite fields), which cannot exist because of our assumption on the irreducible factors of $F_a$ and $F_b$.

More straightforwardly, the product of natural pullback maps \[ H^1 (F_a, \QZ) \times H^1(F_b, \QZ) \to H^1(F_a\times F_b, \QZ)\] is an isomorphism because these cohomology groups are the same as sets of homomorphisms to $\QZ$, hence \begin{equation}\label{pullback-product-1} H^p (G , H^1(F_{a}, \QZ)) \times H^p (G , H^1(F_{b}, \QZ))\to H^p (G , H^1(F_{a} \times F_b, \QZ))\end{equation}  is an isomorphism for all $p$.

Using \eqref{E02splitting}, the $p=2$ case of \eqref{pullback-product-1}, and Lemma \ref{pullback-differential}, it follows that \[ d_{2,ab}^{0,2}: H^0 (G , H^2(F_{ab}, \QZ)) \to H^2 (G , H^1(F_{ab}, \QZ))\] is the product of $ d_{2,a}^{0,2}$ and $ d_{2,b}^{0,2}$. Hence $E_{3,G_{ab}}^{0,2}= \ker d_{2,ab}^{0,2}$ is the product of the kernel $E_{3, G_{a}}^{0,2}$ of $d_{2,a}^{0,2}$  and the kernel $E_{3, G_{b}}^{0,2}$ of $d_{2,b}^{0,2}$, verifying \eqref{E3-product}. 

We have $\tau (\ker \pi_{ab}^*)=0$ if and only if $\tau \circ d_{2,ab}^{1,1}=0$ and $\tau \circ d_{3,ab}^{0,2}=0$.  Using Lemma \ref{pullback-differential} and the $p=1$ case of \eqref{pullback-product-1}, the map $d_{2,ab}^{1,1}$ is the sum of $d_{2,a}^{1,1}$ and $d_{2,b}^{1,1}$, hence $\tau \circ d_{2,ab}^{1,1}=0$ if and only if $\tau \circ d_{2,a}^{1,1}=0$ and $\tau \circ d_{2,b}^{1,1}=0$. Similarly, using Lemma \ref{pullback-differential} and \eqref{E3-product}, $d_{3,ab}^{0,2}$ is the sum of $d_{3,a}^{0,2}$ and $d_{3,b}^{0,2}$, hence $\tau \circ d_{3,ab}^{0,2}=0$ if and only if $\tau \circ d_{3,a}^{0,2}=0$ and $\tau \circ d_{3,b}^{0,2}=0$.  \end{proof}

%For $V_i$ a representation of $G $ that is not dual to $V_j$ for any $j\neq i$, and a natural number $e_i$, define
%\[ \tilde{M}_i ( e_i)= \sum_{ \substack{ 1\to V_i^{e_i} \to H \stackrel{\pi}{\to} G \to 1\\ \tau ( \ker \pi^*) =0 } }  \abs{ E_3^{0,2} }  .\]
%
%For $V_i$, $V_{i'}$ non-isomorphic dual representations, and natural numbers $e_i,e'_{i}$, define
%\[ \tilde{M}_{i,i'} (e_i, e'_{i}) = \sum_{ \substack{ 1\to (V_i^{e_i} \times V_{i'} ^{e'_{i}})  \to H \stackrel{\pi}{\to} G \to 1\\ \tau ( \ker \pi^*) =0 } }  \abs{ E_3^{0,2} }  .\]
%
%For $N_i$ a non-abelian simple $[G]$-group and $f_i$ a natural number, define
%
%\[ \tilde{\eta}_i ( f_i) =  \sum_{ \substack{ 1\to N_i^{f_i} \to H \stackrel{\pi}{\to} G \to 1\\ \tau ( \ker \pi^*) =0 } }  \abs{ E_3^{0,2} }  .\]
%\mn{Probably there are better combinations of letters.}
%
%
%\begin{lemma} Let $F \cong \prod_{i=1}^r V_i^{e_i} \times \prod_{i=1}^{s} N_i^{f_i}$. We have
%
%\begin{equation}\label{first-factorization}  \sum_{ \substack{ 1\to F \to H \to G \to 1\\  \tau ( \ker \pi^*) =0 } }  \abs{ E_3^{0,2} }  = \prod_{ \substack{ i \in \{1,\dots ,r \} \\ V_i^\vee \not\cong V_j \textrm{ for any } j\neq i}} \tilde{M}_i ( e_i) \prod_{\substack{ \{i, i'\} \subseteq \{1,\dots, r\}  \\  i \neq i' \\ V_i \cong V_{i'}^\vee  }} \tilde{M}_{i,i'} (e_i, e'_{i})  \prod_{i=1}^s \tilde{\eta}_i (f_i ) \end{equation}
%
%
%\end{lemma}
%
%\begin{proof} This follows from inductively applying Lemma \ref{inner-sum-product} to separate out the individual $V_i$ factors and dual pairs. \end{proof}

We now define the local factors that we will write $M(\underline{e}, \underline{f})$ as a product of.

For any $V_i$ that is not dual to $V_j$ for any $j\neq i$,  let
\[ M_i (e_i) = \frac{1}{ \abs{ GL_{e_i}(\kappa_i)} }  \frac{ \abs{ H^1(G, V_i^\vee)}^{e_i}} {  \abs{ H^1(G, V_i)}^{e_i}}   \sum_{ \substack{ 1\to V_i^{e_i} \to H \stackrel{\pi}{\to} G \to 1\\ \tau ( \ker \pi^*) =0 } }  \abs{ E_3^{0,2} } .\]

For $V_i$ and $V_{i'}$, dual to each other, and non-isomorphic, define
\[ M_{i,i'} (e_i, e_{i'} ) = \frac{1}{ \abs{ GL_{e_i}(\kappa_i) }\abs{GL_{e_{i'}}(\kappa_{i'})} }    \abs{ H^1(G, V_{i'} )}^{e_i -e_{i'} }   \abs{ H^1(G, V_i)}^{e_{i'}-e_i}  \sum_{ \substack{ 1\to (V_i^{e_i} \times V_{i'} ^{e_{i'}})  \to H \stackrel{\pi}{\to} G \to 1\\ \tau ( \ker \pi^*) =0 } }  \abs{ E_3^{0,2} } .\] 

For any $N_i$, let
\[ \eta_i (f_i) = \frac{1}{ \abs{ Z_{\Out(N_i)} (G)}^{f_i} f_i! } 
%\tilde{\eta}_i (f_i) 
 \sum_{ \substack{ 1\to N_i^{f_i} \to H \stackrel{\pi}{\to} G \to 1\\ \tau ( \ker \pi^*) =0 } }  \abs{ E_3^{0,2} }. \]

\begin{lemma}\label{L:Mefform}
We have
\begin{equation} \label{moment-product-formula} M(\underline{e},\underline{f})  =
\frac{ \abs{G}  \abs{ H^2(G, \QZ) } }{  \abs{ H^1(G, \QZ)}  \abs{  H^3(G , \QZ )}  }\prod_{ \substack{ i \in \{1,\dots ,r \} \\ V_i^\vee \not\cong V_j \textrm{ for any } j\neq i}} {M}_i ( e_i) \prod_{\substack{ \{i, i'\} \subseteq \{1,\dots, r\}  \\  i \neq i' \\ V_i \cong V_{i'}^\vee  }} {M}_{i,i'} (e_i, e_{i'})  \prod_{i=1}^s {\eta}_i (f_i ).\end{equation}  \end{lemma}

\begin{proof}  Combining Lemmas \ref{Mef-formula}, \ref{automorphism-of-extension}, and \ref{automorphism-of-F}, and noting that $H^1( F, \QZ ) = \prod_{i=1}^r (V_i^\vee)^{e_i}$, we have
 \[ M(\underline{e}, \underline{f}) \frac { \abs{ H^1(G, \QZ)}  \abs{  H^3(G ,\QZ )}  } { \abs{G}  \abs{ H^2(G, \QZ) } } \] \[= \prod_{i=1}^r  \frac{1}{ \abs{ GL_{e_i}(\kappa_i)} }\frac{ \abs{ V_i^G}^{e_i} } { \abs{ (V_i^\vee)^G }^{e_i}}  \frac{ \abs{ H^1(G, V_i^\vee)}^{e_i}} {  \abs{ H^1(G, V_i)}^{e_i}}   \prod_{i=1}^s \frac{ \abs{N_i}^{f_i}}{\abs{N_i}^{f_i} \abs{ Z_{\Out(N_i)} (G)}^{f_i} f_i! }   \sum_{ \substack{ 1\to F \to H \to G \to 1\\ \tau ( \ker \pi^*) =0 } }    \abs{ E_3^{0,2} } \]
 \[ =  \prod_{i=1}^r    \frac{1}{ \abs{ GL_{e_i}(\kappa_i) }}\frac{ \abs{ H^1(G, V_i^\vee)}^{e_i}} {  \abs{ H^1(G, V_i)}^{e_i}}   \prod_{i=1}^s \frac{1}{ \abs{ Z_{\Out(N_i)} (G)}^{f_i} f_i! }   \sum_{ \substack{ 1\to F \to H \to G \to 1\\ \tau ( \ker \pi^*) =0 } }    \abs{ E_3^{0,2} } \] since the $V_i$ are irreducible representations, so $(V_i)^G $ and $(V_i^\vee)^G$ are dual and have the same order.

We inductively apply Lemma \ref{inner-sum-product} to express the inner sum over extensions as a product of sums associated to individual $V_i$ and $N_i$ factors or dual pairs of $V_i$. We then note that, by definition, the $M_i, M_{i,i'}$, $\eta_i$ factors incorporate these sums together with the extra  \[ \prod_{i=1}^r   \frac{1}{ \abs{ GL_{e_i}(\kappa_i) } }\frac{ \abs{ H^1(G, V_i^\vee)}^{e_i}} {  \abs{ H^1(G, V_i)}^{e_i}}   \prod_{i=1}^s \frac{1}{ \abs{ Z_{\Out(N_i)} (G)}^{f_i} f_i! }   \]  terms. The lemma immediately follows.
\end{proof}

 The remaining subsections compute the local factors for different types of $N_i, V_i$, in order of increasing difficulty. 
 
  \subsection{Non-abelian groups}

 Recall $\delta_{N_i}$ is the differential $ d^{0,2}_3 \colon H^2 ( N_i, \QZ)^{ {G} } \to H^3 ({G},\QZ)$ appearing in the Lyndon-Hochschild-Serre spectral sequence  computing $H^{p+q} (  {G} \times_{ \Out(N_i)} \Aut(N_i), \QZ) $ from $H^p ( {G}, H^q ( N_i, \QZ))$.
 
 \begin{lemma}\label{non-abelian-tilde} For any $i$ from $1$ to $s$, we have \[ 
 %\tilde{\eta}_i (f_i) =
\sum_{ \substack{ 1\to N_i^{f_i} \to H \stackrel{\pi}{\to} G \to 1\\ \tau ( \ker \pi^*) =0 } }  \abs{ E_3^{0,2} } = \begin{cases} 
1 &  f_i=0 \\ 
0 & \tau \circ \delta_{N_i} \neq 0 \textrm{ and }f_i >0\\
(  \abs{H^2( N_i, \QZ)^{G} }) ^{f_i} 
& \tau \circ \delta_{N_i} = 0
 \end{cases} .\]   
 \end{lemma}
 
  \begin{proof}  We have $H^1(N_i^{f_i}, \QZ)=0$ so $H^p(G, H^1(N_i^{f_i}, \QZ))=0$ for all $p$. Thus the differentials $d_2^{1,1}$ and $d_2^{0,2}$ vanish.
  Note we have $H^0 ( G, H^2 ( N_i^{f_i}, \QZ)) =  H^0 ( G, H^2 ( N_i, \QZ))^{f_i}$, since $H^1(N_i,\Z)=H^2(N_i,\Z)=0$ implies there are no middle or Tor terms in the K\"{u}nneth formula for $H^2 ( N_i^{f_i}, \Z)$.  Thus by  Lemma \ref{pullback-differential}, $d_3^{0,2}$ can be computed by taking products over the differential for the map when $f_i=1$.
We then have $\tau ( \ker \pi^* )=0$ if and only if $\tau \circ d_3^{0,2} =0$, which %by Lemma \ref{pullback-differential} 
happens if and only if $\tau \circ \delta_{N_i}=0$ (at least for $f_i>0$).
  As another consequence, we have $  E_3^{0,2}%= H^0 ( G, H^2 ( N_i^{f_i}, \QZ))
  =  H^0 ( G, H^2 ( N_i, \QZ))^{f_i} .$  The lemma follows.
%
% Thus, when  $\tau \circ \delta_{N_i}=0$, we have
% \[ 
% \sum_{ \substack{ 1\to N_i^{f_i} \to H \stackrel{\pi}{\to} G \to 1\\ \tau ( \ker \pi^*) =0 } } 
% \abs{ E_3^{0,2}} = \abs{ H^0 ( G, H^2 ( N_i, \QZ))} ^{f_i}\] 
%and thus
%\[  \eta_i (f_i) = \frac{1}{ \abs{ Z_{\Out(N_i)} (G)}^{f_i} f_i! }  \sum_{ \substack{ 1\to N_i^{f_i} \to H \stackrel{\pi}{\to} G \to 1\\ \tau ( \ker \pi^*) =0 } } %\abs{ E_3^{0,2}} = \frac{1}{(f_i)!}  \left (\frac{ \abs{ H^0 ( G, H^2 ( N_i, \QZ))}}{ \abs {Z_{\Out(N_i)} (G)}} \right)^{f_i} ,\]
%while if $\tau \circ \delta_{N_i}\ne 0$, then ${\eta_i}(f_i)=0$ if $f_i>0$ and $1$ if $f_i=0$. 
 \end{proof}

\begin{lemma}\label{non-abelian-w} We have \[ \sum_{f_i=0}^{\infty} (-1)^{f_i}  \eta_i (f_i) = w_{N_i} \] with the sum absolutely convergent. \end{lemma}
\begin{proof} By \cref{non-abelian-tilde}, if $\tau \circ \delta_{N_i} \neq 0$ then $\sum_{f_i=0}^{\infty} (-1)^{f_i}  \eta_i (f_i) =1$, and we have defined $w_{N_i}=1$ in this case. 
If  $\tau \circ \delta_{N_i} \neq 0$ then 
\[ \sum_{f_i=0}^{\infty} (-1)^{f_i}  \eta_i (f_i)  = \sum_{f_i=0}^{\infty}  \frac{1}{(f_i)!}  \left (-\frac{ \abs{ H^0 ( G, H^2 ( N_i, \QZ))}}{ \abs {Z_{\Out(N_i)} (G)}} \right)^{f_i}   = e^{ - \frac{ \abs{ H^0 ( G, H^2 ( N_i, \QZ))}}{ \abs {Z_{\Out(N_i)} (G)}}} = w_{N_i} ,\] again by definition of $w_{N_i}$. 
%This sum is convergent because the power series for $e^x$ is absolutely convergent.
\end{proof}

% \begin{lemma}\label{non-abelian-factor} For any $i$ from $1$ to $s$, if $\tau \circ \delta_{N_i} %\neq 0$, then \[ \tilde{\eta}_i (f_i) = \begin{cases} 1 & f_i=0 \\ 0 & f_i >0 \end{cases} ,\] and if  $
%\tau \circ \delta_{N_i} =0$, then \[ {\eta}_i (f_i) =  ( \abs{N_i } \abs{H^2( N_i, \QZ)^{G} }) ^{f_i}.\]  
%\end{lemma}
 
 %\begin{proof} This follows from the definition $ \eta_i(f_i) = \abs{N_i}^{f_i}  \tilde{\eta}_i (f_i) $ %and \cref{non-abelian-tilde}. \end{proof} 
 
  \subsection{Those representations whose dual representations do not appear}
  
Recall $W_i^{\tau}$ is the set of  $\alpha \in W_i$ such that $\tau ( \alpha \cup \beta) =0 $ for all $\beta \in H^1 ( {G}, V_i^\vee)$.  
  We start with a general lemma on one term of the spectral sequence.

\begin{lemma}\label{lem-wedge}
Let $F$ be finite abelian.  Then $E_2^{0,2}=H^2(F,\QZ)^G=(\wedge^2 F^\vee)^G$ (the implicit tensor product is over $\Z)$.  
\end{lemma}  
\begin{proof}
%Let $V_i$ be characteristic $p$. In this proof, tensor products will be over $\Z$.  
%We have a natural map $H_2(V_i^{e_i},\Z) \ra 
%(V_i^{e_i})^{\otimes 2}$ (given, e.g., by the obvious map on chains) whose image clearly %contains 
%all elements of the form $a\otimes b-b\otimes a$ and the order $|H_2(V_i^{e_i},\Z)|$ is 
%well-known to be the size of the subgroup generated by such elements.  Dualizing, we have a %natural 
%isomorphism $H^2(V_i^{e_i},\QZ) \ra \wedge^2 (V_i^{\vee})^{e_i}$.
We have a natural map $\Hom(F\tensor F,\QZ) \ra H^2(F,\QZ)$ in which a bilinear form maps to the cochain that evaluates it.  One can check that the kernel of this map is the set of symmetric bilinear forms, and we conclude that we have a natural injection
 $ \wedge^2 (F^{\vee}) \ra H^2(F,\QZ)$.
Since $H^2(F,\QZ)=\Hom(H_2(F,\Z),\QZ)$ and it is well-known that $|H_2(F,\Z)|=|\wedge^2 F|$, we see the injection must be an isomorphism.  
%,  $ H^2(V_i^{e_i}, \QZ) = \wedge^2 ( V_i^{e_i} ) $ where we view $V_i^{e_i}$ as a vector space %over $\mathbb F_p$. 
%we may embed $\wedge^2(V_i^{e_i})$ as a subspace of $(V_i^{\vee e_i}) \otimes (V_i^{\vee e_i}) %= (V_i^\vee \otimes V_i^\vee)^{e_i^2}$, 
%
%This is a great point, but I don't know that we need it here.
%Even if $p=2$,  we have a natural injection $\wedge^2 (V_i^{\vee})^{e_i}\ra ((V_i^{\vee})^{e_i})^{\otimes 2}$ (given by $a\wedge b \mapsto a\tensor b-b
%\tensor a$)
\end{proof}

  Now we have our next evaluation of one of our local factors.
  \begin{lemma}\label{no-dual-tilde} Let $V_i$ be a representation such that $V_i^\vee$ is not isomorphic to $V_j$ for any $j$ from $1$ to $r$. Then
\begin{align*}
%\tilde{M}_i (e_i) = \abs{W_i^\tau} ^{ e_i  } \quad \textrm{ and } \quad
  {M}_i (e_i) = \frac{1}{ \abs {GL_{e_i} (\kappa_i)}}  \left( \frac{\abs{W_i^\tau}\abs{H^1(G, V_i^\vee)}}{\abs{H^1(G, V_i)}}\right)^{ e_i } .
%\tilde{M}_i (e_i) = q_i^{ e_i \dim_{\kappa_i} W_i^\tau } \quad \textrm{ and } \quad
%  {M}_i (e_i) = q_i^{ e_i (\dim_{\kappa_i} W_i^\tau   + \dim_{\kappa_i} H^1(G, V_i^\vee) - 
%\dim_{\kappa_i} H^1(G, V_i) ) } .
  \end{align*}
  \end{lemma}
  
  \begin{proof} Let us first check in this case that, for any extension of $G$ by $V_i^{e_i}$, \begin{equation}\label{no-dual-no-invariants} H^0 (G, H^2( V_i^{e_i}, \QZ)) =0.\end{equation}
By Lemma~\ref{lem-wedge}, we have   $E_2^{0,2}=(\wedge^2 (V_i^\vee)^{e_i})^G$.
Even if $p=2$,  we have a natural injection $\wedge^2 (V_i^{\vee})^{e_i}\ra ((V_i^{\vee})^{e_i})^{\otimes 2}$ (given by $a\wedge b \mapsto a\tensor b-b
\tensor a$)
so if $ H^2(V_i^{e_i}, \QZ)$  admits a nonzero $G$-invariant vector then so must $ (V_i^\vee\otimes V_i^\vee)^{e_i^2}$ and hence also $V_i^\vee \otimes V_i^\vee$. This would give a nontrivial $G$-equivariant map $V_i^\vee\ra V_i$,
%
% bilinear form on $V_i$, 
necessarily an isomorphism because $V_i$ is irreducible, making $V_i$ self-dual, contradicting our assumption. Thus \eqref{no-dual-no-invariants} holds.

Hence $E_3^{0,2} =0$ and therefore ${M}_i (e_i)$ is   $\frac{1}{ \abs {GL_{e_i} (\kappa_i)}}  \left( \frac{\abs{H^1(G, V_i^\vee)}}{\abs{H^1(G, V_i)}}\right)^{ e_i } $ times the number of
material extensions $\pi \colon H \to G$ by $V_i^{e_i}$, i.e. 
 $\alpha\in W_i^{e_i}$,  such that $\tau ( \ker \pi^* )=0$.  
Since $d_3^{0,2}$ vanishes since $E_2^{0,2}$ does by \eqref{no-dual-no-invariants}, the lemma follows from Lemma~\ref{lem-tau}.
      \end{proof}
      
     \begin{lemma}\label{no-dual-w} Let $V_i$ be a representation such that $V_i^\vee$ is not isomorphic to $V_j$ for any $j$ from $1$ to $r$. Then we have \[ \sum_{e_i=0}^{\infty} (-1)^{e_i}  M (e_i) Q_i (e_i) = w_{V_i} \] with the sum absolutely convergent. \end{lemma}
     
     \begin{proof} Since $V_i$ is not self-dual, it is certainly not A-symplectic, so by definition $Q_i(e_i) = q_i^{ \binom{e_i}{2}}$. Using \cref{no-dual-tilde}, we thus have
     \[ \sum_{e_i=0}^{\infty} (-1)^{e_i}  M (e_i) Q_i (e_i) = \sum_{e_i=0}^{\infty} (-1)^{e_i}   \frac{ q_i^{ \binom{e_i}{2}} }{ \abs {GL_{e_i} (\kappa_i)}}  \left( \frac{\abs{W_i^\tau}\abs{H^1(G, V_i^\vee)}}{\abs{H^1(G, V_i)}}\right)^{ e_i } .\]
        We have $ \abs {GL_{e_i} (\kappa_i)} = \prod_{j=1}^{e_i}  (q_i^{e_i} - q_i^{e_i -j } )$ so $\frac{ q_i^{ \binom{e_i}{2} }}{ \abs {GL_{e_i} (\kappa_i)}} = \frac{1}{ \prod_{j=1}^{e_i} (q_i^j-1)}$.
        We now apply the $q$-exponential identity
   \[ \sum_{e=0}^{\infty} (-1)^e \frac{  u^e}{ \prod_{j=1}^e  (q^j-1)} = \prod_{j=1}^{\infty} ( 1- q^{-j} u ) \] 
     where the left side is absolutely convergent for $q>1$ and any $u$ because the numerators grow exponentially and the denominators grow superexponentially, to obtain
     \[ \sum_{e_i=0}^{\infty} (-1)^{e_i}  M (e_i) Q_i (e_i) = \prod_{j=1}^{\infty} \left( 1- q_i ^{-j}  \frac{\abs{W_i^\tau}\abs{H^1(G, V_i^\vee)}}{\abs{H^1(G, V_i)}} \right) = w_{V_i} \]
      by definition of $w_{V_i}$.
  \end{proof} 
 
 %   \begin{lemma}\label{no-dual-factor} Let $V_i$ be a representation such that $V_i^\vee$ is %not isomorphic to $V_j$ for any $j$ from $1$ to $r$. Then
%  \[ \tilde{M}_i (e_i) = q_i^{ e_i (\dim W_i^\tau   + \dim H^1(G, V_i^\vee) - \dim H^1(G, V_i) ) } .\]
%\end{lemma}
  
%  \begin{proof}This follows from \cref{no-dual-tilde} and the definition of $M_i$ in this case. %\end{proof}

 \subsection{Representations whose duals appear}

If $V_i^\vee=V_j$ for some $j$, then because of Lemma~\ref{lem:to_alpha_zero}, we are interested in the case when $W_i^\tau=0$, which by Lemma~\ref{lem-tau}
means we are interested in the case of the trivial extension, where $H=F \rtimes G$.

\begin{lemma}\label{odd-d3}
If  $H=F \rtimes G$, then $d_3^{0,2}=0$.
\end{lemma}
\begin{proof}
Since $F\rtimes G \ra G$ has a section, the edge map $H^3(G,\QZ)\ra H^3(F\rtimes G,\QZ)$ of the spectral sequence is injective, and thus $d_3^{0,2}=0$.
\end{proof}

\begin{lemma}\label{dual-M}
Let $V_i$ be $\F_p$-self-dual.  Assume $W_i^\tau=0$. Then
$$
 M_i (e_i) = \frac{1}{ \abs{ GL_{e_i}(\kappa_i)} } q_i^{\frac{e_i(e_i-\epsilon_i)}{2}}.
$$
Let $V_i,V_{i'}$ non-isomorphic dual representations of $G$.  Assume $W_i^\tau = W_{i'}^\tau=0$. 
 Then
$$
M_{i,i'} (e_i, e_{i'} ) = \frac{1}{ \abs{ GL_{e_i}(\kappa_i) }\abs{GL_{e_{i'}}(\kappa_{i'})} }   \frac{ \abs{ H^1(G, V_{i'} )}^{e_i -e_{i'} }  }{ \abs{ H^1(G, V_i)}^{e_i-e_{i'}}} q_i^{e_ie_{i'}} .
$$
\end{lemma}

\begin{proof}
Let $F=V_i^{e_i}$ or $F=V_i^{e_i}\times V_{i'}^{e_{i'}}$, depending on the case of the lemma.
By Lemma~\ref{lem-tau}, in the sum over $H$ in the definition of $M_i(e_i)$ (or $M_{i,i'}(e_i,e_{i'})$), we only need to consider the trivial extension $H=F\rtimes G$ (which does appear in the sum by Lemma \ref{odd-d3}).

We have a map $V_i\rtimes G \ra V\rtimes G$ for each of the $e_i$ coordinate inclusions $p_j:V_i \ra V_i^{e_i}$ (and similarly for $p'_j:V_{i' }\ra V_{i'}^{e_{i'}}$),   and thus we can apply Lemma~\ref{pullback-differential} and see that we have a commutative diagram
\[ \begin{tikzcd} H^2(F,\QZ)^G \arrow[r,"d_2^{0,2}"] \arrow[d,"p^*_j"] & H^2(G, F^\vee)  \arrow[d,"p^*_j"] \\ 
H^2(V_i,\QZ)^G \arrow[r,"d_2^{0,2}"]  & H^2(G, V_i^\vee)  \end{tikzcd}\]
(and similarly for the $p'_j$).  
Thus for $\phi \in H^2(F,\QZ)^G$ we have $d_2^{0,2}(\phi)=\sum_j d_2^{0,2}(p^*_j(\phi))+d_2^{0,2}((p')^*_j(\phi))$.
From Lemma~\ref{lem-wedge}, it then follows that in the case that  $\wedge^2(V_i)^G=0$, then $d_2^{0,2}(H^2(F,\QZ)^G)=0$. So if 
$F=V_i^{e_i}$ and $\wedge^2(V_i)^G=0$, then $\epsilon_i=1$ and $|E_3^{0,2}|=q_i^{\frac{e_i(e_i-1)}{2}}$, and the lemma holds in this case.  If $F=V_i^{e_i}\times V_{i'}^{e_{i'}}$, then $|E_3^{0,2}|=q_i^{e_ie_{i'}}$, and the lemma holds in this case.
%As in the proof of Lemma~\ref{no-dual-tilde}, we have $H^2(V_i^{e_i},\QZ)^G=\wedge^2(V_i^{e_i})^G$.

Now we consider the case $\wedge^2(V_i)^G\ne 0$,.
There are three natural actions of $\kappa_i^*=\Aut_G(V_i)$ on $(V_i^\vee \tensor V_i^\vee)^G$,  via the left $V_i^\vee$, the right $V_i^\vee$, or both simultaneously (the \emph{double} action).  
From the fact that $V_i$ is irreducible,  we have that $(V_i^\vee \tensor V_i^\vee)^G$ is a one-dimensional $\kappa_i$ vector space through the action on the  left $V_i$, and
that the action of $\lambda\in \kappa_i^*$ through the right $V_i$ is the same as the 
action of $\sigma(\lambda)\in \kappa_i^*$ through the left $V_i$ for some $\sigma\in \Aut(\kappa_i)$ with $\sigma^2=1$ (because the $G$-invariants are preserved under swapping the $V_i^\vee$ factors).  
We have that $\sigma=1$ when $V_i$ is self-dual over $\kappa$, and $\sigma(\lambda)=\lambda^{p^{d/2}}$, where $d=[\kappa:\F_p]$, when $V_i$ is not self-dual over $\kappa_i$.  The functorial action of $\kappa_i^*$ on $H^2(V_i,\QZ)^G=(\wedge^2 V_i^\vee)^G$ agrees with the double action on 
$(V_i^\vee \tensor V_i^\vee)^G$ under the inclusion $(\wedge^2 V_i^\vee)^G\subset (V_i^\vee \tensor V_i^\vee)^G.$
Note when $\sigma=1$ the stabilizers of this action are $\{\pm 1\}$ and when $\sigma(\lambda)=\lambda^{p^{d/2}}$ the stabilizers are the $(p^{d/2}+1)$th roots of unity.

We choose a non-zero element $\rho\in H^2(V_i,\QZ)^G$.
If $V_i$ is not self-dual over $\kappa$,  then there are $p^{d/2}-1$ elements in the $\kappa$ orbit of $\rho$, all the non-zero elements of $H^2(V_i,\QZ)^G$,
as in this case $|(\wedge^2 V_i^\vee)^G|=p^{d/2}$.  If $V$ is self-dual  over $\kappa$ and $p=2$, there are $p^d-1$ elements in the $\kappa$ orbit of $\rho$, all the non-zero elements of $H^2(V_i,\QZ)^G$.   If $V$ is self-dual  over $\kappa$ and $p$ is odd, there are $(p^d-1)/2$ elements in the $\kappa$ orbit of $\rho$, 
and their linear $\F_p$-span must be all of $H^2(V_i,\QZ)^G$.
In every case,  we see that if $d_2^{0,2}(\rho)=0$, then $d_2^{0,2}(H^2(V_i,\QZ)^G)=0$ and hence $d_2^{0,2}(H^2(F,\QZ)^G)=0,$ which gives
$|E_3^{0,2}|=q_i^{\frac{e_i(e_i-\nu_i)}{2}}$, where $\nu_i$ is $-1$ if $(\wedge^2_{\kappa_i} V_i)^G\ne 0$, and is
$0$ if $(\wedge^2_{\kappa_i} V_i)^G= 0$.  

Now we consider the case when $d_2^{0,2}(\rho)\ne 0$.
Then $d_2^{0,2}$ is injective on $H^2(V_i,\QZ)^G$ (since $\rho$ above was an arbitrary non-zero element of $H^2(V_i,\QZ)^G$).
  For $\phi\in H^2(F,\QZ)^G$, we conclude that $d_2^{0,2}(\phi)=0$ if and only if, for all $j$, we have $p^*_j(\phi)=0$.  
  We note (using Lemma~\ref{lem-wedge}), that $\oplus_j p^*_j:  H^2(F,\QZ)^G \ra \bigoplus_j H^2(V_i,\QZ)^G$ is surjective.
Thus, we compute
$|E_3^{0,2}|=q_i^{\frac{e_i(e_i-1)}{2}}$.

Let $p$ be the characteristic of $V_i$. 
We can check that the map $H^2(V_i,\Z/p^2\Z) \ra H^2(V_i,\QZ)$ has a $G$-equivariant homomorphic section, since 
all classes in $H^2(V_i,\QZ)$ come from using bilinear forms $V_i\tensor V_i\ra\Z/p\Z$ as cochains (see the proof of Lemma~\ref{lem-wedge})
and symmetric forms (which are exactly the forms representing the trivial class in $H^2(V_i,\QZ)$)  can be checked to also give the trivial class
in $H^2(V_i,\Z/p^2\Z)$.  Thus %every $\rho\in H^2(V_i,\QZ)^G$ is the image of 
$i \colon H^2(V_i,\Z/p^2\Z)^G\ra H^2(V_i,\QZ)^G$ is surjective, and we have $i(\sigma)=\rho$ for some   $\sigma\in H^2(V_i,\Z/p^2\Z)^G$ that can be represented using a bilinear form as a cochain. 
Using the functoriality of the spectral sequence in the coefficients, we see that $d_2^{0,2} \circ i={d'}_2^{0,2}$,
where ${d'}_2^{0,2}$ is the differential in the analogous spectral sequences with $\Z/p^2\Z$ coefficients. (Here we use the fact that the natural map $H^1(V_i, \Z/p^2\Z) \to H^1(V_i, \QZ)$  is an isomorphism to identify the targets of the two differentials.)
Thus $d_2^{0,2}(\rho)=0$ if and only if ${d'}_2^{0,2}(\sigma)=0$.
% for some $\sigma\in H^2(V_i,\Z/p^2\Z)^G$ that can be represented using a bilinear form as a cochain. 

 From the properties of the edge map, ${d'}_2^{0,2}(\sigma)=0$ if and only if
$\sigma$ is in the image of $H^2(V\rtimes G,\Z/p^2\Z)$. 
We define $\mathcal{H}$ and $\ASp_{\F_p}(V_i)$ as in the introduction (with $4$ replaced by $p^2$, and 
we let the extension class $\sigma$ of $\mathcal{H}$  be the class associated to a fixed $G$-invariant symplectic form $\omega$ as in Section~\ref{S:Notation},
which can be represented using a bilinear form as a cochain). 
When $p$ is odd, we can use $\omega/2$ as the cochain for $\sigma$, and thus see that
$\Sp_{\F_p}(V_i)$ acts on $\mathcal{H}$ and so $\ASp_{\F_p}(V_i)\ra \Sp_{\F_p}(V_i)$ has a section.

Finally, we will show that $d_2^{0,2}(\rho)=0$ if and only if the action of $G$ on $V_i$ factors through $\ASp_{\F_p}(V_i)$. 
We have that $\sigma$ is in the image of $H^2(V\rtimes G,\Z/p^2\Z)$ if and only if there is a central extension $1\ra \Z/p^2\Z \ra H \stackrel{f}{\ra} V_i\rtimes G\ra 1$ such that $1\ra \Z/p^2\Z \ra f^{-1}(V_i) \stackrel{f}{\ra} V_i \ra 1$ has extension class $\sigma$.  
If there is such an extension, then $G$ acts on $f^{-1}(V_i)$ via lifting and conjugation, fixing $\Z/p^2\Z$ pointwise and respecting the action on $V_i$,
so the action on $G$ lifts to $\ASp_{\F_p}(V_i)$.
Conversely, if the action on $G$ lifts to $\ASp_{\F_p}(V_i)$, then $\mathcal{H}\rtimes G$ provides such a central extension.  
Note that an $\F_p$-symplectic action $G$ lifts to $\ASp_{\kappa_i}(V_i)$ if and only if it lifts to
$\ASp_{\F_p}(V_i)$ and it is $\kappa_i$-symplectic.

In particular, if $p$ is odd, then we always have $d_2^{0,2}(\rho)=0$, and we note the $\nu_i$ defined above is the same as $\epsilon_i$, and the lemma holds.  
If $p=2$, then in the $d_2^{0,2}(\rho)=0$ case, we have that $G$ acts through $\ASp_{\F_p}(V_i)$ and the $\nu_i$ defined above agrees with $\epsilon_i$, and in the $d_2^{0,2}(\rho)\ne 0$ case we have $\epsilon_i=1$, and in all cases the lemma holds.
\end{proof}

     \begin{lemma}\label{self-dual-w} Let $V_i$ be a self-dual representation. Assume $W_i^\tau=0$ and $\bG$ is an attainable $\bG$-extension.
      Then we have \[ \sum_{e_i=0}^{\infty} (-1)^{e_i}  M (e_i) Q_i (e_i) = w_{V_i} \] with the sum absolutely convergent.  \end{lemma}
     
     \begin{proof}We first consider the case when $V_i$ is not A-symplectic. Then by \cref{dual-M}  \[  M_i (e_i) =  \frac{1}{ \abs{ GL_{e_i}(\kappa_i)} } q_i^{\frac{e_i(e_i-\epsilon_i)}{2}} \]
     and $Q_i(e_i)= q_i^{ \binom{e_i}{2}}$ so,
      \[ \sum_{e_i=0}^{\infty} (-1)^{e_i}  M (e_i) Q_i (e_i) = \sum_{e_i=0}^{\infty} \frac{ q_i^{\frac{e_i(e_i-\epsilon_i)}{2} + \binom{e_i}{2}} }{ \abs{GL_{e_i} (\kappa_i)}} = \sum_{e_i=0}^{\infty} \frac{ q_i^{e_i\left(e_i- \frac{ 1+ \epsilon_i}{2}\right)   }}{ \abs{GL_{e_i} (\kappa_i)}}  .\]
     
We evaluate this first term using the $q$-exponential identity
\[ \sum_{e=0}^{\infty} (-1)^e \frac{ q^{ e^2 } u^e}{ \abs{GL_e(\mathbb F_q)}} = \prod_{j=1}^{\infty}  \frac{1}{ 1+ u q^{1-j } } \] applied with $u= q_i ^{ - \frac{1 +\epsilon_i}{2}}$. This series is absolutely convergent because $u<1$ (using that we are not in the $A$-symplectic case).
    Thus  \[ \sum_{e_i=0}^{\infty} (-1)^{e_i}  M (e_i) Q_i (e_i) = \prod_{j=1}^{\infty} \frac{1}{ 1+ q_i^{ - j - \frac{ \epsilon_i-1}{2}}} = w_{V_i},\] by definition of $w_{V_i}$.

We now consider the case when $V_i$ is A-symplectic. The same calculation of $M(e_i)$ applies, with $\epsilon_i=-1$, and we have $Q_i(e_i) = q_i ^{  \frac{e_i (e_i-3)}{2}}$ so
  \[ \sum_{e_i=0}^{\infty} (-1)^{e_i}  M (e_i) Q_i (e_i) =
   \sum_{e_i=0}^{\infty} \frac{ q_i^{\frac{e_i(e_i+1)}{2} +  \frac{e_i (e_i-3)}{2}} }{ \abs{GL_{e_i} (\kappa_i)}}  
   = \sum_{e_i=0}^{\infty} \frac{ q_i^{e_i(e_i-1)} }{ \abs{GL_{e_i} (\kappa_i)}} .\]  Using the same $q$-exponential identity as before, taking $u = q_i^{-1} $, we obtain
  \[ \sum_{e_i=0}^{\infty} (-1)^{e_i}  M (e_i) Q_i (e_i) =\prod_{j=1}^{\infty} \frac{1}{ 1+ q_i^{ - j }} = w_{V_i},\] again by definition, using that $\bG$ is attainable. 
 \end{proof}      
  \begin{lemma}\label{pair-w} Let $V_i,V_{i'}$ non-isomorphic dual representations of $G$. Assume $W_i^\tau = W_{i'}^\tau= 0$. Then 
  \begin{equation}\label{eq-oc-pair-w} \sum_{e_i=0}^{\infty} \sum_{e_{i'}=0}^\infty (-1)^{e_i + e_{i'}}  M_{i,i'} (e_i, e_{i'}) Q_i(e_i) Q_{i'} (e_{i'}) = w_{V_i} w_{V_i'} ,
  \end{equation}
and the sum is absolutely convergent.     
     \end{lemma}

\begin{proof}  We have $q_i =q_{i'}$. For simplicity, let $q=q_i =q_{i'}$ and $\kappa= \kappa_i = \kappa_{i'}$. Let $v= \frac{ \abs{ H^1(G, V_{i'} )}  }{ \abs{ H^1(G, V_i)}}$.
In this case, $V_i$ and $V_{i'}$ are non self-dual and thus not (A-)symplectic, so $Q_i (e_i) =q^{ \binom{e_i }{ 2}}  $ and $Q_{i'} (e_{i'}) = q^{ \binom{e_{i'}}{2}}$
%Plugging in \cref{odd-characteristc-M}, we obtain 
%\[  M_{I,i'} (e_i, e_{i'}) Q_i(e_i) Q_{i'} (e_{i'}) \] \[=   \frac{q^{ \binom{e_i}{2} + \binom{e_{i'}}{2}} }{ \abs{ GL_{e_i}(\kappa_i) }\abs{GL_{e_{i'}}(\kappa_{i'})} }    \frac{ \abs{ H^1(G, V_{i'} )}^{e_i -e_{i'} }  }{ \abs{ H^1(G, V_i)}^{e_i-e_{i'}}} \sum_{\alpha \in (W_i^\tau)^{e_i}, \alpha'\in  (W_{i'}^\tau)^{e_{i'}}} q^{{(e_i-\rank(\alpha))(e_{i'}-\rank(\alpha'))}}   .\]
Plugging in \cref{dual-M}, we obtain
\[  \sum_{e_i=0}^{\infty} \sum_{e_{i'}=0}^\infty (-1)^{e_i + e_{i'}}  M_{i,i'} (e_i, e_{i'}) Q_i(e_i) Q_{i'} (e_{i'})  = \sum_{e_i=0}^{\infty} \sum_{e_{i'}=0}^\infty   (-1)^{e_i + e_{i'}}   \frac{1}{ \abs{ GL_{e_i}(\kappa_i) }\abs{GL_{e_{i'}}(\kappa_{i'})} }   v^{e_i -e_{i'}  } q^{ \binom{e_i}{2} + e_i e_{i'} + \binom{e_i'}{2} } .\]
First we check that this sum is absolutely convergent. Each term in the sum is \[O \left( v^{e_i -e_{i'}} \frac{ q^{ \binom{e_i}{2} + e_i e_{i'}  +\binom{e_{i'}}{2}} }{ q^{ e_{i}^2 + e_{i'}^2}}\right) =  O \left( v^{ e_i-e_{i'} } q^{  - \frac{ (e_i-e_{i'})^2 + e_i + e_{i'} }{2}} \right) = O \left( q^{ - \frac{e_i +e_i'} {2} } \right)\]  because $v^ { e_i-e_{i'}} q^{ - \frac{ (e_i-e_{i'})^2}{2}}$ is bounded for any $v$, so the sum is absolutely convergent.

Next we observe that, by the definition of $v$, it is necessarily a power of $q$. If $v$ is a positive integer power of $q$, then we can arrange the sum as
% over $d$, $d'$ as 
\begin{equation}\label{odd-w-pair-rearranged} \sum_{e_{i'}=0}^{\infty} \frac{ (-1)^{e_{i'}} v^{-e_{i'}} q^{ \binom{e_{i'}}{2} } }{ \abs{GL_{e_{i'}(\kappa)}}}  \sum_{e_i=0}^{\infty} \frac{ (-1)^{e_i} v^{e_i}  q^{ e_i e_{i'} +\binom{e_i}{2} }  }{ \abs{GL_{e_i}(\kappa)}} =  \sum_{e_{i'}=0}^{\infty} \frac{ (-1)^{e_{i'}} v^{-e_{i'}} q^{ \binom{e_{i'}}{2} } }{ \abs{GL_{e_{i'}(\kappa)}}}  \prod_{j=1}^{\infty} (1 - v  q^{e_{i'}} q^{-j}) = 0\end{equation} because $v q^{e_{i'}}$ is always equal to $q^j$ for some $j$.
Symmetrically, if $v$ is a negative integer power of $q$, the sum vanishes. 
So the sum is nonvanishing only if $v=1$, i.e. if $ \dim H^1(G, V_{i'} ) = \dim  H^1(G, V_i)$. By definition, $w_{V_i}$ and $w_{V_i'}$ vanish when $v \neq 1$ and the identity \eqref{eq-oc-pair-w} is automatically satisfied.

%Setting $v=1$, we see that the sums over $r$ and $r'$ vanish unless $w=0$ or $w'=0$. Thus the sum vanishes if $\dim W_i^\tau >0 $ or $\dim W_{i'}^\tau >0$, and by definition $w_i(\tau)=0$ or $w_{i'}(\tau)=0$ in this case, so \eqref{eq-oc-pair-w} is again satisfied.

We are thus reduced to the case $v=1$. In this case, examining \eqref{odd-w-pair-rearranged}, we see that only the $e_{i'}=0$ term is nonvanishing, giving a value of $\prod_{j=1}^{\infty} (1-q^{-j})$. Correspondingly, in this case $w_{V_i} =w_{V_{i'}} = \prod_{j=1}^{\infty} (1-q^{-j})^{1/2}$, so \eqref{eq-oc-pair-w} is again satisfied.
\end{proof}

\subsection{Proof of Proposition~\ref{P:evalsum}}

\begin{proof}
By Lemma~\ref{L:Mefform} ,
\begin{align*}
& \sum_{ \underline{e}, \underline{f}}    M ( \underline{e},\underline{f})\frac{1}{ \abs{\Aut(\bG)}} (-1)^{ \sum_i e_i + \sum_i f_i} \prod_i Q_i (e_i)    
\\=& 
\frac{ \abs{G}  \abs{ H^2(G, \QZ) } }{  \abs{ H^1(G, \QZ)}  \abs{  H^3(G , \QZ )}\abs{\Aut(\bG)}  } \sum_{ \underline{e}, \underline{f}}    (-1)^{ \sum_i e_i + \sum_i f_i} \prod_i Q_i (e_i) \\
&\times  \prod_{ \substack{ i \in \{1,\dots ,r \} \\ V_i^\vee \not\cong V_j \textrm{ for any } j\neq i}}  {M}_i ( e_i) \prod_{\substack{ \{i, i'\} \subseteq \{1,\dots, r\}  \\  i \neq i' \\ V_i \cong V_{i'}^\vee  }} {M}_{i,i'} (e_i, e_{i'})  \prod_{i=1}^s {\eta}_i (f_i ).
\end{align*}

This sum now splits as a product over individual $V_i$'s, $N_i$'s and dual pairs, and we may apply one of \cref{non-abelian-w}, \cref{no-dual-w}, \cref{self-dual-w}, and \cref{pair-w} to evaluate each term, obtaining
\begin{equation}\label{E1}
 \sum_{ \underline{e}, \underline{f}}    M ( \underline{e},\underline{f})\frac{1}{ \abs{\Aut(\bG)}} (-1)^{ \sum_i e_i + \sum_i f_i} \prod_i Q_i (e_i)   = \frac{ \abs{G}  \abs{ H^2(G, \QZ) } }{  \abs{ H^1(G, \QZ)}  \abs{  H^3(G , \QZ )}\abs{\Aut(\bG)}  } \prod_{i=1}^r w_{V_i} \prod_{i=1}^s w_{N_i} \end{equation}
and, in particular, obtaining that the individual terms in the product are absolutely convergent and thus the entire sum is absolutely convergent.

Now note that we have the chain of identities (assuming all sums are absolutely convergent)
\begin{align}\label{E2}
 &\sum_{\bH \in I} \frac{ T_{\bH} }{ \abs{ \Aut(\bH)} } \frac{  \abs{H}  \abs{ H_2 (H, \mathbb Z ) } }{  \abs{ H_1(H, \mathbb Z) } \abs{ H_3(H, \mathbb Z) }}  
 =\sum_{\bH \in I} \sum_{ \pi \colon \bH\to \bG}  \frac{ T_{\pi} }{ \abs{ \Aut(\bH)} } \frac{  \abs{H}  \abs{ H_2 (H, \mathbb Z ) } }{  \abs{ H_1(H, \mathbb Z) } \abs{ H_3(H, \mathbb Z) }}\\\notag
=& \sum_{ \underline{e}, \underline{f}}  \sum_{\bH \in I} \sum_{ \substack{ \pi \colon \bH\to \bG\\ \textrm{type }\underline{e},\underline{f}} }  \frac{1}{ \abs{\Aut(\bG)}} (-1)^{ \sum_i e_i + \sum_i f_i} \prod_i Q_i (e_i)  \frac{ 1 }{ \abs{ \Aut(\bH)} } \frac{  \abs{H}  \abs{ H_2 (H, \mathbb Z ) } }{  \abs{ H_1(H, \mathbb Z) } \abs{ H_3(H, \mathbb Z) }}\\\notag
 =& \sum_{ \underline{e}, \underline{f}}    M ( \underline{e},\underline{f})\frac{1}{ \abs{\Aut(\bG)}} (-1)^{ \sum_i e_i + \sum_i f_i} \prod_i Q_i (e_i).
\end{align} 
(In the second equality, we spread the $T_{id}=1$ term out into $|\Aut(\bG)|$ terms, one for each isomorphism $\pi:\bG \ra \bG$.)
Next observe that, since the last sum is absolutely convergent, then the next-to-last sum is as well, because it is obtained by expanding out the sum defining $M( \underline{e}, \underline{f})$ which is a sum of nonnegative terms and thus preserves absolute convergence. The third-to-last and fourth-to-last sums are obtained from this by rearranging and grouping terms, respectively, and these operations preserve absolute convergence as well. 

Combining \eqref{E2} and \eqref{E1}, we deduce the proposition.
\end{proof} 

\section{Existential Theory}\label{s-existence}
In this section, we see some consequences of our results for the existence or non-existence of $3$-manifold groups with certain finite quotients but not others.
In Section~\ref{SS:genexist}, we give general necessary and sufficient conditions for when there exists a (closed, oriented) 3-manifolds with fundamental group with a surjection to $\bG$ that does not lift in certain ways,  determine what groups can be the level-$\C$ completion of a $3$-manifold group, and prove \cref{closure-characterization} characterizing the closure of the set of (profinitely completed) $3$-manifold groups in the space of all profinite groups.
In Section~\ref{SS:concrete}, we give examples to see how these results play out in certain cases.
In Section \ref{SS:obstructed}, we find all finite groups that are in the closure of the set of (profinitely completed) $3$-manifold groups, i.e. all finite groups that can be arbitrarily approximated by $3$-manifold groups.

\subsection{General necessary and sufficient conditions for existence of $3$-manifold groups}\label{SS:genexist}

\begin{definition}\label{def-spatial} Let $\bG$ be a finite oriented group. We say a pair consisting of an irreducible representation $V$ of $G$
(over a finite field) with field of endomorphisms $\kappa$ and a $\kappa$-subspace $W$ of $H^2(G,V)$ is \emph{spatial} if
\begin{enumerate}[(a)] % (a), (b), (c), ...

\item We have $\dim H^1 (G, V) \geq \dim H^1(G, V^\vee) + \dim W^{ \tau}$ where \[W^\tau =\{ \alpha \in W \mid \tau(\alpha \cup \beta) = 0 \textrm{ for all }\beta \in H^1(G, V^\vee) \}.\]

\item If $V$ has odd characteristic and is 
$\kappa$-symplectic,
%symplectically self-dual,
 then $\dim_{\kappa} H^1(G, V)$ is even.

\item If $V$ has even characteristic, is 
$\kappa$-symplectic,
%symplectically self-dual, 
and the map $G \to \Sp_\kappa(V)$ lifts to the affine symplectic group  $\ASp_\kappa(V)$, then $\dim_{\kappa} H^1(G, V) \equiv 2\tau ( c_{V} ) \mod 2$.

\end{enumerate}
\end{definition}
The term ``spatial" is used because these are the representations that will occur for 3-manifolds, as we will see in the following results.

\begin{remark}\label{R:easyspatial}
If the characteristic  of $V$ does not divide $G$, then $V,W$ is always spatial.  Also, if $V$ is a self-dual representation that is not $\kappa$-symplectic (e.g. a trivial representation), then $V,W$ is  spatial if $W=0$.
\end{remark}

Recall from Section \ref{ss-lc-notation} that $V_1,\dots, V_n$ is a finite list of irreducible representations of $G$ over prime fields, $W_i$ is a $\kappa_i$-subspace of $V_i$ for each $i$ (where $\kappa_i=\operatorname{End}_G(V_i)$), and $N_1,\dots N_m$ is a finite list of non-abelian finite simple $[G]$-groups. 
\cref{main-existence} is our main theorem on the existence of $3$-manifolds and gives a simple criterion that determines when there exists a $3$-manifold group with a surjection to $\bG$ not lifting to specified spaces of minimal extensions.

\begin{theorem}\label{main-existence} Let $\bG$ be a finite \Mdcorated/ group.
There exists a closed, oriented $3$-manifold $M$ and an \Mdcorated/ surjection $f\colon {\bf \pi_1(M) }\to \bG$ such that 
\begin{itemize}

\item  For each $i$ from $1$ to $n$, for each extension $1 \to V_i \to H \to G \to 1$ whose extension class lies in $W_i$, the map $f$ does not lift to a surjection from $\pi_1(M)$ to $H$.

\item For each $i$ from $1$ to $m$, the map $f$ does not lift from to a surjection from $\pi$ to $\Aut(N_i) \times_{ \Out(N_i)} G$.

\end{itemize}
if and only if, for each $i$ from $1$ to $n$, $(V_i,W_i)$ is spatial.

%\begin{itemize}
%
%\item We have $\dim H^1 (G, V_i) \geq \dim H^1(G, V_i^\vee) + \dim W_i^\tau $
%
%\item If $V_i$ has odd characteristic and is $\kappa_i$-symplectic, then $\dim_{\kappa_i} H^1(G, V_i)$ is even.
%
%\item If $V_i$ has even characteristic, is $\kappa_i$-symplectic, and the map $G \to \Sp_{\kappa_i}(V)$ lifts to $\ASp_{\kappa_i}(V)$, then $\dim_{\kappa_i} H^1(G, V_i) \equiv 2\tau ( c_{V_i} ) \mod 2$.
%
%\end{itemize}
Furthermore, in the ``if" direction, we can take $M$ to be a hyperbolic 3-manifold.
\end{theorem}

We give a group theory lemma first to clarify the argument.

\begin{lemma}\label{L:12implya}
Let $\bH$ be a profinite oriented group, $\bG$ a finite oriented group, and $f:\bH \ra \bG$ an oriented surjection.  If $V$ is an irreducible representation of 
$G$ over some $\F_p$, and $W\subset H^2(G,V)$ a $\operatorname{End}_G(V)$-subspace of extensions that $f$ cannot be lifted to, then  conditions (1), (2), (3), (4) from Theorem~\ref{pi1-properties} for $H,V$ (over the endomorphism field of $V$) imply conditions (a), (b), (c) in Definition~\ref{def-spatial} for $G,V$. %Also, $\dim H^1(G, V) = \dim H^1( H, V)$.
\end{lemma}
\begin{proof}
From Lemma~\ref{lem-even_ei} and the condition on $f$ not lifting, we obtain
%that $f$ does not extend to $V_i\rtimes G$, we have
\begin{equation}\label{pi1-observation} \dim H^1(G, V) = \dim H^1( H, V). \end{equation}
and that $ W \to H^2(G, V) \to H^2 (H ,V) $
is injective.  As usual let $W^\tau\subset W$ be the elements that, via cup product and $\tau_G$, pair to $0$ with every element of $H^1(G,V^\vee)$.
By condition (2) for $H$, for each $\alpha$ in $W^\tau$ there must exist $\beta \in H^1(H, V^\vee)$ with $ \tau_{H}(\alpha \cup \beta)\neq 0$. This defines a surjection $H^1( H, V^\vee) \to (W^\tau)^\vee$. By the definition of $W^\tau$, we have that $H^1(G, V^\vee)$ must be in the kernel of this surjection, so \[ \dim W^\tau + \dim H^1(G, V^\vee) = \dim (W^\tau )^\vee+ \dim H^1(G, V^\vee) \leq \dim H^1(H, V^\vee) \] \[ = \dim H^1( H, V) = \dim H^1( G, V) \] by condition (1) and \eqref{pi1-observation}, giving (a).  Using \eqref{pi1-observation},  conditions (3) and (4) imply (b) and (c).
\end{proof}

\begin{proof}[Proof of \cref{main-existence}] We use separate arguments for the ``if" and ``only if" directions.
The ``only if" direction is implied by \cref{pi1-properties} and Lemma~\ref{L:12implya}.

For ``if", we use \cref{localized-counting}, which gives a formula for the limit of the expected number of oriented surjections $f\colon {\bf \pi_1(M)} \to \bG$  as above for a random 3-manifold. 
Under the conditions of the proposition, we can check from the chart of the $w_{V_i}$ and $w_{N_i}$
in Section~\ref{ss-lc-notation} that the limiting expectation is positive, and thus
%If this limit is positive, then surely
 a 3-manifold with such a surjection must exist. Maher has shown that a random Heegaard splitting of a fixed genus is hyperbolic with probability $\ra 1$ as $L\ra\infty$ \cite[Theorem 1.1]{Maher2010}. Thus, the limiting expectation of the number of such surjections from hyperbolic 3-manifolds is positive, and we have a hyperbolic $M$ with a surjection as desired.
%For the limit to be positive, it suffices that $w_{V_i}>0$ for all $i$ and $w_{N_i}>0$ for all $i$. For $w_{N_i}$, this is automatic, as $w_{N_i}$ is either one or %a power of $e$.  
%If $V_i$ is self-dual, then by assumption
%\[ \dim W_i^\tau \leq \dim H^1 (G, V_i) - \dim H^1(G, V_i^\vee)  = 0 ,\]
%and, if $\epsilon_i=-1$, then by assumption the parity condition is satisfied, so regardless of $\epsilon_i$, $w_{V_i}$ is a convergent product of positive terms %and thus is positive.
%If $V_i$ is not dual to $V_j$ for any $V_j$, then by assumption \[ \frac{ \abs{ Q_i^\tau} \abs{ H^1(G, V_i^\vee)}}{ \abs{H^1(G, V_i)}} \leq 1,\] and again in %this case $w_{V_i}$ is a convergent product of positive terms. 
%
%If $V_i$ and $V_j$ are dual with $i \neq j$ , then
%\[  0 \leq  \dim W_i^\tau  \leq \dim H^1(G, V_i) - \dim H^1(G, V_j) \leq - \dim W_j^\tau \leq 0 \]
%so in fact all of these are $0$, and again in this case $w_{V_i}$ is a convergent product of positive terms.
\end{proof}

Using \cref{main-existence} and Lemma~\ref{L:Linnicecase}, 
we can describe the level-$\mathcal C$ completions of $\pi_1(M)$ for any $\mathcal C$.

\begin{definition} We say an irreducible representation $V$ of $G$ over $\mathbb F_p$ is \emph{level-$\mathcal C$} if $V \rtimes G$ is level-$\mathcal C$. (Note that this is not necessarily equivalent to $V$ being level-$\mathcal C$ as an abstract group.)
For such a $V$, let $H^2(G, V)^{\mathcal C}$ consist of extension classes such that the corresponding extension of $G$ by $V$ is level-$\mathcal C$. 
\end{definition}

\begin{prop}\label{level-C-existence} Let $\mathcal C$ be a finite set of finite groups and 
$\bG$ a finite level-$\mathcal C$ oriented group. 
%For $\tau\colon H^3(G, \QZ)\to \QZ$, let $H^2(G, V)^{\mathcal C, \tau}$ be the space of $\alpha \in H^2(G, V)^{\mathcal C}$ such that $\tau(\alpha \cup \beta) =0$ for all $\beta \in H^1(G, V^\vee)$.
There exists a closed, oriented $3$-manifold $M$ such that ${\bf \pi_1(M)^{\mathcal C}} \cong {\bG}$ if and only if %there exists  $\tau\colon H^3(G, \QZ)\to \QZ$ such that 
for each level-$\C$ irreducible representation $V$ of $\bG$ over any $\F_p$, the pair $(V, H^2(G, V)^{\mathcal C})$ is spatial. 
\end{prop}

\begin{proof}[Proof of \cref{level-C-existence} ]  This follows from combining \cref{main-existence} and Lemma~\ref{L:Linnicecase}, setting $W_i = H^2(G,V_i)^\mathcal C$ for all level-$\C$ $V_i$.
\end{proof}

Finally,  to consider all levels at once, our next goal is to prove Theorem~\ref{closure-characterization}, which gives, in the space of all relevant profinite groups, the closure of the set of $3$-manifold groups.  First we have a lemma to help clarify that we have the correct space of profinite groups.

\begin{lemma}\label{maximal-quotient-finite} For $X \in \Prof$, and $\mathcal C$ a finite set of finite groups, the level-$\mathcal C$ completion $X^{\mathcal C}$ of $X$ is a finite group. \end{lemma}

\begin{proof} 
Let $\C_i$ be the set of all quotients of groups in $\C$ of order at most $i$.
We will show by induction that $X^{\C_i}$ is finite for all $i$ and conclude that $X^{\C}$ is finite.  Because $\C_1$ consists only of the trivial group, $X^{\C_1}$ is trivial and thus finite.

Thus we assume $X^{\C_{i-1}}$ is finite. For $Q$ a level-$\C_i$ quotient of $X$, we must have $Q^{ \C_{i-1}}$ a quotient of $X^{\C_{i-1}}$ so there are finitely many possibilities for $Q^{\C_{i-1}}$. By \cref{one-lower-semisimple}, we have $Q\ra  Q^{\C_{i-1}}$ semisimple, and thus $Q$ is a fiber product of finitely many minimal extensions of $Q^{C_{i-1}}$, all level-$\C$ quotients of $X$.  By Lemma~\ref{L:Linnicecase}, there are only finitely many level-$\C$ minimal extensions of $Q^{\C_{i-1}}$, and by the definition of $\Prof$, there are only finitely many quotients of $X$ isomorphic to one of these finitely many extensions, so there are finitely many possibilities for $Q$, and thus $X^{\C_i}$, the inverse limit of all level-$\C_i$ quotients $Q$ of $X$, is finite, completing the induction.
\end{proof}

\begin{proof}[Proof of \cref{closure-characterization}] 

For ``if", assume that there exists $\tau \colon H^3(G, \mathbb Q/\mathbb Z) \to \mathbb Q/\mathbb Z$  satisfying the conditions (1),(2),(3),(4) of \cref{closure-characterization} for every $V$. Let $\mathcal C$ be a finite set of finite groups. We have an induced map $\tau_{G_\mathcal C} \colon H^3(G^{\mathcal C}, \QZ) \to H^3( G, \QZ) \to \QZ$.
Let $V$ be an irreducible level-$\C$ representation of $G^{\mathcal C}$ defined over $\F_p$. 
By Lemma~\ref{L:12implya}, $(V,H^2(G^\C,V)^\C)$ is spatial, and by \cref{level-C-existence}, this implies there exists a (closed, oriented) $3$-manifold  with ${\bf \pi_1(M)} \cong \bG^\C$.  So $G$ is in the closure as desired.

For ``only if" direction, for each finite set $\C$ of finite groups,  we have a $3$-manifold $M$ with $\pi_1(M)^\C\cong G^\C$, such that the orientation on $M$ gives some  $\tau_{G^\C} \colon H^3(G^{\mathcal C}, \QZ) \to \QZ$.   We consider all such $\tau_{G^\C}$ coming from manifolds at each level-$\C$, and we have an inverse system of non-empty finite sets, which is non-empty, and thus there is a $\tau \colon H^3(G, \QZ) \to \QZ$ such that for every $\C$ the induced map $H^3(G^{\mathcal C}, \QZ) \to \QZ$ comes from a manifold (with an isomorphism $\pi_1(M)^\C\cong G^\C$).
 Thus by Proposition~\ref{level-C-existence}, for any level-$\C$ irreducible representation $V$ of $G^\C$ over any $\F_p$, we have $(V,H^2(G^\C,V)^\C)$ is spatial, using the orientation induced from $\tau$.  

Now let $V$ be an irreducible representation of $G$ over some $\F_p$ with endomorphism field $\kappa$. We will show conditions (1)-(4) of 
\cref{closure-characterization} for $V$ over $\kappa$.
For (1), let $A$ be the image of $G$ inside $\Aut(V)$ and take  $\mathcal C$ to consist of $ V \rtimes A$ and $V^\vee \rtimes A$. Then $A$ is a quotient of $G^{\mathcal C}$ and thus $V$ descends to a representation of $G^{\mathcal C}$. Note $V \rtimes G^{\mathcal C} = (V \rtimes A) \times_A (G^{\mathcal C})$ is level-$\mathcal C$, 
%the representation $V$ is admissible,
 and similarly for $V^\vee \rtimes G^{\mathcal C}$.  
Thus  \[ \dim H^1(G^\mathcal C, V) \geq \dim H^1(G^{\mathcal C}, V^\vee) + \dim H^2 (G^{\mathcal C}, V)^{\mathcal C,\tau} \geq H^1(G^{\mathcal C}, V^\vee) \] \[\geq \dim H^1(G^{\mathcal C}, V) + \dim H^2 (G^{\mathcal C}, V^\vee)^{\mathcal C,\tau}  \geq \dim H^1(G^{\mathcal C}, V) \] so all inequalities appearing are equalities and hence
\begin{align}\label{E:spatialcon} \dim H^1(G^{\mathcal C}, V) = \dim H^1(G^{\mathcal C}, V^\vee)  
\quad \textrm{and}\quad
H^2 (G^{\mathcal C}, V)^{\mathcal C,\tau} = H^2 (G^{\mathcal C}, V^\vee)^{\mathcal C,\tau} = 0
\end{align}
and Lemma~\ref{lem-even_ei} implies $\dim H^1(G, V) =\dim H^1(G, V^\vee).$
% and \[ \dim H^1(G, V) = \dim H^1(G^{\mathcal C}, V) = \dim H^1(G^{\mathcal C}, V^\vee)  =\dim H^1(G, V^\vee).\]

For (2), fix $\alpha \in H^2(G, V) = \lim_{ U \subseteq G} H^2(G/U, V^U)$, with the limit taken over open normal subgroups, with $\alpha\neq 0$.  Take $B=G/U$ to be a finite quotient of $G$ from which the class $\alpha$ arises.  Thus there is an extension $H$ of $B$ by $V$ with class $\overline{\alpha}$, such that the pullback to $G$ of $\overline{\alpha}$ is $\alpha$. Take $\mathcal C$ to consist of $ V \rtimes A$, $V^\vee \rtimes A$, and $H$. Then $B$ is level-$\mathcal C$ and thus is a quotient of $G^{\mathcal C}$, giving a class $\overline{\alpha} \in H^2(G^{\mathcal C}, V)$ that pulls back to $\alpha$. This class is nontrivial as it pulls back to the nontrivial class $\alpha$, and it lies in $H^2(G^{\mathcal C}, V)^{\mathcal C}$ because the extension group is given by $H \times_B G^{\mathcal C}$. 
From \eqref{E:spatialcon}, $H^2(G^{\mathcal C},V)^{\mathcal C, \tau}=0$, so by definition there is $\overline{\beta} \in H^1(G^{\mathcal C}, V^\vee)$ with $\tau(\overline{\alpha} \cup \overline{\beta})\neq 0$. The pullback of $\overline{\beta}$ to $G$ then gives the desired $\beta$.

For (3), take $\mathcal C$ to consist of $V \rtimes A$, and for (4), to consist of $V \rtimes A$ together with the image of $G$ in the affine symplectic group of $V$. Then $H^1(G, V) = H^1(G^{\mathcal C}, V)$ by Lemma~\ref{lem-even_ei} and so the parity property for $G^{\mathcal C}$ implies the desired parity property for $G$.

 Note if (1), (2),(3),(4) are satisfied for a representation $V$ over a field $\kappa$, then they are satisfied for $V \otimes_\kappa \kappa'$ for any field extension $\kappa'$.  Since every absolutely irreducible representation $V'$ over a finite field $\kappa'$ is $V \tensor_{\kappa} \kappa'$
 for some irreducible representation $V$ over $\F_p$ with endomorphism field $\kappa$, we have (1), (2), (3), and (4) for every absolutely irreducible $V$.

Also, if (1),(2),(3),(4) are satisfied for a representation $V$ over a field $\kappa$, then they are satisfied for $V$ viewed as a representation over a subfield $\kappa'$.   Every irreducible representation $V$ over $\kappa'$ is an absolutely irreducible representation over its endomorphism field $\kappa'$.  
Thus we have (1), (2), (3), and (4) for any  irreducible $V$ over a finite field.
\end{proof}

\subsection{Concrete corollaries on non-existence and existence of $3$-manifolds}\label{SS:concrete}
We now give some simple concrete consequences of \cref{main-existence}, our main existence result.
We begin with some negative results, showing that 3-manifolds which have a certain group as a quotient but don't have certain other groups as a quotient do not exist.

\begin{prop}\label{negative-20} Let $M$ be a  closed, oriented  3-manifold. Suppose that $G=\mathbb Z/5 \rtimes \mathbb Z/4$ is a quotient of $\pi_1(M)$, where the generator of $\mathbb Z/4$ acts on $\mathbb Z/5$ by multiplication by $2$. Then $(\mathbb Z/5)^2 \rtimes \mathbb Z/4$ is a quotient of $\pi_1(M)$, where the generator of $\mathbb Z/4$ acts on $\mathbb Z/5$ with eigenvalues $2$ and $3$. \end{prop}

\begin{proof}
Let $G$ act on $V=\Z/5$ with a generator of $\Z/4$ acting by multiplication by $3$.  
Then $H^1(G,V)=0$ but $H^1(G,V^\vee)=\Z/5$ (e.g. by  \cref{lem-even_ei}), so  $V,W=0$ is not spatial, failing condition (a). By \cref{main-existence},  any surjection $\pi_1(M)\ra G$ lifts to  a surjection to
$V\rtimes G$.
\end{proof}

\begin{prop}\label{negative-S3} Let $M$ be a  closed, oriented 3-manifold. Suppose that $S_3$ is a quotient of $\pi_1(M)$. Then one of $S_4, S_3 \times S_2$, or $\mathbb Z/3 \rtimes \mathbb Z/4$ (the semidirect product taken with respect to the unique nontrivial action) is a quotient of $\pi_1(M)$. \end{prop}

\begin{proof}
Since $H^2(S_3,\F_2)$ is cyclic, it only has a single nontrivial 2-torsion class $c$,  and we divide into cases based on whether $\tau(c)=0$.
If $\tau(c)=0$, then the trivial representation $V=\F_2$, with $W=H^2(S_3,\F_2)$, is not spatial for $(S_3,\tau)$, since $V=(V)^\vee$ and 
$W^\tau=H^2(S_3,\F_2)=\F_2$, and so condition (a) fails.  In this case,  by \cref{main-existence}, any surjection $\pi_1(M)\ra S_3 $ lifts to a surjection to one of the extensions of $S_3$ by $\F_2$, which are $S_3\times S_2$ and
$\Z/3 \rtimes \Z/4$.  

If $\tau(c)\neq 0$,  we let $V=\F_2^2$ and with action of $S_3$ through the identification $S_3=\Sp_{\F_2}(V)$.
Recall from the proof of Proposition~\ref{cvnonzero}, that we have a splitting of $\ASp_{\F_2}(V_1)\ra \Sp_{\F_2}(V)$ and that $H^1(S_3,V)=0$.  
By Remark~\ref{R:fromS3}, we have that $c_{V}=c$.  Since $\tau(c_{V})\ne 0$ and $H^1(S_3,V)=0$,  
we have that $V,W=0$ is not spatial,  as  condition (c) fails.
By \cref{main-existence}, 
any surjection $\pi_1(M)\ra S_3 $ lifts to a surjection to $V\rtimes S_3\cong S_4$.
\end{proof}

 The following result may be the simplest example that (1) involves only two groups and (2) follows from the parity results and thus not purely from Poincar\'e duality and Euler characteristic arguments. It does involve slightly  larger groups than the previous two examples, although the group $G$ occurs in multiple contexts as the Mathieu group $M_9$ and as $PSU_3(\mathbb F_2)$.
 
 \begin{prop}\label{negative-72}  
 Let $Q_8$ be the $8$-element quaternion group.
Let $V$ be the two-dimensional irreducible representation of $Q_8$ over $\mathbb F_3$. 
 Let $G= V \rtimes Q_8$, and $H =  V^2 \rtimes Q_8$.
 Any closed, oriented $3$-manifold such that $G$ is a quotient of $\pi_1(M)$ also has $H$ as a quotient of $\pi_1(M)$. \end{prop}

\begin{proof} 
 We have $H^1(G,V)=\F_3$ (by  \cref{lem-even_ei} since $V$ is absolutely irreducible),  but $V$ is $\F_3$-symplectic,  so $V$,$W= 0$ is not spatial, 
 failing condition (b). Then apply \cref{main-existence}.
\end{proof} 

The next example is similar, but involves parity of a non-projective representation and thus doesn't follow directly from semicharacteristic theory.

\begin{prop}\label{negative-non-projective} 
Let $V$ be the standard representation of $\SL_2(\mathbb F_3)$ and 
 $G= V \rtimes \SL_2(\mathbb F_3) $, and  $H= V^2 \rtimes \SL_2(\mathbb F_3) $.
Any closed, oriented $3$-manifold such that $G$ is a quotient of $\pi_1(M)$ also has $H$ as a quotient of $\pi_1(M)$. \end{prop}

\begin{proof}  
As a  representation of the subgroup $Q_8$ of $\SL_2(\mathbb F_3)$, $V$ is absolutely irreducible,  and nontrivial and thus $H^0(Q_8, V) =0$, and since $Q_8$ has order prime to $3$, $H^i(Q_8, V)=0$ for all $i>0$. Since $Q_8$ is a normal subgroup of $\SL_2(\mathbb F_3)$ (with quotient $\mathbb Z/3$), there is a spectral sequence computing $H^{p+q} ( SL_2(\mathbb F_3), V)$ from $H^p ( \mathbb Z/3, H^q( Q_8, V))=H^p(\mathbb Z/3, 0) =0$,  and so $H^i (\SL_2(\mathbb F_3), V)=0$ for all $i$.   Note $V$ is a symplectic representation of $G$.
We then have $H^1(G, V)=\F_3$,  by  \cref{lem-even_ei}.  So $V,W=0$, is not spatial, failing condition (b), and the result is implied by  \cref{main-existence}.
%In particular, $H^1(\SL_2(\mathbb F_3), V)=0$.
%Thus by Lemma~\ref{lem-even_ei} we have 
%Thus,  \[ H^1(G, V) = H^0 (SL_2(\mathbb F_3), H^1( V, V)) = H^0(SL_2(\mathbb F_3), V \otimes V^\vee) = \mathbb F_3\] since $V$ is absolutely irreducible.
%Since $V$ is a symplectic representation of $G$,  we then have that
%$G$ fails the conditions of \cref{main-existence} taking $V_1= V$ and $W_1= 0$. It follows that there does not exist a 3-manifold $M$ and map $\pi_1(M) \to %G$ which does not lift to a  map $\pi_1(M) \to V\rtimes G=H$. 
\end{proof} 

Now we turn to specific existence results that follow from \cref{main-existence}.  
We obtain many existence results from representations that are automatically spatial (see Remark~\ref{R:easyspatial}).  
For a set of primes $S$, an \emph{$S$-group} is a finite group whose order is a product of primes in $S$.

\begin{cor}\label{C:goodcharexist}
Let $S$ be a finite set of primes, $G$ a finite group whose order is not divisible by primes in $S$.
There exists a closed, oriented $3$-manifold $M$ such that $\pi_1(M)$ has a surjection $\pi_1(M)\ra G$ that does not lift to any surjection $\pi_1(M) \ra H\rtimes G$
where $H$ is an $S$-group.
\end{cor}

The manifold $M$ with surjection $\pi_1(M) \to G$ produced by \cref{C:goodcharexist} determines a covering space $\tilde{M} \to M$. By construction, $\tilde{M}$ has a free action of $G$ and $\pi_1(\tilde{M})$ has no surjection to any $S$-group. In particular, its mod $p$ homology vanishes for any $p \notin S$, which unless $S =\emptyset$ forces it to be a rational homology 3-sphere. This strengthens the main result of \cite{CooperLong}, which states that $G$ has a free action on a rational homology 3-sphere. 

Pardon \cite{Pardon1980} proved a similar result in higher dimensions, showing that a finite group $G$ of order prime to $p$ has a free action on a simply-connected mod $p$ homology $n$-sphere for $n>3$ odd. As pointed out in \cite[Proposition 5]{ADEM2019}, the same methods could be used to produce free actions of $G$ on  $3$-manifolds that are %not-necessarily-simply-connected 
mod $p$ homology 3-spheres (but not necessarily simply-connected). \cref{C:goodcharexist} provides a stronger existence result since we can take $|S|>1$, and we also restrict non-abelian quotients of $\pi_1$.

Corollary~\ref{C:goodcharexist} appears interesting even for groups as small as $G=(\Z/2)^3$.

\begin{cor}\label{C:self-dual non symplectic}
Let $G$ be a finite group and $V$ irreducible self-dual representation of $G$ over a prime field, not symplectic over its endomorphism field (e.g. $V$ a trivial representation $\F_p$). There exists a closed, oriented $3$-manifold $M$ such that $\pi_1(M)$ has a surjection $\pi_1(M)\ra G$ that does not lift to any surjection $\pi_1(M) \ra V \rtimes G$.
\end{cor}

%In \cref{C:self-dual non symplectic} we could repl $V\rtimes G$, where 
%\begin{prop}
%Let $\bG, \underline{V},\underline{N}$ be as above and for each $i$ let $E_i$ be an extension of $G$ by $V_i$ and $F_i$ be an extension of $G$ by $N_i$ 
%(with the appropriate $G$ action), and similarly let 
%\end{prop}
We can also give existence results for 3-manifolds whose fundamental groups display some unusual properties.

\begin{prop}\label{spherical-approximation-exists} Let $G$ be the (finite) fundamental group of a spherical $3$-manifold and let $n$ be a natural number. There exists a closed, oriented hyperbolic 3-manifold $M$ such that $G$ is a quotient of $\pi_1(M)$ and every finite group of order $\leq n$ that is a quotient of $\pi_1(M)$ is a quotient of $G$. \end{prop}

%For example, there is a hyperbolic 3-manifold $M$ such that $SL_2(\mathbb F_5)$ is a quotient of $\pi_1(M)$, and the only other quotients of $\pi_1(M)$ of order $\leq 10^{10^{10}}$ are $PSL_2(\mathbb F_5)$ and the trivial group.  

\begin{proof} Let $\mathcal C$ be the set of all groups of order $\leq n$. We apply \cref{level-C-existence}, showing there exists a 3-manifold $M$ such that $\pi_1(M)^{\mathcal C} \cong G$ if and only if $G$ satisfies certain conditions. Because we can take $M$ to be hyperbolic in \cref{main-existence}, we can take $M$ to by hyperbolic in the ``if" direction of \cref{level-C-existence}, so it remains to check the conditions. 

To check the conditions, we use the fact that there exists a manifold whose fundamental group is $G$, that being the spherical one, so the conditions are necessarily satisfied by the ``only if" direction of \cref{level-C-existence}. \end{proof}

While \cref{spherical-approximation-exists} follows quickly from our results, such an $M$ could be constructed via suitable Dehn surgery on a hyperbolic knot in $S^3/G$  (following \cite[Remark 2.4]{DunfieldGaroufalidis}). However, we can also produce 3-manifolds whose fundamental groups ``approximate" arbitrarily well certain specific finite groups which are not themselves the fundamental groups of 3-manifolds. For these it is not at all clear how such $3$-manifolds could be constructed by without our probabilistic theorem.  
For a set of primes $S$,  the \emph{pro-$S$ completion} of a group $G$ is the inverse limit of all finite quotients whose order is a product of primes in $S$.
 
\begin{prop}\label{fake-3-manifolds}

Let $a,b,c,d$ be relatively prime odd integers. Let $Q(8a,b,c)$ be the semidirect product $(\mathbb Z/a \times \mathbb Z/b \times \mathbb Z/c) \rtimes Q_8$, where $Q_8$ is the $8$-element quaternion group,  we let $\chi_i$ be the three non-trivial homomorphisms $Q_8\ra \{\pm 1\}$,  and $Q_8$ acts
on $\mathbb Z/a \times \mathbb Z/b \times \mathbb Z/c$ by $g(x,y,z)=(x^{\chi_1(g)},y^{\chi_2(g)},z^{\chi_3(g)})$.

Then for all finite sets $S$ of primes including $2$ and all primes dividing $a,b,c,d$,  there exists a closed, oriented 3-manifold $M$ such that the $G:=Q(8a,b,c)\times \Z/d$ is the pro-$S$ completion of $\pi_1(M)$.

However, unless two of $a,$ $b$, and $c$, are equal to $1$, $G$ is not itself the fundamental group of a closed, oriented 3-manifold. \end{prop}

We recall that a group $G$ has \emph{periodic cohomology} of period  $d\ne 0$ if for all $n>0$ we have $H^n(G,M)\cong H^{n+d}(G,M)$ for all $G$-modules $M$.
By \cite[Ch. VI, Thm.  9.1]{Brown1982},  $G$ has periodic cohomology of period $4$ if and only if $H^4(G, \mathbb Z) \cong \mathbb Z/\abs{G}$.
By the universal coefficient theorem such a group has  $H^3(G,\QZ)\cong \mathbb Z/\abs{G},$ and it is this latter condition that we will mostly use.

\begin{proof}There are finitely many isomorphism classes of finite simple groups with order divisible by only primes in $S$,  so there are finitely many isomorphism classes of finite simple $[Q(8a,b,c)]$-groups with order divisible by only primes in $S$. The group $Q(8a,b,c)$ is the pro-$S$ completion of $\pi_1(M)$ if and only if $\pi_1(M)$ surjects onto $Q(8a,b,c)$ and this surjection doesn't lift to an extension of $Q(8a,b,c)$ by any of these groups. So it suffices to check that there is a $\tau$ satisfying the conditions of \cref{main-existence} when the $V_i$ are the set of irreducible representations of $Q(8a,b,c)$ over $\F_p$, where $p$ varies over the primes of $S$  and $W_i=H^2(Q(8a,b,c),V_i)$.

Milnor \cite[p. 628]{Milnor1957} describes $Q(8a,b,c)$ with a presentation, and in his presentation,  our $\Z/a\Z$ is generated by $y^4$ and the quaternion group is generated by $x,y^a$. %$x,y^n$
   From \cite[Theorem 3]{Milnor1957},  we see that $G$ has periodic cohomology of period $4$,  and hence
$H^3 ( G, \QZ) \cong \mathbb Z/(8abcd)$.

We fix $\tau\colon H^3 ( G, \QZ) \cong \mathbb Z/(8abcd)\to \QZ$ an injection. Let us check that $(G, \tau)$ satisfies the conditions of \cref{main-existence} for each irreducible representation $V$ over $\F_p$ of $G$, i.e. that $V$ is spatial taking $W= H^2(G,V)$. 

Let $\kappa$ be the field of endomorphisms of $V$ and $p$ the characteristic of $\kappa$. Let us first restrict attention to the special case $p \nmid bc$. We split our argument further depending on whether the subgroup $\mathbb Z/b \times \mathbb Z/c$ acts trivially on $V$.  This is a normal subgroup of $Q(8a,b,c)$, with quotient $Q(8a,1,1) \times \Z/d=Q_{8a} \times \Z/d$.

If $\mathbb Z/b \times \mathbb Z/c$ acts trivially on $V$  then the map $H^i ( Q_{8a} \times \Z/d, V) \to H^i ( G, V) $ is an isomorphism,  because the kernel of the quotient  $G\to Q_{8a} \times \Z/d$ has order $bc$ prime to the characteristic of $V$. 
The same is true for the dual representation and, in characteristic $2$, for the $2$-torsion part of the cohomology with coefficients in $\QZ$.  
The existence of the $3$-manifold $S^3/( Q_{8a} \times \Z/d)$ \cite[Theorem 2]{Milnor1957} implies that
  $V$ is spatial for $Q_{8a}\times \Z/d$, and hence
 $V$  is also spatial for $G$. (Note that both the $\tau :H^3(Q_{8a} \times \Z/d,\QZ)\ra \QZ$ inherited from our choice above and the $\tau$ coming from the fundamental class of $S^3/ ( Q_{8a} \times \Z/d)$ are injections.  So in either case, 
a  class is nontrivial if and only if its image under $\tau$ is nontrivial.  With this guaranteed, the spatial condition does not depend on the choice of $\tau$, so it is not necessary to check the $\tau$'s are the same.)

On the other hand, if $\mathbb Z/b \times \mathbb Z/c$ acts nontrivially on $V$ then \[H^i (G, V) = H^i (Q_{8a}\times \Z/d, V^{ \mathbb Z/b \times \mathbb Z/c} ) = H^i(Q_{8a}\times \Z/d, 0)=0\] for all $i$ since the subgroup $\mathbb Z/b \times \mathbb Z/c$ is normal so it does not fix any vectors when it acts nontrivially on an irreducible representation. The same is true for $V^\vee$. Examining the conditions that make a representation spatial, we see that $H^i(G, V)= H^i(G, V^\vee)=0$ for all $i$ ensures that $V$ is spatial unless $p=2$, $V$ is affine symplectic, and $\tau(c_V)\neq 0$.

Combining these two paragraphs, we see that if $p\nmid bc$, then $V$ is spatial unless $p=2$, $\mathbb Z/b\times \mathbb Z/c$ acts nontrivially on $V$, $V$ is affine symplectic, and $\tau(c_V)\neq 0$. 

The isomorphism class of the group $G$ is invariant under permutations of $a,b,c$, since $Q_8$ has automorphisms that permute the three characters $\chi_1,\chi_2,\chi_3$. Thus, by symmetry, it follows that if $p \nmid ac$, then $V$ is spatial unless $p=2$, $\mathbb Z/a \times \mathbb Z/c$ acts nontrivially on $V$, $V$ is affine symplectic, and $\tau(c_V) \neq 0$, and a similar statement follows for $ab$.

We are now ready to prove that any $V$ is spatial. First assume $p$ odd. Then since $a,b,c$ are relatively prime, $p$ divides at most one of $a,b,c$, and thus doesn't divide the product of the other two. Since $p \neq 2$, it follows that $V$ is spatial.

Next assume $p=2$. Then $p\nmid a,b,c$, so if $V$ is not spatial it follows that $V$ is affine symplectic,  $\tau(c_V)\neq 0$, and $\mathbb Z/a \times \mathbb Z/b$, $\mathbb Z/a \times \mathbb Z/c$, and $\mathbb Z/b \times \mathbb Z/c$ all act nontrivially on $V$.  We will show that this is impossible.
Let $\kappa'$ be an extension of $\kappa$ containing all the $abcd$'th roots of unity. Then $V \otimes_{\kappa} \kappa'$ remains irreducible and satisfies all the other conditions from this paragraph.

We will now describe the representation $V$ as an induced representation.  The central subgroup $\mathbb Z/2$ of $ Q_8$ acts trivially on $\mathbb Z/a \times \mathbb Z/b \times \mathbb Z/c$, so $G$ contains a normal subgroup $N = \mathbb Z/a \times \mathbb Z/b \times \mathbb Z/c \times \mathbb Z/2 \times\Z/d$ with quotient $(\mathbb Z/2)^2$.
By construction of $\kappa'$, the restriction of $V$ to $N$ contains a one-dimensional irreducible representation $U$ of $N$ (with character 
$\chi_U: N \ra (\kappa')^*$) 
so $V$ admits a map from the induced representation $\operatorname{Ind}_N^{G} U$.  Either this map is an isomorphism,  or the orbit of the character 
$\chi_U$ under the conjugation action of $\mathbb Z/2 \times \mathbb Z/2$ has less than four elements. The second case happens only when $\chi_U$ factors through $\mathbb Z/a \times \mathbb Z/2 \times\Z/d$,   $\mathbb Z/b \times \mathbb Z/2 \times\Z/d$, or $\mathbb Z/ c \times \mathbb Z/2 \times\Z/d$. If $\chi_U$ factors through $\mathbb Z/a \times \mathbb Z/2 \times\Z/d$,  then $\mathbb Z/b \times \mathbb Z/c$ acts trivially on $\operatorname{Ind}_N^{G} U$ and thus on $V$, contradicting our assumption, and similarly for the other possible factorizations, so the second case cannot occur, and $V \cong   \operatorname{Ind}_N^{G} U$.

Using this, we will check that $\tau(c_V)=0$, contradicting another assumption. 
Since the kernel $G\ra Q_8$ has odd order, 
 $H^3( Q_8, \QZ) \to H^3( G , \QZ)$ is an isomorphism on the $2$-torsion part, so the pullback map $H^3( G, \QZ) \to H^3(Q_8, \QZ)$ using the semidirect product structure must be an isomorphism on the $2$-torsion part, and it suffices to show that the pullback of $c_V$ to $H^3(Q_8, \QZ)$  is trivial. 

We do this by applying \cref{pi1-properties} to $S^3/ Q_8$. The representation $V$, restricted to $Q_8$, is the induced representation of the trivial representation from the central $\mathbb Z/2$, so $H^i ( S^3/Q_8, V) \cong H^i ( S^3/ (\mathbb Z/2), {\kappa})$. This cohomology group has dimension $1$ for both $i=0$ and $i=1$, so the sum of the dimensions is $2$, which is even, so the integral of $c_V$ over $S^3/Q_8$ vanishes.  Using the spectral sequence computing the cohomology of $S^3/Q_8$
from $H^p(Q_8,H^q(S_3,\QZ))$,
 the integration map $H^3( Q_8, \QZ) \to \QZ$ is injective, so the pullback of $c_V$ is zero, as desired.

Since the conditions are satisfied in every case, such an $M$ exists for every $S$.

On the other hand, if $G$ were itself the fundamental group of a 3-manifold, then by Perelman's Elliptization Theorem, $G$ would be a subgroup of $SO(4)$ with no fixed points on $S^3$. %Such groups are known to be cyclic groups or central extensions of dihedral, tetrahedral, octahedral, or icosahedral groups by groups of cyclic even order. 
However,  $G$ 
%by its center is $(\mathbb Z/a\times \mathbb Z/b \times \mathbb Z/c) \rtimes (\mathbb Z/2)^2$,  
is not such a subgroup, unless two of $a,b,c$ are $1$ %, in which case it is a product of dihedral group and a cyclic group
 \cite[Theorem 3]{Milnor1957}.
\end{proof}

\subsection{Classification of unobstructed groups}\label{SS:obstructed}

We can generalize the results of the last subsection.  We call a finite group \emph{unobstructed} if it lies in the closure of the set
\[ \{ \widehat{ \pi_1(M)} \mid M \textrm{ a closed, oriented 3-manifold},\}\]
and \emph{obstructed} otherwise.  
\cref{closure-characterization} gives a necessary and sufficient condition for groups to be unobstructed in terms of their cohomology.
If a finite group $G$ is obstructed, 
by examining its cohomology groups to find for which representations $V$ 
which condition of \cref{closure-characterization} fails for which values of $\tau$, 
one can produce an explicit list of representations $V_i$ and subspaces $W_i$ such that the condition of \cref{main-existence} fails, hence produce an explicit list of extensions of $G_i$ by $V_i$ such that every surjection $\pi_1(M) \to G$ lifts to at least one of these extensions.
%Thus for obstructed $G$ there is a finite list of finite groups $H_1,\dots, H_n$ such that every 3-manifold fundamental group with $G$ as a quotient admits $H_i$ as a quotient for some $i$.   
The previous results give explicit examples of obstructed and unobstructed groups. In this subsection, we will give a complete classification of unobstructed groups,
using the group cohomological conditions of \cref{closure-characterization}.

Certainly every $G$ that is itself the fundamental group of a 3-manifold is unobstructed. These are classified as a consequence of the geometrization theorem. There is one further family of unobstructed groups constructed in \cref{fake-3-manifolds}. In this subsection we will show  that these are the only examples.

\begin{lemma}\label{unobstructed-periodic} Let $G$ be an unobstructed finite group. Then $G$ has periodic cohomology of period 
$4$. \end{lemma}

\begin{proof} 
Let $V_1,\dots, V_n$ be all the irreducible representations of $G$ of characteristic dividing $\abs{G}$. Let $W_i = H^2(G, V_i)$. By \cref{main-existence}, since $G$ is unobstructed, there exists a manifold $M$ and a surjection $\pi_1(M) \to G$ which does not lift to any extension of $G$ by an irreducible representation of $G$ of characteristic dividing $\abs{G}$. Fix one such $M$, and note that $\operatorname{ker} ( \pi_1(M) \to G)^{ab} $ is a finitely-generated abelian group.
Thus $\operatorname{ker} ( \pi_1(M) \to G)^{ab} $ is a finite abelian group of order prime to $\abs{G}$,  or else it would have a non-trivial elementary abelian $p$-group quotient for some $p\mid \abs{G}$ and $\pi_1(M) \to G$ would lift to an extension by one of the $V_i$.

Let $\tilde{M}$ be the $G$-covering of $M$ induced by this surjection. Then $\tilde{M}$ is a closed 3-manifold with fundamental group this kernel. Thus $H_1( \tilde{M})$ is finite of order prime to $\abs{G}$. Thus $H^0(\tilde{M},\mathbb Z) = H^3(\tilde{M}, \mathbb Z) = \mathbb Z$, $H^1(\tilde{M},\mathbb Z) = 0$, and $H^2(\tilde{M}, \mathbb Z)$ is finite of order prime to $\abs{G}$.

There is a spectral sequence whose second page is $E^{p,q}_2 = H^p(G, H^q(\tilde{M}, \mathbb Z))$ converging to $H^{p+q}(M,\mathbb Z)$. By the above calculations, we see that $E^{p,q}$ vanishes for $q\neq 0,2,3$ and for $q=2$ vanishes for $p>0$.  Because of this vanishing, the only possibly nontrivial differentials (i.e. differentials whose source and target both may be nonzero) are  
\[ d^{p,3}_4\colon E^{p,3}_4 \to E^{p+4,0}_4\
\quad \textrm{and} \quad
 d^{0,2}_3 \colon E^{0,2}_3 \to E^{3,0}_3 \] although $d^{0,2}_3$ must vanish because the source is finite of order prime to $\abs{G}$ and the target is $\abs{G}$-torsion. 

Since $d^{p,3}_4$ are the only nonvanishing differentials, we have $E^{p,q}_4 = H^p (G, H^q ( \tilde{M}, \mathbb Z))$, and $E^{p,3}_\infty = \ker d^{p,3}_4$, and $E^{p+4,0}_{\infty}= \coker d^{p,3}_4$. This gives a long exact sequence
\[ % H^3 (G, H^0(\tilde{M},\mathbb Z) )\to 
H^3(M, \mathbb Z) \to H^0(G, H^3(\tilde{M}, \mathbb Z) ) \to H^4 ( G, H^0(\tilde{M}, \mathbb Z))
 \to H^4 (M, \mathbb Z) 
%\to H^1(G, H^3(\tilde{M}, \mathbb Z) )
\]
Using $H^4(M, \mathbb Z) =0$ and the fact that $G$ acts trivially on $H^3(\tilde{M}, \mathbb Z)$, we see that $H^4(G, \mathbb Z) = H^4(G, H^0(\tilde{M}, \mathbb Z))$ is the cokernel of the map $H^3(M,\mathbb Z) \to H^3( \tilde{M}, \mathbb Z)$, which is equal to the pullback map on cohomology. 

Since $M$ and $\tilde{M}$ are closed oriented 3-manifolds and $\tilde{M} \to M$ is a degree $\abs{G}$ covering, this map is the multiplication-by-$\abs{G}$ map $\mathbb Z \to \mathbb Z$, so its cokernel is $\mathbb Z/\abs{G}$, as desired.
\end{proof}

The following lemma, which gives an alternate way of seeing that $c_V$ is non-trivial, will help us show that many finite groups with periodic cohomology are obstructed.

\begin{lemma}\label{dihedral-parity-facts} Let $D_n$ be the dihedral group with $2n$ elements for $n$ odd. 
Let $V$ be a representation of $D_n$ that is an induced representation of any faithful character of $\mathbb Z/n$ over a characteristic $2$ finite field. 
Then $V$ is $\kappa$-symplectic and the induced map $D_n \to \Sp_\kappa(V)$ lifts to $\ASp_\kappa(V)$.
Further,  %$H^3 ( D_n, \QZ)$ has Sylow $2$-subgroup of order $2$,  and 
$c_V$ is  non-trivial in $H^3 ( D_n, \QZ)$ and pulls back from %the unique non-trivial class in
$H^3 ( \Z/2\Z, \QZ)$.
% and the pullback of the class $c_V$ along this lift is the pullback of the generator of $H^3 ( \mathbb Z/2, \QZ)$ under the map $D_n \to \mathbb Z/2$. 
\end{lemma}

%It follows
%from \cref{pi1-properties}
% that, for $M$ a 3-manifold, $ \pi_1(M) \to D_n$ a surjection, and $V$ a faithful representation of $D_n$ over a field of characteristic $2$, if $\dim H^1( \pi_1(M), %V) $ is even then the integral of the pullback of the generator of $H^3( \mathbb Z/2, \mathbb Q/\Z)$ vanishes.

\begin{proof}% A key fact is that, restricted to the subgroup $\mathbb Z/n$ of $D_n$,  the representation $V$ must be the sum of a faithful character and its inverse.
 Because $V$ is two-dimensional, it is symplectic if and only if its determinant is the trivial representation.  %Using the key fact, 
 The determinant of $V$  has trivial action of the subgroup $\mathbb Z/n \subset D_n$,  and thus the action of $D_n$ factors through $\Z/2\Z$.  
 Since the determinant of $V$ is one-dimensional over a characteristic $2$ field,  the action of $D_n$ on the determinant must be trivial, and $V$ is $\kappa$-symplectic.

 To check that $V$ lifts to the affine symplectic group, it suffices to check that the pullback to $H^2(D_n, V^\vee)$ of the extension class in $H^2(\Sp_\kappa(V), V^\vee)$
of the affine symplectic group vanishes. In fact, we will check that $H^i(D_n, V^\vee)$ vanishes for all $i$ when $n\geq 3$. To see this, by the spectral sequence, it suffices to check that $H^i ( \mathbb Z/n, V^\vee)$ vanishes for all $i$. Since $\mathbb Z/n$ has order prime to $2$, its cohomology with coefficients in a characteristic $2$ representation vanishes in degrees above $0$, and in degree $0$ is equal to the $\mathbb Z/n$-invariants, which vanish when 
%our key fact and 
$n\geq 3$.  For $n=1$, we can factor the map $D_1\ra D_3\ra \Sp_\kappa(V)$ and thus see it lifts to the affine symplectic group.
 
% The pullback of the class $c_V$ is a two-torsion element of $H^3( D_n, \QZ)$. 
The natural map $D_n \to \mathbb Z/2$ has kernel a group of order prime to $2$ and thus the induced map $H^i(\mathbb Z/2 , \QZ) \to H^i (D_n, \QZ)$ is an isomorphism on $2$-power torsion for all $i$. 
 We will check $c_V$ is nontrivial by restricting to a subgroup $S\cong\mathbb Z/2$ of $D_n$, and checking it remains nontrivial there, where $V$ becomes isomorphic to the regular representation of $\mathbb Z/2$ over $\kappa$. To do this, it suffices to find a 3-manifold with a homomorphism $\pi_1(M) \to \mathbb Z/2$ such that $\dim H^0( M, V) + \dim H^1(M, V)$ is odd. An example is provided by $M = \mathbb R \mathbb P^3$, taking the homomorphism to be the unique isomorphism $\pi_1(\mathbb R \mathbb P^3) \cong \mathbb Z/2$. Since $V$ is the regular representation over $\kappa$, we have $H^0(M, V) = H^0(S^3, \kappa) = \kappa$ and $H^1(M,V) = H^1(S^3, \kappa) =0$, so the sum of their dimensions is indeed odd. 
 Thus by \cref{universal-class-exists} $c_V$ is non-trivial in $H^3(S,\QZ)$,  where it is pulled back from $c_V\in   H^3(D_n,\QZ)$, and hence the latter must be non-trivial.
  \end{proof}

Finally we have the classification of unobstructed groups.

\begin{prop}\label{obs-class} 
The unobstructed finite groups are exactly the finite subgroups of $SO(4)$ acting freely on $S^3$ 
and the groups of the form $Q(8a,b,c) \times \mathbb Z/d$  where $a,b,c,d$ are odd and pairwise relatively prime. \end{prop}

Some works toward the classification of finite groups that appear as the fundamental groups of 3-manifolds highlighted $Q(8a,b,c)\times \mathbb Z/d$ as examples of groups that could not be ruled out by their methods \cite{Milnor1957, Lee}.  
Indeed, these were the only groups that could not be ruled out as fundamental groups of 3-manifolds before Perelman's Geometrization Theorem.
 \cref{obs-class} gives a heuristic explanation for this,  as it shows these groups are arbitrarily close to $3$-manifold groups in our topology, and they are the only such finite groups that aren't $3$-manifold groups themselves.

\begin{remark}\label{R:obsS}
If $S$ is a finite set of primes,  and $G$ a finite group that is the maximal $S$-group quotient of a closed,  oriented $3$-manifold, then 
by Remark~\ref{R:easyspatial} and \cref{main-existence}, $G$ must be unobstructed.
\end{remark}

\begin{proof} 
First assume $G$ is unobstructed.
 By \cref{unobstructed-periodic}, $G$ has periodic cohomology with period 4. Such groups are classified.  See, for example \cite[between Proposition 6.9 and Theorem 7.10]{Nicholson2021} for a convenient list whose notation we will use. 
Of these classes given in \cite{Nicholson2021}, those listed as (i)', (iii)', (iv'), and (v') are known to be finite subgroups of $SO(4)$ acting freely on $S^3$ (see, e.g., \cite[Theorem 2]{Milnor1957}). It suffices to show that $G$ cannot be any of the remaining ones, except $Q(8a,b,c) \times \mathbb Z/d$.

Generalizing our previous definition of $Q(8a,b,c)$, for odd coprime positive integers $a,b,c$, and $n\geq 3$, we define
$Q( 2^n a, b,c) = (\mathbb Z/a \times \mathbb Z/b \times \mathbb Z/c) \rtimes Q_{2^n}$ where $Q_{2^n}$ is the 
generalized
quaternion group of order $2^n$ of presentation $\langle x,y| x^{2^{n-1}}=y^4=1, x^{2^{n-2}}=y^2, yxy^{-1}=x^{-1}\rangle$,
and $Q_{2^n}$ acts on $\mathbb Z/a \times \mathbb Z/b \times \mathbb Z/c$ as follows.  
Let $\chi_i$ be the three non-trivial homomorphisms $Q_{2^n}\ra \{\pm 1\}$,  with $\chi_1(x)=1$, and the action is given by
  $g(x,y,z)=(x^{\chi_1(g)},y^{\chi_2(g)},z^{\chi_3(g)})$.
% acting on $\mathbb Z/a$, $\mathbb Z/b$, $\mathbb Z/c$ by the three nontrivial quadratic characters, acting on $\mathbb Z/a$ by the one whose kernel is cyclic

The remaining groups with periodic cohomology of period 4  from the list in \cite{Nicholson2021} are as follows (in each case,  we  also take the product with a cyclic group $C$ of coprime order):
\begin{enumerate}
\item[(ii)'] $D_n$,  the dihedral group of order $2n$, for $n\geq 3$ odd (which Nicholson and Milnor call $D_{2n}$).

\item[(vi)'] $P''_{48n}$,   the extension of the binary octahedral group $\tilde{O}$ (SmallGroup(48,28)) by the cyclic group $C_n$ such that the extension has cyclic Sylow $3$-subgroup and the action of $\tilde{O}$ on $C_n$ is through the order $2$ quotient of $\tilde{O}$ and sending $x\in C_n$ to $x^{-1}$,
 for $n\geq 3$ and odd.

\item[(vii)'] $Q( 2^n a, b,c)$,  for 
 $n\geq 3$ and $a,b,c$ odd coprime integers with $b>c$. 
\end{enumerate}

We will show that each group $G$ on this list except those in the last case with $n=3$ are obstructed.
Let $V_m$ be a representation of $D_m$ that is an induced representation of any faithful character of $\mathbb Z/m$ over a characteristic $2$ finite field. 
For $m>1$, we can check that $V_m$ is irreducible, and  by \cref{dihedral-parity-facts},
 $V_m$ is affine symplectic.
We will find maps $G\ra D_m$ with $H^1(G, V_m)=0$.
Since 
 $H^3(G,\QZ)\cong \mathbb Z/\abs{G},$ we have that   $H^3(G,\QZ)$ has a unique non-trivial $2$-torsion element (which we will identify as $c_{V_m}$).
Using \cref{dihedral-parity-facts} and \cref{closure-characterization}, we will prove that $G$ is obstructed.

%and then observing that the pullback of the generator of $H^3(\mathbb Z/2, \QZ)$ to $H^3(G, \QZ)$ is nontrivial, so by \cref{finite-pd-facts} its integral must be %nontrivial, getting a contradiction.

In case (ii)',  $G= D_n\times C$ and we take the projection $G\ra D_n$ and let $V=V_n$. 
By \cref{dihedral-parity-facts},  %$V$ is $\kappa$-symplectic, the induced map $G \to \Sp_\kappa(V)$ lifts to $\ASp_\kappa(V)$,  and
 $H^3 (G, \QZ)$ has Sylow $2$-subgroup of order $2$,  and $c_{V}$ is  non-trivial in $H^3 ( G, \QZ)$.
We can see that $H^1 (D_n\times C, V)=0$ because, restricting to the normal subgroup $\mathbb Z/n$, we already have $H^*(\mathbb Z/n, V)=0$ because $V$ has characteristic coprime to $n$ and no $\mathbb Z/n$-invariants (as $n>1$).
By Theorem~\ref{closure-characterization} (4),  if $G$ was unobstructed for a particular $\tau$,  then $\tau$ would vanish on the 
$2$-torsion of $H^3 (G, \QZ)$.  However, then with $V=\F_2$,  since $H^2(G,\F_2)=\F_2$, we see that Theorem~\ref{closure-characterization}(2) cannot be satisfied.
Thus we conclude that $G$ is obstructed.

%On the other hand, the pullback of the generator $H^3(\mathbb Z/2, \QZ) \to H^3(D_n, \QZ)$ is nontrivial because the kernel of $D_n \to \mathbb Z/2$ has order %prime to $2$ and so that map is an isomorphism on $2$-power torsion.

In case (vi), $G=P''_{48n}\times C$.  We note that both $D_{3n}$ and $\tilde{O}$ have unique normal subgroups with quotient $S_3$,  and we can check the fiber product 
$D_{3n} \times_{S_3} \tilde{O}$ satisfies the definition of $P''_{48n}$.
We take the projection $G\ra D_{3n}$ and let $V=V_{3n}$.  
%The normal subgroup $\mathbb Z/n$ of $D_{3n}$ maps trivially to $S_3$ and thus remains normal in $G$.
We have a normal subgroup $\Z/n\Z$ of $D_{3n}$ and $G$.
As in the previous case, we have $H^*( \mathbb Z/n, V)=0$ for all $i$ and thus  $H^1 (G, V)=0$. 
%To check that $H^1 (G, V)$ vanishes, it suffices to check that $H^i( \mathbb Z/n, V)$ vanishes for $i=0$ and $i=1$. We have $H^i( \mathbb Z/n, V)=0$ for all $i$, %again because $V$ is a representation with no invariants (as $n>1$) of a group $\mathbb Z/n$ of order prime to the characteristic..

Next, we will show that the pullback of the generator of $H^3( \mathbb Z/2, \QZ)$ to $H^3(G, \QZ)$ is nonzero. The map $G \to \tilde{O}$ has kernel prime to $2$ and thus the map $H^3(\tilde{O}, \QZ) \ra H^3(G, \QZ)$ is an isomorphism on $2$-power torsion. %, and $H^3(G, \QZ)$ has a unique non-trivial $2$-torsion element.
%so it suffices to show the pullback of the generator of $H^3( \mathbb Z/2, \QZ)$ to $H^3 ( \tilde{O}, \QZ)$ is nonzero.   
The kernel of the map $\tilde{O}\ra \Z/2\Z$ is $\tilde{T}$, the binary tetrahedral group.  In the spectral sequence computing 
$H^n ( \tilde{O}, \QZ)$ from $H^p(\Z/2,H^q(\tilde{T},\QZ))$,  we have that $E_2^{1,1}=H^1(\Z/2,\Z/3\Z)=0$ and
$E_2^{0,2}=H^0(\Z/2,0)=0$.  Thus $E^{3,0}=H^3( \mathbb Z/2, \QZ)$ receives no non-trivial differentials and 
$H^3( \mathbb Z/2, \QZ) \ra H^3 ( \tilde{O}, \QZ)$ is an injection.  
So by \cref{dihedral-parity-facts}, $c_V$ is non-zero in $H^3(G, \QZ)$.  

By Theorem~\ref{closure-characterization} (4),  if $G$ was unobstructed for a particular $\tau$,  then since $H^1 (G, V)=0$, we would have $\tau(c_V)=0.$
Further, since $c_V$ is the unique non-trivial $2$-torsion element of  $H^3(G, \QZ)$,  then $\tau$ vanishes on the 
$2$-torsion of $H^3 (G, \QZ)$.  However, then with $V=\F_2$,  since $H^2(G,\F_2)=H^2(\tilde{O},\F_2)=\F_2$, we see that Theorem~\ref{closure-characterization}(2) cannot be satisfied.
Thus we conclude that $G$ is obstructed.

In case (vii)',  $G=Q( 2^n a, b,c)\times C$.  
We have $D_b$ as a quotient of $Q( 2^n a, b,c)$  using a map $ (\mathbb Z/a \times \mathbb Z/b \times \mathbb Z/c) \rtimes Q_{2^n} \to \mathbb Z/b \rtimes \mathbb Z/2$ which sends $Q_{2^n}$ to $\mathbb Z/2$ under the quadratic character $\chi_2$ by which $Q_{2^n}$ acts on $\mathbb Z/b$. 
Let $V=V_b$.
We have a normal subgroup
 $\mathbb Z/b$ of  $G$. As before, $H^*(\mathbb Z/b, V)=0$ and thus $H^1 (G, V) =0$. (Here we use $b> c \geq 1$ to ensure $V$ has no $\mathbb Z/b$-invariants, and later it will ensure that $V$ is irreducible.)

Next we will show that the pullback of the generator of $H^3 ( \mathbb Z/2, \QZ)$ under the above map $G \to \mathbb Z/2$ is nontrivial. Since the projection $G\to Q_{2^n}$ has kernel prime to $2$, it suffices to check that the pullback of the generator of $H^3(\mathbb Z/2, \QZ)$ to $H^3( Q_{2^n}, \mathbb Z/2)$ via $\chi_2$ is nontrivial.
One can do this directly,  but we give a different argument.  Let $K=\ker \chi_2,$ and note when $n>3$, we have $K\cong Q^{2^{n-1}}$.
Let $W=\operatorname{Ind}_{K}^{Q_{2^n}} \F_2$.  The action of $Q_{2^n}$ on $W$ factors through $\chi_2$.  
By \cref{dihedral-parity-facts} (for $n=1$), we have that $W$ is affine symplectic.  We then have $c_W\in H^3(S^3/  Q_{2^n},\QZ)$. 
We can compute $H^0(S^3/  Q_{2^n},W)=H^0(S^3/K,\F_2)=\F_2$ and $H^1(S^3/  Q_{2^n},W)=H^1(S^3/K,\F_2)=H^1(K,\F_2)=\F_2^2$ (using $n> 3$).
By \cref{universal-class-exists}, we then note that the integral of $c_W$ over $S^3/  Q_{2^n}$ is non-trivial so $c_W\in  H^3(Q_{2^n},\QZ)$ must be non-trivial and it factors through $H^3(\mathbb Z/2, \QZ)$ via $\chi_2^*$.  We then conclude that the map $G \to \mathbb Z/2$ above induces an injection
$H^3 ( \mathbb Z/2, \QZ)\ra H^3 ( G, \QZ)$.  It follows that for $n>3$, we have that $c_V$ is non-trivial in $H^3 ( G, \QZ)$. 

%can also use \cref{pi1-properties}(4), taking $M = S^3/  Q_{2^n}$ and looking at the cohomology of the pullback of the regular representation of $\mathbb Z/2$. %The cohomology of the regular representation in degree $i$ matches the cohomology of $S_3 / Q_{2^{n-1}}$ in degree $i$ since $Q_{2^{n-1}}$ is the kernel of %the character. For $i \leq 2$, this matches the cohomology of $Q_{2^{n-1}}$. We have $H^0 (Q_{2^{n-1}}, \mathbb Z/2) = \mathbb Z/2$ and $H^1( Q_{2^{n-1}}, 
%\mathbb Z/2) = (\mathbb Z/2)^2$ so the sum of their dimensions is odd, showing that the integral of the pullback of the generator of  $H^3(\mathbb Z/2, \QZ)$ mu
%st be nontrivial, as desired.

By Theorem~\ref{closure-characterization} (4),  if $G$ was unobstructed for a particular $\tau$,  then since $H^1 (G, V)=0$, we would have $\tau(c_V)=0.$
Further, since $c_V$ is the unique non-trivial $2$-torsion element of  $H^3(G, \QZ)$,  then $\tau$ vanishes on the 
$2$-torsion of $H^3 (G, \QZ)$.  However, then with $V=\F_2$,  since $H^2(G,\F_2)=H^2(Q_{2^n},\F_2)
$,  and the latter has a non-trivial $2$-torsion subgroup (from the universal coefficient theorem),
we see that Theorem~\ref{closure-characterization}(2) cannot be satisfied.
Thus we conclude that $G$ is obstructed.

Having eliminated all cases but $Q(8a,b,c)\times C$, we see that an unobstructed group must be of the claimed types of the proposition.  In \cref{fake-3-manifolds},
we showed $Q(8a,b,c)\times C$ is unobstructed, and the finite subgroups of $SO(4)$ acting freely on $S^3$ are $3$-manifold groups themselves, which proves the proposition.
\end{proof}

\section{Probabilistic Theory}

\cref{localized-counting} gives a complete description of the limiting distribution of $\widehat{\pi_1(M_{g,L})}$ and hence has many probabilistic consequences, some of which we give in this section.  In particular we can use it to give the distribution of the maximal $p$-group or nilpotent class $s$ quotient of $\pi_1(M_{g,L})$ (see Section~\ref{SS:pgroup}) or the distribution of $H_1(M,\Z_p)$ with its torsion linking pairing (see Section~\ref{SS:torsionlinking}).  In Section~\ref{SS:prove12} we prove \cref{intro-prob}
on the existence of a limiting distribution, and in Section~\ref{S::VPFBN} we show the limiting probability of a $G$-cover with positive first Betti number is $0$.

\subsection{Maximal $p$-group and nilpotent class $c$ quotients}\label{SS:pgroup}
If one is interested in, for example, the distribution of the maximal $p$-group quotient of $\pi_1(M_{g,L})$, then one can apply \cref{localized-counting} 
and obtain formulas that simplify substantially from the general case.

\begin{prop}\label{p-group-probability}

Let $p$ be a prime and let $s$ be a natural number or $\infty$. 
Let $P$ be a finite $p$-group of nilpotency class $\leq s$. The limiting probability that $P$ is is the maximal  quotient of $\pi_1(M_{g,L})$ that is a $p$-group of nilpotency class $\leq s$ (in the limit as $L$ goes to $\infty$ and then $g$ goes to $\infty)$, is equal to
\[ \frac{  \abs{H_2 (P, \mathbb Z) } \abs{P} }{ \abs{H_1(P,\mathbb Z)} \abs{ \Aut(P) }} \frac{N_s(P) }{ \abs{H_3(G, \mathbb Z)}} \prod_{j=1}^{\infty} (1+p^{-j} ) ^{-1}\]
  where $N_s(P)$ is the number of $\tau\colon H^3( P, \QZ) \to \QZ $ such that for all nonzero $\alpha \in H^2(P,\mathbb Z/p)$ whose induced extension $1 \to \mathbb Z/p \to \tilde{P} \to P \to 1$ has nilpotency class $\leq s$, there exists $\beta \in H^1(P, \mathbb Z/p)$ such that $ \tau(\alpha \cup \beta) \neq 0$.
  
  \end{prop}

We note that $N_\infty(P)=0$ unless $P$ is cyclic or generalized quaternion $Q_{2^k}$ for $k\geq 3$ by Remark~\ref{R:obsS}.  
%So cyclic and generalized quaternion groups are the only possible maximal $p$-group quotients of manifolds.

The same argument works to show the generalization of \cref{p-group-probability} for the maximum $S$-group of nilpotence class $s$ quotient for $S$ a finite set of primes, where $S$-groups are defined to be the groups whose order is divisible only by the primes lying in $S$,
and the resulting probabilities are as above but with a $\prod_{p\in S}$ before the product over $j$.

\begin{proof} We will apply \cref{localized-counting} with $V_1=\mathbb Z/p$, and $W_1$ the subspace of $H^2(P, \mathbb Z/p)$ corresponding to extensions $1 \to \mathbb Z/p \to \tilde{P} \to P \to 1$ of nilpotency class $\leq s$, and no other $V_i$s or $G_i$s. Then a surjection $\pi_1(M) \to P$ is an isomorphism beween the maximal $p$-group quotient of $\pi_1(M)$ of nilpotency class $\leq s$ if and only if it does not lift to any extension $\tilde{P}$ of $P$ by $\mathbb Z/p$ of nilpotency class $\leq s$, since any nontrivial extension of $p$-groups factors through an extension of $\mathbb Z/p$.

%So we may apply \cref{main-prob} 
We sum \cref{localized-counting} and sum over each orientation of $P$,   equivalently, each $\tau \in H^3(P,\QZ)^\vee$. 
Because $V_1$ is symmetrically self-dual, we have $w_{V_1}=0$ unless $W_1^\tau=0$ and $w_{V_1} =\prod_{j=1}^{\infty} (1+p^{-j} ) ^{-1}$ if $W_1^\tau=0$.
Furthermore, by definition, $W_1^\tau=0$ if and only if, for all nonzero $\alpha \in H^2(P,\mathbb Z/p)$ whose induced extension $\tilde{P}$ of $P$ has nilpotency class $\leq s$, there exists $\beta \in H^1(P, \mathbb Z/p)$ such that $ \tau(\alpha \cup \beta) \neq 0$.
Thus the sum over $\tau$ of $w_{V_1}(\tau)$ is equal to $\prod_{j=1}^{\infty} (1+p^{-j} ) ^{-1}$ times the number of $\tau$ satisfying that condition, as desired.
\end{proof}

\subsection{The distribution of the torsion linking pairing}\label{SS:torsionlinking}

We can give an even simpler formula in the $s=1$ case, i.e. with the $p$-part of the abelianization of $\pi_1(M)$. It turns out that $\tau$, in this setting, carries the information of the torsion linking pairing, and the equidistribution result will be particularly convenient to state in terms of this pairing. We begin with a lemma:

\begin{lemma}\label{linking-injective} Let $G$ be a finite abelian group. Consider the pairing $H^1(G , \mathbb Q/\Z) \tensor H^1(G , \mathbb Q/\mathbb Z) \to H^3(G, \QZ)$ that is defined by taking the Bockstein map $H^1(G, \mathbb Q/\mathbb Z) \to H^2(G, \mathbb Z)$
(associated to $1\ra \Z\ra \mathbb{Q}\ra\QZ\ra 1$)
 in the second variable and then taking the cup product.

This map is symmetric, and the induced map $\Sym^2 ( H^1(G, \mathbb Q/\mathbb Z)) \to H^3(G, \mathbb Q/\mathbb Z)$ is injective. \end{lemma}

\begin{proof} The symmetry follows from the standard argument that the Bockstein map satisfies the Leibniz rule.

For injectivity, consider a nonzero class  $\alpha\in \Sym^2 ( H^1(G, \mathbb Q/\mathbb Z)) $. If we write $G$ as 
$\prod_i \Z/n_i$ for $n_1\mid n_2\mid \cdots$,
%$\prod_i \mathbb Z/p^{e_i}$, 
then $\alpha$ consists of, for each pair $i,j$, an element $a_{ij} \in \mathbb Z/{\gcd(n_i,n_j)}$. Let us check that $G$ has a cyclic subgroup, restricted to which, this class remains nontrivial. This is clearly the case if $a_{ii}\neq 0$ for any $i$, so we may assume $a_{ii}=0$ for all $i$ and thus that $a_{ij} \neq 0$ for some distinct $i,j$. Without loss of generality $n_i \mid n_j$, and then pulling $\alpha$ back to the subgroup $\mathbb Z/n_i$ embedded diagonally by $x \mapsto (x, (n_j/n_i) x)$ we obtain a non-trivial element of $\Sym^2 ( H^1(\mathbb Z/n_i, \mathbb Q/\mathbb Z)).$

By pulling back to this cyclic subgroup, we may reduce to the case when $G$ is a cyclic group. Since $H^3(G,\QZ) \cong G$, it suffices to show a generator of $\Sym^2( H^1(G, \QZ))$ is sent to a generator of $H^3(G,\QZ)$ by this map. Since a generator of $H^1(G, \QZ)$ is sent by Bockstein to a generator of $H^2(G, \Z)$, and cupping with a generator of  $H^2(G, \Z)$ gives the periodicity isomorphism $H^i(G, \QZ) \to H^{i+2}(G, \QZ)$, this follows.  \end{proof}

For $M$ a $3$-manifold, the torsion linking pairing of classes $a,b \in H^1(M, \QZ)$ is defined by $ (a,b) \mapsto \int_M (a \cup B b)$ where $B \colon H^1(M, \QZ) \to H^2(M,\mathbb Z)$ is the Bockstein map.   
 If $H^1(M, \QZ)$ is finite, this is a nondegenerate symmetric pairing, 
 and regardless it becomes nondegenerate after quotienting by the divisible subgroup $D_M$ of $H^1(M, \QZ)$.
If $T_M$ is the torsion subgroup of $H_1(M, \Z),$ then note $T_M^\vee$ is naturally isomorphic to $H^1(M, \QZ)/D_M$, so we have a nondegenerate pairing
$T_M^\vee \tensor T_M^\vee \ra \QZ,$ which is equivalent to an isomorphism $T_M^\vee \cong T_M$, and, by taking the inverse of that isomorphism, is equivalent to a nondegenerate pairing $T_M \tensor T_M \ra \QZ.$  This latter pairing is what is usually called the torsion linking pairing.

\cref{localized-counting} in fact gives the complete distribution on the homology groups $H_1(\pi_1(M_{g,L}),\Z_p)$ 
(where $\Z_p$ is the $p$-adic integers), along with the torsion linking pairings on these homology groups.  In addition it gives these distributions simultaneously for any finite set of primes $p$, as we see below.

\begin{prop}\label{abelian-p-group-prob} 
Let $S$ be a finite set of primes and $\Z_S=\prod_{p\in S}\Z_S$.
Let $G$ be a finite abelian $S$-group and let $ \ell \colon G^\vee \times G^\vee \to \QZ$ be a nondegenerate symmetric pairing.
Then
\[
\lim_{g\ra\infty} \lim_{L\ra\infty} \Prob\left[ H_1(M_{g,L},\Z_S)\cong G, \textrm{torsion linking going to $\ell$}\right]= \frac{1}{ \abs{G} \abs{ \Aut(G, \ell)}} \prod_{p\in S}\prod_{j=1}^{\infty} \frac{1}{1+p^{-j} }.\] \end{prop}

Dunfield and Thurston \cite[\S8]{DunfieldThurston} found the limiting distribution of $H_1(M_{g,L},\Z/p)$, and \cref{abelian-p-group-prob} enriches this by extending to $\Z_p$ coefficients and tracking the torsion linking pairing.  
Dunfield and Thurston discussed the fact that the distribution on elementary abelian $p$-groups that they found from $H^1(M_{g,L},\Z/p)$ does not match the limiting distribution if one takes a quotient of a free group on $g$ generators by $g$ relations.  This latter model of a random group,  a ``random balanced presentation,''  was studied in \cite[\S3]{DunfieldThurston}, as well as by Friedman and Washington \cite{Friedman1989} and the second author \cite{Wood2015a} in connection to the Cohen-Lenstra heuristics,  and in the Cohen-Lenstra  philosophy \cite{Cohen1984} is the natural distribution on finite abelian groups (that have no additional structure).
We see from \cref{abelian-p-group-prob} that the groups $H_1(M_{g,L},\Z_p)$ are distributed as the pushforward of a natural distribution on abelian $p$-groups with symmetric pairings, where  a group with pairing appears with probability inversely proportional to $\abs{G} \abs{\Aut(G,\ell)}.$
This distribution was first introduced by Clancy, Leake, and Payne \cite[\S4]{Clancy2015a} in their study of Jacobians of random graphs, which are also groups with a natural symmetric pairing,  following the Cohen-Lenstra  philosophy that random groups should be considered with all of their additional structure.

\begin{proof}
The group $H_1(M_{g,L},\Z_S)$ is the maximal abelian $S$-group quotient of $\pi_1(M_{g,L})$.
Any surjection $\pi_1(M_{g,L}) \to G$ not  lifting to any abelian extension $1 \to \mathbb Z/p \to H \to G \to 1$ with $p\in S$
must be an isomorphism between $H_1(M_{g,L},\Z_S)$ and $P$. 
This isomorphism sends the linking pairing to $\ell$ if and only, if, for $a,b \in H^1( G, \QZ)$, we have $\ell(a,b) = \tau ( a \cup B b)$.
So we may apply \cref{localized-counting}, 
  summing over possible values of $\tau$,
with $V_i = \mathbb Z/p_i$, for $p_i\in S$, and $W_i$ the set of classes in $H^2(G,\mathbb Z/p_i)$ corresponding to abelian extensions, and no other $V_i$'s or $G_i$'s. Because $V_i$ is symmetric,  we have $w_{V_i}=0$ unless $W_i^\tau=0$ and $w_{V_i} =\prod_{j=1}^{\infty} (1+p_i^{-j} ) ^{-1}$ if $W_i^\tau=0$.

Let $N_\ell(G)$ be the number of $\tau\colon H^3( G, \QZ) \to \QZ $ such that $\ell(a,b) = \tau ( a \cup B b)$ for all $a,b,$ and, for all $i$ and all nonzero $\alpha  \in W_i$, there exists $\beta \in H^1(G, \mathbb Z/p_i)$ such that $ \tau(\alpha \cup \beta) \neq 0$. Then the limiting probability that $H_1(M_{g,L},\Z_S)$ is isomorphic to $G$ by an isomorphism sending $\ell$ to the torsion linking pairing is
\[ \frac{  \abs{H_2 (G, \mathbb Z) } }{ \abs{ \Aut(G, \ell) }} \frac{N_\ell (G) }{ \abs{H_3(G, \mathbb Z)}} \prod_{p\in S} \prod_{j=1}^{\infty} (1+p^{-j} ) ^{-1}.\]

Next we evaluate $N_\ell(G)$. First we note that $W_i$ is the image of $H^1(G, \QZ)$  under the Bockstein map $H^1(G,\QZ) \to H^2(G, \mathbb Z/p_i)$, as every abelian extension of $G$ by $\mathbb Z/p_i$ adds a $p_i$th root to a character in the dual group of $G$, and taking the image of that character under Bockstein gives the extension class.

We will show the second condition in $N_\ell(G)$ is superfluous, i.e. if $\tau (a \cup B b) = \ell(a,b)$ for all $a,b\in H^1(G,\QZ)$, then, for all nonzero $\alpha \in W_i$, there exists $\beta \in H^1(G, \mathbb Z/p_i)$ such that $ \tau(\alpha \cup \beta) \neq 0$. For $\alpha \in W_i\subset H^2(V_i,\Z/p_i)$,  we have $\alpha = B \gamma$ for some $\gamma \in H^1(G, \QZ)$,  by the previous paragraph.  
Since $\alpha$ is nonzero, $\gamma$ is not divisible by $p$.  For any $\beta \in H^1(G, \mathbb F_p) \subseteq H^1(G, \QZ)$ we have 
$\tau ( \alpha \cup \beta) = \tau ( B \gamma \cup \beta) = \tau( \beta \cup B\gamma) = \ell(\beta,\gamma)$. 
Because the linking pairing is nondegenerate and $\gamma$ is not divisible by $p$, we can choose a $p$-torsion $\beta$ making $\ell(\beta,\gamma)$ nonzero.

So $N_\ell(G)$ is simply the number of linear forms $\tau$ that restrict to $\ell$ on classes of the form $ a\cup B b$. By \cref{linking-injective}, the classes of the form $a\cup B b$ generate a submodule of $H^3(G, \QZ)$ isomorphic to $\Sym^2(G^\vee)$, and so $\ell$ extends to exactly $\frac{ \abs{ H^3(G, \QZ)}}{ \abs{\Sym^2(G^\vee)}} $ forms $\tau$. This gives the formula for the probability
\[ \frac{  \abs{H_2 (G, \mathbb Z) }}{  \abs{ \Aut(G, \ell) } \abs{ \Sym^2(G^\vee)} }\prod_{p\in S} \prod_{j=1}^{\infty} (1+p^{-j} ) ^{-1}.\]
It is well-known that $\abs{H_2 (G, \mathbb Z)}=\abs{\wedge^2(G^\vee)}$, and we can easily compute 
$\abs{ \Sym^2(G^\vee)} =\abs{\wedge^2(G^\vee)} \abs{G}$,  proving the proposition.

\end{proof}

\subsection{Proof of  \cref{intro-prob}}\label{SS:prove12}

We now prove  \cref{intro-prob}.  The main thing remaining to show is that there is no escape of mass in the limit of distributions, and the essential ingredient for that is \cref{P:robustness-exchange}.
Note that the support of the probability distribution of \cref{intro-prob} is equal to the closure of the set of profinite completions of fundamental groups of oriented $3$-manifolds. So there are no open subsets of $\Prof$ with zero measure that contain 3-manifold groups.

\begin{proof}[Proof of \cref{intro-prob}]

We first construct a probability measure $\mu$ on the space $\operatorname{OrProf}$ of (isomorphism classes of) oriented profinite groups in $\operatorname{Prof}$.  We use the Borel $\sigma$-algebra for the topology generated by the basic opens $U_{\C,\bG}=\{ \bK | \bK^\C\cong \bG\}$ indexed by $\C$ a finite set of finite groups and $\bG$ an oriented finite group.   We can define a pre-measure $\mu$ on the algebra $\mathcal{A}$ of sets generated by the basic opens  $U_{\C,\bG}$
by
$$
\mu(A):=\lim_{g\to \infty} \lim_{L \to \infty} \mu_{g, L} (A ).
$$
Note when $A=U_{\C,\bG}$, \cref{localized-counting} gives the limiting value, and $\mu$ is additive because finite sums commute with the limits. 
If we take $C_\ell$ to be the set of all groups of order at most $\ell$,  by \cref{P:robustness-exchange} in the special case $\bH=1$, we have that
\begin{equation}\label{E:levelelltoone}
\sum_{G \in I^{C_\ell}} \mu (U_{C_\ell,G} )= \lim_{g\to \infty} \lim_{L \to \infty}  \sum_{G \in I^{C_\ell}} \mu_{g,L} (U_{C_\ell,G} ) =1.
\end{equation}
Now, \cite[Proof of Theorem 9.1]{Liu2020}, using \eqref{E:levelelltoone} in place of \cite[Theorem 9.2]{Liu2020}, shows that $\mu$ is countably additive on $\mathcal{A}$.
Then, Carath\'eodory's extension theorem implies $\mu$ extends uniquely to a measure on the Borel $\sigma$-algebra.
Since any open set in our topology is a disjoint union of basic opens,  for any open $U$ by Fatou's lemma we have
$$
\mu(U) \leq \liminf_{g\to \infty} \liminf_{L \to \infty} \mu_{g, L} (U),
$$
which proves the weak convergence of $\mu_{g,L}$ to $\mu$.  We obtain the theorem by pushing forward the distribution to $\operatorname{Prof}$ and summing over $\tau$.
\end{proof}

\begin{cor}\label{C:dis-num-surj}
For every finite group $G$ and natural number $k$, the limit
$$
P_{G,k}:=\lim_{g\ra \infty} \lim_{L\ra \infty} \Prob[\pi_1(M_{g,L}) \textrm{ has exactly $k$ surjections to $G$}]
$$
exists, and for each $G$,  the $P_{G,k}$ give a probability distribution on the natural numbers $k$.
\end{cor}
\begin{proof}
Let $\C=\{ G\}$.  The number of surjections $\pi_1(M_{g,L})\ra G$ is the same as the number of surjections  $\pi_1(M_{g,L})^\C\ra G.$
Let $V_{\C,k}$ be the set of groups $K$ in $\Prof$ with $K^\C$ having exactly $k$ surjections to $G$.  Then $V_{\C,k}$ is the  union of $U_{\C,G_i}$ for some $G_i$, and the complement of the  union of $U_{\C,G'_i}$ for some other $G'_i$, and thus is open and closed.  Thus it follows from Theorem~\ref{intro-prob} that 
$$
\lim_{g\ra \infty} \lim_{L\ra \infty} \Prob[\pi_1(M_{g,L})\in V_{\C,k}] =\mu(V_{\C,k}),
$$
and the corollary follows.
\end{proof}

\subsection{The limiting probability of a $G$-cover with positive first Betti number is $0$}\label{S::VPFBN}

Dunfield and Thurston introduced their model of random Heegaard splittings in order to shed light on the Virtual Haken Conjecture, and the stronger
Virtual Positive Betti Number Conjecture (prior to Agol's Theorem
\cite{Agol2013}).
Dunfield and Thurston showed that for a fixed abelian group $Q$, the limit in $L$ of the probability that $M_{2,L}$ has a $Q$-cover with positive first Betti number is $0$,  and Rivin \cite[Theorem 11.5]{Rivin} generalizes this to solvable groups and fixed $g>1$.  
In the limit as $g\ra\infty$, the following result addresses this question for $G$-covers for any finite group $G$.

\begin{prop}\label{virtually-fibered} For all $n$, 
$$
\lim_{g\ra\infty}\lim_{L\ra\infty} \Prob[ M_{g,L} \textrm{ has a degree $\leq n$ cover with positive first Betti number} ]=0.
$$
%the limit of $g$ goes to infinity of the limit as $L$ goes to $\infty$ of the probability that $M_{g,L}$ has a degree $\leq n$ cover with positive first Betti number %vanishes. 
\end{prop}

\begin{proof} We will first prove, for each finite group $G$, 
$$
\lim_{\substack{p\ra\infty\\ p \textrm{ prime}} }
\lim_{g\ra\infty}\lim_{L\ra\infty} \Prob[ \pi(M_{g,L}) \textrm{ has surjection to $G$ with kernel that has quotient $\Z/p$} ]=0.
$$
%that the limit as the prime $p$ goes to $\infty$ of the limit of $g$ goes to $\infty$ of the limit as $L$ goes to $\infty$ of the probability that $\pi_1( M_{g,L})$ has %a surjection to $G$ whose kernel has a surjection to $\mathbb Z/p$ is $0$.
To do this, we bound the probability that $\pi_1(M_{g,L})$ has a surjection by the expected number of such surjections. Furthermore, we represent the expected number of such surjections as the expected number of surjections to $G$ minus the expected number of surjections to $G$ whose kernel does not have a surjection to $\mathbb Z/p$.

By \cite[Theorem 6.21]{DunfieldThurston}, the triple limit of the expected number of surjections to $G$ is $\frac{ \abs{H_2(G, \mathbb Z)} \abs{G} } { \abs{H_1(G, \mathbb Z)}}$.
The number of surjections from $\pi_1(M_{g,L}) $ to $G$ whose kernel does not have a surjection to $\mathbb Z/p$ is the sum over \Mdcorated/ groups $\mathbf G $ with underlying group $G$ of $L_{\mathbf G, \underline{V}, \underline{W}, \underline{N}}$ where $\underline{V}$ consists of all irreducible representations of $G$ mod $p$, $\underline{W}$ of all extension classes of these representations, and $\underline{N}$ is empty. Indeed, the kernel has a surjection to $\mathbb Z/p$ if and only if it has a $G$-equivariant surjection to some irreducible mod $p$ representation, which happens if and only if $\pi_1$ surjects onto some extension of $G$ by that representation.

Thus, by \cref{localized-counting}, the expected number of surjections to $G$ whose kernel does not have a surjection to $\mathbb Z/p$ is \[ \frac{ \abs{H_2(G, \mathbb Z)} \abs{G} } { \abs{H_1(G, \mathbb Z)}\abs{H_3(G,\mathbb Z)}} \sum_{ \tau : H^3(G, \QZ) \to \QZ } \prod_i w_{V_i}(\tau).\]
It suffices to prove $\prod_i w_{V_i}(\tau)$ converges to $1$ as $p$ goes to $\infty$, % independently of $\tau$, 
as then this sum will converge to  $\frac{ \abs{H_2(G, \mathbb Z)} \abs{G} } { \abs{H_1(G, \mathbb Z)}}$ and so the difference will converge to $0$, as desired. 

Because we are taking a limit as $p$ goes to $\infty$, we restrict attention to the case that $p$ does not divide $2|G|$. We then have $\dim H^1 ( G, V_i)=\dim H^2(G, V_i)=0$ for all representations $V_i$ of characteristic $p$. Hence the condition for $w_{V_i}(\tau)$ to be nonzero is automatically satisfied, so $w_{V_i}(\tau) = \prod_{k=1}^{\infty} (1- q_i^{-k})^{-1/2}$ if $V_i$ is not self-dual, $\prod_{j=1}^{\infty} ( 1+ q_i^{-j})^{-1}$ if $V_i$ is self-dual with $\epsilon_i=\pm 1$, or $\prod_{j=1}^{\infty} (1 + q_i^{ - j - \frac{1}{2} })^{-1}$ if $\epsilon_i$ is $0$.

All these factors converge to $1$ as $q_i$ goes to $\infty$. Since $q_i$ is at least the characteristic $p$ of $V_i$, they converge to $1$ as $p$ goes to $\infty$. The number of factors in $\prod_i w_{V_i}(\tau)$ is the number of isomorphism classes of irreducible representations of $G$ over $\mathbb F_p$, which is bounded by the number of isomorphism classes of irreducible representations of $G$ over $\mathbb C$ and thus bounded independently of $p$, so the product goes to $1$, as desired, and thus the  limiting probability of a surjection to $G$ whose kernel has a surjection to $\mathbb Z/p$ is $0$.

Since a group with a surjection to $\Z$ has a surjection to $\mathbb Z/p$ for all $p$, it follows that
$$
\lim_{g\ra\infty}\lim_{L\ra\infty} \Prob[ \pi(M_{g,L}) \textrm{ has surjection to $G$ with kernel that has quotient $\Z$}]=0.
$$
%
% the double limit of the probability that $\pi_1(M_{g,L})$ has a surjection to $G$ whose kernel has a surjection to $\mathbb Z$ is $0$. 
Summing over all $G$ of order $\leq n!$, 
$$
\lim_{g\ra\infty}\lim_{L\ra\infty} \Prob[ \pi(M_{g,L}) \textrm{ has a normal subgroup of index $\leq n!$ that has quotient $\Z$}]=0.
$$
%the double limit of the probability that  $\pi_1(M_{g,L})$ has a normal subgroup of index $\leq n!$ whose kernel has a surjection to $\mathbb Z$ is $0$. 
Because every subgroup $H$ of index $\leq n$ contains a normal subgroup $N$ of index $\leq n!$,  and if $H$ has a surjection to $\Z$ then so does $N$, we have
$$
\lim_{g\ra\infty}\lim_{L\ra\infty} \Prob[ \pi(M_{g,L}) \textrm{ has a subgroup of index $\leq n$ that has quotient $\Z$}]=0.
$$
%the double limit of the probability that $\pi_1(M_{g,L})$ has a subgroup of index $n$ whose kernel has a surjection to $\mathbb Z$ is $0$.
Finally, $M_{g,L}$ has a covering of degree $\leq n$ with positive first Betti number if and only if  $\pi_1(M_{g,L})$ has a subgroup of index $n$ whose kernel has a surjection to $\mathbb Z$.
\end{proof}

\section{Some Further Directions}

\subsection{Algorithms}
It may be possible to obtain from our results an algorithm which,  given finite groups $G_1,\dots, G_n, $ $H_1,\dots, H_m$, returns whether there exists a 3-manifold whose fundamental group admits $G_1,\dots, G_n$ as a quotient but not $H_1,\dots, H_m$. This happens if and only if there exists a finite level-$\C$ group, with $G_1,\dots, G_n$ as a quotient but not $H_1,\dots, H_m$, that satisfies the criteria of \cref{level-C-existence}, for $\mathcal C =\{G_1,\dots, G_n, H_1,\dots, H_m\}$. These criteria are straightforwardly computable for a given group. Thus the main difficulty is that there are infinitely many level-$\C$ group.

\subsection{Other random groups}
On the probabilistic side, it would be interesting to generalize this work to other models of random 3-manifolds (e.g. see \cite[Section 7.4]{Aschenbrenner2015},
\cite{Petri2020}). 
Do they produce the same probability measure?  If not, can our methods, or other new methods, be applied to find the new distribution? There are some models, such as the mapping torus of a random element of the mapping class group, that certainly give different distributions, as the fundamental groups of mapping tori always surject onto $\mathbb Z$ and thus onto $\mathbb Z/n$ for all $n$.

Are there other topological, geometric, or algebraic constructions of random groups that give the distribution found in this paper?
Liu \cite[Appendix A]{Liu2022} constructs a random pro-$\ell$ group by taking a quotient of the pro-$\ell$ completion of a surface group defined using a random automorphism of the group.  She proves these groups have the same limiting (non-oriented) moments as pro-$\ell$ completions of random Heegaard splittings, 
and we expect our methods will show this implies they have the same limiting distribution.

In particular, we wonder what are other ways to give a natural random oriented group, i.e. a group $G$ with a specified element of $H_3(G,\Z)$?
And what distributions arise from such groups?  In the abelian case, there is a universality theorem 
 \cite{Wood2017} that says many different constructions of random abelian groups with symmetric pairings have the same limiting distribution (which the the same as the limiting distribution of $H_1(M_{g,L})$ in this paper).  Is there a non-abelian version of this universality?

\subsection{Questions on the limiting measure $\mu$}
In addition to questions that have direct relevance to 3-manifolds, we can ask about properties of the limiting measure on $\operatorname{Prof}$ obtained from the fundamental groups of random 3-manifolds, or its support, that are not necessarily logically related to the same question for 3-manifolds.

A starting point is Dunfield and Thurston's question \cite[Section 9]{DunfieldThurston}, for a fixed finite group $G$, about the probability that a random Heegaard splitting has a $G$-cover with positive first Betti number.  \cref{virtually-fibered} shows these limiting probabilities are $0$, and so by countable additivity
a random group according to $\mu$ has an open subgroup with a surjection to $\hat{\mathbb Z}$ with probability zero.
Because this condition is neither a closed nor an open condition, 
Theorem~\ref{intro-prob} gives no logical implication between this question of $\mu$ and the limit of the analogous question for the distribution of $\pi_1(M_{g,L})$.
Indeed,  Agol's Theorem \cite{Agol2013} shows that most 3-manifolds have a subgroup of finite index with a surjection to $\mathbb Z$.

Thus, while random groups according to $\mu$ behave like 3-manifold groups in various ways, the analogy is not perfect. Some analogous questions have different answers on the two sides. It would be interesting to investigate how often this happens. 
In other words, to take properties of groups known or suspected to hold for all or almost all 3-manifold groups, detectable by the profinite completion, and ask with what probability they hold for $\mu$-random profinite groups. Owing to the great recent progress in the theory of 3-manifold groups, we will find more interesting questions of this form taking known properties rather than suspected ones.

{\bf Poincar\'e duality - } One question along these lines has to do with Poincar\'e duality. For $V$ a representation of $\pi_1(M)$, we have by Poincar\'e duality the cup product and fundamental class give a perfect pairing $H^i (M, V) \times H^{3-i}(M, V^\vee) \to \QZ$. If $M$ is irreducible and has infinite fundamental group, then $M$ is aspherical, so $H^i (M, V) = H^i(\pi_1(M),V)$ and thus we have a perfect pairing $H^i (\pi_1(M), V) \times H^{3-i}(\pi_1(M), V^\vee) \to \QZ$.   Because $\pi_1(M)$ is automatically a good group in the sense of Serre, we have $H^i(\pi_1(M),V)= H^i ( \hat{\pi}_1(M), V)$, so we have a perfect pairing $H^i (\hat{\pi}_1(M), V) \times H^{3-i}(\hat{\pi}_1(M), V^\vee) \to \QZ$.

Does something similar hold for $G$ a random group according to the measure $\mu$? In other words, for every representation $V$ of $G$, is $\tau(\alpha \cup \beta) \colon H^i(G, V) \times H^{3-i}(G,V^\vee) \to \QZ$ a perfect pairing? For $i>3$, this is just the statement that $H^i(G,V)$ should vanish.

{\bf Surface subgroups -} Most 3-manifolds are hyperbolic, and hyperbolic 3-manifolds are known to contain plentiful subgroups isomorphic to the fundamental group of a hyperbolic surface $\Sigma_g$. Because these hyperbolic 3-manifold groups are LERF, this produces an injective map on profinite completions $\hat{\pi}_1(\Sigma_g) \to \hat{\pi}_1 (M)$.

For $G$ a $\mu$-random group, with what probability do we have an injection  $\hat{\pi}_1(\Sigma_g) \to G$?

Something weaker, an injection $\pi_1(\Sigma_g) \to G$, exists with probability $1$: By \cite[Theorem 1.1]{BGSS}, it suffices to check that $G$ admits an injection from a free group with probability $1$. In fact, with probability $1$, two random elements of $G$ generate a free group, which one can check using the fact that $G$ has infinitely many distinct finite simple quotients with probability $1$.

{\bf Topological finite generation - } The fundamental groups of 3-manifolds are topologically generated, so their profinite completions are topologically finitely generated. Does the same hold for random groups according to the measure $\mu$? Since 3-manifolds of Heegard genus $g$ can require up to $g$ generators, and $\mu$ is obtained by a large $g$ limit, it is not clear what to expect. (This question was suggested by Mark Shusterman and Jordan Ellenberg.)

{\bf Groups with special structure - } In Section~\ref{SS:obstructed}, we found all finite groups in the support of the measure $\mu$. One could seek a stronger version of this result by replacing the condition ``finite" with a weaker condition such as virtually abelian, virtually nilpotent, or virtually solvable. All virtually solvable fundamental groups of 3-manifolds are known  \cite[Theorem 1.11.1]{Aschenbrenner2015}, but as in the finite case, there may be new examples that are limits of 3-manifold groups but not themselves 3-manifold groups. One could also ask similar questions for other restricted classes of groups.

{\bf Linear representations - } Hyperbolic 3-manifold groups are subgroups of $SL_2(\mathbb C)$, and thus have two-dimensional linear representations with image dense in $SL_2$. These can be defined over a number field $K$, so we obtain representations of the profinite fundamental group into $SL_2( K_v)$ for the completion $K_v$ of $K$ at each place $v$. We can ask whether a $\mu$-random group has such representations, and more generally what representations over $p$-adic fields it has.

\appendix

\section{Semicharacteristics}\label{ss-semicharacteristics}

Let $M$ be a manifold of dimension $2n+1$ and $\kappa$ a field. The semicharacteristic of $M$ with coefficients in $\kappa$ is $\sum_{i=0}^n (-1)^i \dim H_i(M, \kappa)$. 

Generalizing this, for $M$ a manifold with a surjection $\pi_1(M) \to G$ and associated covering $\tilde{M}\to M$ (equivalently, for $\tilde{M}$ a manifold with a free action of $G$), Lee \cite[Definition 2.3]{Lee} defined a semicharacteristic class $\sum_{i=0}^n (-1)^i [ H_i(\tilde{M},\kappa)]$ taken in the Grothendieck group of representations of $G$ over $\kappa$ modulo a certain subspace depending on the parity of $n$ and the characteristic of $\kappa$. The main result of \cite{Lee} is that, modulo this subspace, the semicharacteristic is a bordism invariant \cite[Theorem 2.7 and Theorem 3.8]{Lee}.

This result bears an obvious similarity to \cref{parity-is-bordism}. The difference is that \cite{Lee} considers the cohomology of the cover in $K$-theory, while we consider cohomology twisted by a representation. Cohomology twisted by a representation is the more powerful invariant: Recall that representations of a group are equivalent to modules of the group algebra and a module $P$ is projective if the functor $\Hom(P, -)$ is exact. We will check that the dimension of the $i$th cohomology twisted by each indecomposable projective module for the group algebra determines the class of the $i$th cohomology of the cover in $K$-theory. However, the $K$-theory does not determine the cohomology of non-projective modules.

Using this, we will show that our result implies the main result of \cite{Lee} in the cases where they both apply. It would be interesting to find a suitable generalization of our result (equivalently, strengthening of Lee's) to the higher-dimensional even characteristic case.

We begin by formally defining the subspace we quotient the Grothendieck group by.
For $\kappa$ of characteristic $\neq 2$, and $n$ odd, the semicharacteristic is valued in the quotient of the Grothendieck group by the subspace generated by all symmetrically self-dual
representations together with the regular representation, while for $\kappa$ of characteristic $\neq 2$, and $n$ even, the semicharacteristic is valued in the quotient of the Grothendieck group by the subspace generated by all symplectic representations together with the regular representation \cite[Definitions 2.1 and 2.3]{Lee}. For $\kappa$ of characteristic $2$, the semicharacteristic is valued in the quotient of the Grothendieck group by the subspace generated by all \emph{even representations} together with the regular representation \cite[Theorem 3.8]{Lee}, where the even representations  $V$ are those admitting a nondegenerate symmetric bilinear $G$-invariant form $\phi$ such that $\phi(x,tx)=0$ for all $x\in V$ and all $t\in G$ of order exactly $2$ \cite[p. 190]{Lee}.

We next recall that for $V$ an irreducible representation of a finite group $G$ over a field $\kappa$, there is a unique indecomposable projective module $\mathcal P(V)$ for $\kappa[G]$ that admits $V$ as a subrepresentation. This is also the unique indecomposable projective module for $\kappa[G]$ that admits $V$ as a quotient. 

\begin{lemma}\label{multiplicity-vs-projective} For $V$ a absolutely 
 irreducible representation of $G$ over $\kappa$, the number of times $V$ appears in the Jordan-H\"older decomposition of $H_i (\tilde{M}, \kappa)$ is equal to $\dim H^i (M, \mathcal P(V))$. \end{lemma}

\begin{proof} Since projective modules are stable under duality, projective modules are also injective. Because $\mathcal P(V)$ is injective, $H^i(M, \mathcal P(V))= \Hom_G ( H_i(\tilde{M},\kappa), \mathcal P(V))$. 

So it suffices to prove that for a representation $W$, the number of times $V$ appears in the Jordan-H\"older decomposition of $W$ is equal to $\dim \Hom_G( W, \mathcal P(V))$.

Again because $\mathcal P(V)$ is injective, both sides are additive in exact sequences, so we may reduce to the case when $W$ is irreducible, and the statement is that for $W,V$ irreducible, $\Hom_G(W, \mathcal P(V)) \cong \kappa$ if $W \cong V$ and $0$ if $W \not\cong V$, which is standard. 
\end{proof}

\begin{lemma}\label{projective-envelope-properties} For $V$ a absolutely  irreducible representation of $G$ over a field $\kappa$ of characteristic not two, $V$ is symmetrically self-dual
 if and only if $\mathcal P(V)$ is,  and $V$ is symplectic if and only if $\mathcal P(V)$ is.

For $V$ a absolutely irreducible representation of $G$ over a field $\kappa$ of characteristic two, $\mathcal P(V)$ is symplectic
 if and only if $V$ is self-dual but not an even representation. \end{lemma}

\begin{proof}Fix $V$ an absolutely irreducible representation of $G$.

Let $f$ be a homomorphism $\mathcal P(V) \to \kappa[G]$ of left $\kappa[G]$-modules. Then $f$ defines an embedding $V \to \mathcal P(V) \to \kappa[G]$. Any such embedding must have the form $x \in V \mapsto \sum_{g\in G}  a_f( g^{-1}\cdot  x) [g] $ for some linear form $a_f \in V^\vee$.  Furthermore, composition with $f$ defines a linear map $V \cong   \Hom (\kappa[G], V) \to \Hom( \mathcal P(V), V) \cong \Hom(V, V) = \kappa$, and thus a linear form $b_f$ on $V$.

We claim that $b_f = \lambda a_f$ for some $\lambda \in \kappa^\times$. To see this, note that both $f \mapsto a_f$ and $f \mapsto b_f$ are nontrivial homomorphisms $\Hom ( \mathcal P(V),\kappa[G]) \to V^\vee$ that are equivariant for the right $G$ action of $\kappa[G]$. As a right $G$-module, $\Hom ( \mathcal P(V),\kappa[G]) \cong \mathcal P(V)^\vee \cong \mathcal P(V^\vee)$ has a unique quotient isomorphic to $V^\vee$, so any two such homomorphisms differ by a $G$-invariant automorphism of $V^\vee$, i.e. by scalar multiplication.

Now fix one such $f$ that is a split injection. The pullback of any bilinear form on $\kappa[G]$ along $f$ gives a bilinear form on $\mathcal P(V)$. (All bilinear forms we consider will be $G$-invariant.) The pullback of a symmetric bilinear form is symmetric, and the pullback of a symplectic bilinear form is symplectic. Furthermore, because $f$ is split, every symmetric bilinear form on $\mathcal P(V)$ arises by pullback along $f$ from a symmetric bilinear form on $\kappa[G]$, and similarly with symplectic forms.  Thus, to test when $V$ is symmetrically self-dual,  we will calculate all symmetric bilinear forms on $\kappa[G]$ and check when the pullback of one along $f$ is symmetric, and similarly in the symplectic case.

We first describe the bilinear forms on $\kappa[G]$. These are parameterized by tuples $\mathbf d= (d_g)_{g\in G}$ of coefficients in $\kappa$ associated to $g\in G$, and are given by the formula
\[ \langle \sum_{g\in G} a_g[g], \sum_{g\in G} b_{g} [g] \rangle_{\mathbf d} = \sum_{g \in G} \sum_{h\in G} a_g b_h  d_{ h^{-1} g } .\]

Then  \[ \langle \sum_{g\in G} a_g[g], \sum_{g\in G} b_{g} [g] \rangle_{\mathbf d} =  \langle \sum_{g\in G} b_{g} [g] , \sum_{g\in G} a_g[g]\rangle_{\overline{\mathbf d}}\] where \[ \overline{d}_g = d_{g^{-1}} \] so $\langle, \rangle_{\mathbf d}$ is symmetric if $d_g =d_{g^{-1}}$ for all $g$ and symplectic if $d_g =- d_{g^{-1}}$ for all $g$ with, in characteristic $2$, the additional condition $d_{e} =0$ where $e$ is the identity.

Now a bilinear form $\mathcal P(V) \times \mathcal P(V) \to \kappa$ is nondegenerate if and only if the induced map $\mathcal P(V) \to \mathcal P(V)^\vee$ is injective, i.e. if its kernel is zero, which  happens if and only if the induced map $V \to \mathcal P(V) \to \mathcal P(V)^\vee$ is injective. Thus, the pullback of $\langle, \rangle_{\mathbf d}$ is nondegenerate if and only if there exists $x \in V$ and $y\in \mathcal P(V)$ such that $\langle f(x), f(y) \rangle_{\mathbf d} \neq 0$.

Now the map $L_f \colon \kappa[G] \to V^\vee $ that sends $\alpha \in \kappa[G]$ to the linear form $x \mapsto \langle f(x), \alpha \rangle_{\mathbf d}$ is $G$-equivariant since $f$ and $\langle,\rangle_{\mathbf d}$ are. Thus it defines an element of $V^\vee$. We calculate this element by evaluating at $\alpha=[e] \in \kappa[G]$.

For $x\in V$, $ f(x) = \sum_{g\in G}  a_f( g^{-1} x) [g] $ by definition.  Thus 
\[  L_f([e]) (x) = \langle f(x), [e] \rangle_{\mathbf d}  = \sum_{g\in G}  a_f (g^{-1}\cdot x)   d_{ g }= a_f \Bigl( \sum_g d_g g^{-1} \cdot x \Bigr) .\]

The pullback of the bilinear form $\langle,\rangle_{\mathbf d}$ is nondegenerate if and only if the composition of $L_f$ with $f \colon \mathcal P(V) \to \kappa[G]$ is nonzero.

If $V$ is not self-dual, then the homomorphism $\mathcal P(V) \to \kappa[G] \to V^\vee$ automatically vanishes and thus there is no such nondegenerate bilinear form. So suppose that $V$ is self-dual, so in particular there is a map $\gamma : V^{\vee} \to V$.  The isomorphism $\Hom (\kappa[G], V) \cong V$ sends a linear map $L$ to $L([e])$ so it sends $\gamma \circ  L_f$ to 
\[ \gamma (L_f([e]) ) = \gamma ( x \mapsto a_f ( \sum_g d_g g^{-1}\cdot x ) )= \sum_{g\in G} d_g g\cdot  \gamma(a_f) \in V.\]

So by the definition of $b_f$, the composition  $\gamma \circ L_f \circ f \colon \mathcal P(V) \to \kappa[G] \to V^\vee\to V$ is nonzero if and only if 
\[ b_f \Bigl( \sum_{g \in G} d_g  g\cdot \gamma(a_f) \Bigr) \neq 0 \]
and a nondegenerate symmetric (or symplectic) bilinear form exists if and only if $b_f ( \sum_{g \in G} d_g  g\cdot \gamma(a_f) ) \neq 0$ for some $\mathbf d$ satisfying the conditions to be symmetric (or symplectic).
To simplify this, note that $b_f (  x) = \langle \gamma(b_f), x \rangle_V$ for $\langle, \rangle_V$ the bilinear form on $V$, and recall $b_f = \lambda a_f$ so we can express the nonvanishing condition more simply as
\[  \sum_{g\in G  }d_g \langle \gamma(a_f), g\cdot \gamma(a_f) \rangle_V \neq 0 .\]

We now specialize to particular cases. In characteristic not two, 
\[ \sum_{g\in G  }d_g \langle \gamma(a_f), g\cdot \gamma(a_f) \rangle_V = \sum_{g\in G  }d_g \langle g^{-1} \cdot \gamma(a_f),  \gamma(a_f) \rangle_V =   \sum_{g\in G  }d_g \langle  \gamma(a_f),  g^{-1} \cdot \gamma(a_f) \rangle_V   \cdot \begin{cases} 1 & V \textrm{ symmetric} \\ -1 & V \textrm{ symplectic} \end{cases} \] \[= \sum_{g\in G  }d_{g}  \langle  \gamma(a_f),  g \cdot \gamma(a_f) \rangle_V   \cdot \begin{cases} 1 & V \textrm{ symmetric} \\ -1 & V \textrm{ symplectic} \end{cases}  \cdot   \begin{cases} 1 & \langle, \rangle_{\mathbf d} \textrm{ symmetric} \\ -1 & \langle, \rangle_{\mathbf d}  \textrm{ symplectic} \end{cases}.\] 

If $V$ is symmetric and $\langle, \rangle_{\mathbf d}$ is symplectic, or vice versa, then the signs don't match and so $\sum_{g\in G  }d_g \langle \gamma(a_f), g\cdot \gamma(a_f) \rangle_V $ is equal to its own negation and thus vanishes.  Since a self-dual absolutely irreducible representation is either symmetric or symplectic, we see there is no nondegenerate symmetric form on $\mathcal P(V)$ unless $V$ is symmetric and no nondegenerate symplectic form on $\mathcal P(V)$ unless $V$ is symplectic. Conversely, for any nonzero $\gamma(a_f)$, there is always some $h$ such that $ \langle \gamma(a_f), h\cdot \gamma(a_f) \rangle_V \neq 0$ by irreducibility of $V$. In this case, we can take $d_h = 1$, and $d_{h^{-1}}=1$ if $V$ is symmetric or $d_{h^{-1}}=-1$ if $V$ is symplectic, and $d_g=0$ for $g \neq h, h^{-1}$, and this ensures $\sum_{g\in G  }d_g \langle \gamma(a_f), g\cdot \gamma(a_f) \rangle_V\neq 0 $ .

In characteristic $2$, the unique bilinear form on $V$ is necessarily symmetric. Thus if $\langle, \rangle_{\mathbf d}$ is symplectic, the contributions of $g$ and $g^{-1}$ to the sum  $\sum_{g\in G  }d_g \langle \gamma(a_f), g\cdot \gamma(a_f) \rangle_V$ are equal. Thus these contributions cancel each other unless $g = g^{-1}$, i.e. if $g$ has order dividing $2$. Since $d_{e}=0$, we need only consider the contribution from $g$ of exact order $2$. If $V$ is even, then the contribution vanishes by definition and so there are no nondegenerate symplectic forms. Conversely, if $V$ is not even then for some $g$ and $x$ we have $\langle x, g\cdot x \rangle_V$ is nonzero. Since $g$ has order $2$, $\langle x, g\cdot x \rangle_V$ defines a Frobenius-semilinear form, so it vanishes for all $x$ outside a proper subspace. Choose $h$ such that $h\cdot \gamma(a_f)$ is not in that subspace, and observe that $\langle \gamma(a_f), h^{-1}gh \cdot \gamma(a_f)\rangle_V = \langle h \gamma(a_f), gh\cdot \gamma(a_f)\rangle_V \neq 0$, so choosing $\mathbf d$ supported on $h^{-1} gh$ we construct a nondegenerate symplectic form.
\end{proof}

We are now ready to describe how the semicharacteristic studied by Lee is determined by the cohomology groups controlled in \cref{parity-is-bordism}, and thus to deduce Lee's theorem (except in the even characteristic $n>1$ case) from \cref{parity-is-bordism}.

\begin{lemma}\label{semicharacteristic-vs-cohomology} Let $G$ be a finite group and $\kappa$ a finite splitting field for $G$. 
Let $n$ be a natural number. We will always take $M$ to be a $2n+1$-dimensional oriented manifold with a homomorphism $\pi_1(M)\to G$.

\begin{enumerate}

\item If $n$ is odd and $\kappa$ has characteristic $\neq 2$, the class of $\sum_{i=0}^n (-1)^i [ H_i(M,\kappa)]$ in the Grothendieck group of representations of $G$ over $\kappa$, modulo the classes of symmetrically self-dual representations, is determined by $\sum_{i=0}^n (-1)^i \dim H^i(M, V) \mod 2$ for even-dimensional symplectic representations $V$ of $G$ over $\kappa$ that are projective. In particular, it is an invariant of the class of $M$ in the oriented bordism group of $BG$.

\item  If $n$ is even and $\kappa$ has characteristic $\neq 2$, the class of $\sum_{i=0}^n (-1)^i [ H_i(M,\kappa)]$ in the Grothendieck group of representations of $G$ over $\kappa$, modulo the classes of symplectic representations and the regular representation, is determined by $\sum_{i=0}^n (-1)^i \dim H^i(M, V) \mod 2$ for even-dimensional symmetrically self-dual representations $V$ of $G$ over $\kappa$ that are projective. In particular, it is an invariant of the class of $M$ in the oriented bordism group of $BG$.

\item If $n$ is odd and $\kappa$ has characteristic $\neq 2$, the class of $\sum_{i=0}^n (-1)^i [ H_i(M,\kappa)]$ in the Grothendieck group of representations of $G$ over $\kappa$, modulo the classes of even representations, is determined by $\sum_{i=0}^n (-1)^i \dim H^i(M, V) \mod 2$ for even-dimensional symplectic  representations $V$ of $G$ over $\kappa$ that are projective and lift to $\ASp_\kappa(V)$.  In particular, if $n=1$ then it is an invariant of the class of $M$ in the oriented bordism group of $BG$.

\end{enumerate}

\end{lemma}

In the first and third cases, it is not necessary to mod out by the regular representation as the regular representation is symmetrically self-dual and, in characteristic 2, even.

\begin{proof} We handle part (1) first, then describe how the arguments in the remaining cases differ.

A class in the Grothendieck group can be represented as $\sum_V m_V[V]$, the sum taken over irreducible representations $V$ of $G$ over $\kappa$, for some integers $m_V$. 

Note that $V \oplus V^\vee$ is always symmetrically self-dual. Thus, two classes arising from two tuples of integers $m_V, m_V'$ are equivalent modulo the symmetrically self-dual representations if $m_V - m_{V^\vee} = m_{V}' - m_{V^\vee}' $ for all irreducible representations $V$ and $m_V - m_{V'}$ is even for all $V$ self-dual but not symmetrically self-dual. Indeed, in this case, the difference between the classes is a sum of irreducible symmetrically self-dual representations, sums of a non-self-dual representation and its dual, and even multiples of a symplectic representation, which are sums of a representation and its dual (itself).

When representing the class of $\sum_{i=0}^n (-1)^i [ H_i(M,\kappa)]$ in the Grothendieck group this way, $m_V(M) =  \sum_{i=0}^n (-1)^i \operatorname{mult}_V H_i(M,\kappa) $ . So to show this class, modulo symmetrically self-dual representations is determined by $\sum_{i=0}^n (-1)^i \dim H^i(M, V) \mod 2$ for even-dimensional symplectic representations of $G$ over $\kappa$ that are projective, it suffices to show that $m_V(M) - m_{V^\vee}(M)$ is determined, as is $m_V$ mod $2$ for irreducible representations that are self-dual but not symmetrically self-dual.

For the first part, the fact that $m_V (M)= m_{V^\vee} (M)$ was already proven by Lee, using Euler characteristic and Poincar\'e duality arguments. For the second part, if $V$ is irreducible and self-dual but not symmetric, then it must be symplectic so by \cref{projective-envelope-properties}, $\mathcal P(V)$ is symplectic Because $\mathcal P(V)$ is symplectic, it is even-dimensional. Thus by \cref{multiplicity-vs-projective}, \[m_V(M) = \sum_{i=0}^n (-1)^i \operatorname{mult}_V H_i(M,\kappa) = \sum_{i=0}^n (-1)^i H^i(M ,\mathcal P(V)) \mod 2 \] so $m_V(M)\mod 2$ is determined by $\sum_{i=0}^n (-1)^i H^i(M ,\mathcal P(V) )\mod 2$, and $\mathcal P(V)$ satisfies all of the assumed properties.

Finally, by \cref{parity-is-bordism}, $\sum_{i=0}^n (-1)^i H^i(M ,\mathcal P(V))  \mod 2$ is determined by the bordism class of $M$.

For part (2), the argument is similar, except for the following: First, we use the fact that $V \oplus V^\vee$ is always symplectic. Second, we prove that $\mathcal P(V)$ is symmetrically self-dual, and so we can no longer use the fact that $\mathcal P(V)$ is symplectic to guarantee it is even-dimensional. Instead we use the fact that we need only determine $\sum_{i=0}^n (-1)^i [ H_i(M,\kappa)]$ in the Grothendieck group modulo both the symmetrically self-dual representations and the regular representation.

Adding a copy of the regular representation does not affect $m_V - m_{V^\vee}$, but it swaps the parity of $m_V$ if $V$ is self-dual of odd multiplicity in $\kappa[G]$. Since the multiplicity of an irreducible representation $V$ in $\kappa[G]$ is equal to $\dim \mathcal P(V)$, adding a copy of the regular representation swaps the parity of $m_V$ for all irreducible representations $V$ with $\dim \mathcal P(V)$ odd. Thus, to determine the class modulo symplectic representations and the regular representation, it suffices to know $m_V -m_{V^\vee}$ for all irreducible representations $V$, $m_V\mod 2$ for all symmetrically self-dual irreducible representations $V$ with $\dim \mathcal P(V)$ even, and $m_V +m_W \mod 2$ for all pairs $V,W$ of symmetrically self-dual irreducible representations with $\dim \mathcal P(V), \dim \mathcal P(W)$ odd.

Thus, in the second part, it suffices to know $\sum_{i=0}^n H^i (M, \mathcal P(V) )\mod 2$ where $\mathcal P(V)$ is projective, symmetrically self-dual, and even-dimensional, and in the third part, it suffices to know $\sum_{i=0}^n H^i(M, \mathcal P(V) \oplus \mathcal P(W))\mod 2$ where $\mathcal P(V) \oplus \mathcal P(W)$ is projective, symmetrically self-dual, and even-dimensional. So we still need consider only representations that satisfy all the assumed properties. Then we use \cref{parity-is-bordism} the same way.

For part (3) it is again similar to part (1). We now use the fact that $V \oplus V^\vee$ is even, which may be less familiar -- the form $\langle (x_1,y_1) , (x_2,y_2) \rangle = x_1 \cdot y_2 + x_2 \cdot y_1$ is symmetric, and for $g$ of order $2$,\[  \langle (x,y) , g\cdot  (x,y) \rangle = \langle (x,y) , (g x,g y) \rangle = x \cdot g y + g x \cdot y = x \cdot g y + x \cdot g^{-1} y = x \cdot gy + x \cdot gy = 0 \] where we use $g= g^{-1}$ and the fact that the characteristic is two, so this form is even. 

We can again use the argument that symplectic representations must be even-dimensional, but we now face the difficulty that \cref{projective-envelope-properties} ensures that $\mathcal P(V)$ is symplectic  but we want the action of $G$ to lift to $\ASp_\kappa(V)$.  However, the obstruction to such a lift is contained in $H^2(G, \mathcal P(V))$ which vanishes since $\mathcal P(V)$ is projective, so a lift always exists. Finally, here \cref{parity-is-bordism} is restricted to the $n=1$ case only.
\end{proof}

By combining \cref{odd-characteristic-parity} and \cref{semicharacteristic-vs-cohomology}, we can check that the semicharacteristic vanishes in the odd characteristic $n=1$ case. Again, this requires only the projective case of \cref{odd-characteristic-parity}, and the general case may be significantly stronger.

\bibliographystyle{alpha}

\bibliography{../references.bib,MyLibrary3man.bib}

\end{document}